\documentclass[12pt]{amsart}
\usepackage{amssymb}
\usepackage{amsfonts}

\theoremstyle{plain}

\newtheorem{algorithm}{Algorithm}[section]
\newtheorem{thm}{Thm}

\newtheorem{corollary}[algorithm]{Corollary}

\newtheorem{lemma}[algorithm]{Lemma}

\newtheorem{theorem} [algorithm] {Theorem}
\newtheorem{theoremlet}[thm]{Theorem}

\newtheorem*{remarknonum}{Remark}
\newtheorem*{definitionnonum}{Definition}
\newtheorem*{question1}{Question 1}
\newtheorem*{question2}{Question 2}

\newtheorem{keylemma}[algorithm]{Key Lemma}

\newtheorem{proposition}[algorithm]{Proposition}
\newtheorem{remark}[algorithm]{Remark}

\numberwithin{equation}{algorithm}

\input{tcilatex.tex}
\input{tcilatex}
\begin{document}
\title{How to Lift Positive Ricci Curvature}
\author{Catherine Searle}
\address{Department of Mathematics\\
Oregon State University\\
368 Kidder Hall \\
Corvallis, OR 97331}
\email{searleca@math.oregonstate.edu}
\urladdr{https://sites.google.com/site/catherinesearle1/home}
\author{Frederick Wilhelm}
\address{Department of Mathematics\\
University of California\\
Riverside, CA 92521}
\email{fred@math.ucr.edu}
\urladdr{http://mathdept.ucr.edu/faculty/wilhelm.html}
\thanks{The first author was supported in part by CONACyT Project
\#SEP--106923. She is also grateful to the Mathematics Department of the
University of California at Riverside for its hospitality during two visits
where a large portion of this research was conducted.}

\begin{abstract}
We show how to lift positive Ricci and almost non-negative curvatures from
an orbit space $M/G$ to the corresponding $G$--manifold, $M.$ We apply the
results to get new examples of Riemannian manfiolds that satisfy both
curvature conditions simultaneously.
\end{abstract}

\maketitle

\section*{Introduction}

Lawson and Yau showed that $M$ admits positive scalar curvature provided $M$
is a compact $G$--manifold, with $G$ a compact, non-abelian, connected Lie
group \cite{LawYau}. By Myers' Theorem, this result can not be generalized
to positive Ricci curvature; however, one might ask about the case when the
fundamental group of $M$ is finite. Towards this end we have the following
result.

\begin{theoremlet}
\label{main}Let $G$ be a compact, connected Lie group acting isometrically
and effectively on a compact Riemannian manifold $M.$ Suppose the
fundamental group of a principal orbit is finite and the orbital distance
metric on $M/G$ has Ricci curvature $\geq 1$. Then $M$ admits a $G$%
--invariant metric with positive Ricci curvature.
\end{theoremlet}

\begin{remarknonum}
Various definitions of lower Ricci curvature bounds on metric spaces are
proposed in \cite{KuSh}, \cite{LottVill}, \cite{Oht}, \cite{Stu1}, \cite%
{Stu2}, and \cite{ZhZh}. Our proof only requires that the quotient space of
the principal orbits, $M^{\text{reg}}/G,$ has Ricci curvature $\geq 1,$ and
since $M^{\text{reg}}/G$ is a Riemannian manifold, it does not matter which
definition we choose.
\end{remarknonum}

The analogous result for positive sectional curvature is false. Let $%
SO\left( 3\right) $ act transitively on the second factor of $\mathbb{R}%
P^{2}\times \mathbb{R}P^{2}.$ By Synge's Theorem, the positively curved
metric on the quotient, $\mathbb{R}P^{2},$ cannot be lifted to a positively
curved metric on $\mathbb{R}P^{2}\times \mathbb{R}P^{2}$. Similarly, the
examples of Grove-Verdiani-Wilking-Ziller in \cite{GVWZ} and He in \cite{He}%
, show that the analog of Theorem \ref{main} is also false for non-negative
curvature.

On the other hand, we can lift almost non-negative curvature, and we do not
even need the hypothesis on the fundamental group of the principal orbits.

\begin{theoremlet}
\label{alm nonneg thm}Let $G$ be a compact, connected Lie group acting
smoothly and effectively on a compact smooth $n$--manifold $M.$ Let $\left\{
g_{\alpha }\right\} _{a=1}^{\infty }$ be a sequence of Riemannian metrics on 
$M$ for which the $G$--action is isometric.

Suppose $\left\{ \left( M/G,\mathrm{dist}_{\alpha }\right) \right\}
_{a=1}^{\infty }$ has almost non-negative curvature, where each $\mathrm{dist%
}_{\alpha }$ is the induced orbital distance metric. Then $M$ admits a $G$%
--invariant family of metrics with almost non-negative sectional curvature.
\end{theoremlet}

As both $M$ and $M/G$ are Alexandrov spaces, the following definition of
almost non-negative curvature is valid for both spaces.

\begin{definitionnonum}
We say that a sequence of Alexandrov spaces $\left\{ \left( X,\mathrm{dist}%
_{\alpha }\right) \right\} _{\alpha =1}^{\infty }$ is almost non-negatively
curved if and only if there is a $D>0$ so that 
\begin{eqnarray*}
\mathrm{Diam}\left( X,\mathrm{dist}_{\alpha }\right) &\leq &D, \\
\mathrm{curv}\left( X,\mathrm{dist}_{\alpha }\right) &\geq &-\frac{1}{\alpha 
}.
\end{eqnarray*}
\end{definitionnonum}

Together Theorems \ref{main} and \ref{alm nonneg thm} are more interesting
than either result is separately since their proofs yield the following.

\begin{theoremlet}
\label{supplement}If $\left\{ \left( M,G,g_{\alpha }\right) \right\}
_{\alpha =1}^{\infty }$ satisfies the hypotheses of both Theorems \ref{main}
and \ref{alm nonneg thm}, then $M$ admits a family of metrics that
simultaneously has positive Ricci curvature and is almost non-negatively
curved.
\end{theoremlet}

We believe that Theorems \ref{main}, \ref{alm nonneg thm}, and \ref%
{supplement} will ultimately lead to many new examples with positive Ricci
and almost non-negative sectional curvatures. To apply these theorems, one
needs an orbit space with positive Ricci and/or almost non-negative
sectional curvature. Unfortunately, there does not seem to be an extensive
catalog of such orbit spaces, leading us to ask the following two questions.

\begin{question1}
\label{pos ricci quest}Let $\mathcal{M}$ be the class of compact smooth,
manifolds $M$ admitting a smooth, effective action by a compact, connected
Lie group, $G$, with $\pi _{1}\left( \mathrm{principal}\text{ }\mathrm{orbit}%
\right) $ finite. Which $M\in \mathcal{M}$ admit a $G$--invariant metric
with $Ric\left( M^{\text{reg}}/G\right) \geq 1?$
\end{question1}

\begin{question2}
\label{alm nonneg quest}Let $\mathcal{M}$ be the class of compact smooth,
manifolds $M$ admitting a smooth, effective action by a compact, connected
Lie group, $G$. Which $M\in \mathcal{M}$ admit a family of $G$--invariant
metrics $\left\{ g_{\alpha }\right\} _{a=1}^{\infty }$ for which $\left\{
\left( M/G,\mathrm{dist}_{\alpha }\right) \right\} _{a=1}^{\infty }$ is
almost non-negatively curved$?$
\end{question2}

Let $G$ be a Lie group with a bi-invariant metric. Let $H\subset G\times G$
act on $G$ from the left and right, and not freely. The bi-quotient $G//H$
is non-negatively curved, and it seems likely that the technique of \cite%
{SchTu1} could yield that the Ricci curvature of $G^{\text{reg}}//H$ is also
positive, if $\pi _{1}\left( G//H\right) <\infty .$ One could then search
for smooth $H_{1}$--manifolds with $M/H_{1}=G//H,$ and find a solution to
Question 2 and possibly a solution to Question 1.

We have yet to pursue this line of inquiry, but we have proven the following
theorem.

\begin{theoremlet}
\label{examples}Let $\mathcal{Y}$ be the class of compact, smooth, manifolds
consisting of\vspace*{0.1in}

\noindent $\Sigma ^{7}\equiv $ $\left\{ \text{all exotic 7--spheres}\right\} 
$,\vspace*{0.1in}

\noindent $\Sigma _{BP}^{15}\equiv \left\{ \text{all exotic 15--spheres that
bound parallelizable manifolds}\right\} $\vspace*{0.1in},

\noindent \textbf{F$\mathbb{H}$P2 }$\equiv \left\{ \text{all double mapping
cylinders on }S^{3}\text{--bundles over }S^{4}\right. $\newline
\textbf{\ } \textbf{\hspace*{3.1in}}whose total spaces are homeomorphic to $%
S^{7}\}\vspace*{0.1in},$

\noindent \textbf{F$\mathbb{O}$P2 }$\equiv \{$all double mapping cylinders
on $S^{7}$--bundles over $S^{8}$\newline
\textbf{\hspace*{3.1in}}whose total spaces are homeomorphic to $S^{15}\}%
\vspace*{0.1in}.$

Any $M\in \mathcal{Y}$ admits a family of metrics $\left\{ g_{a}\right\}
_{\alpha =1}^{\infty }$ that is simultaneously almost non-negatively curved
and has positive Ricci curvature.
\end{theoremlet}

For the purpose of Theorem \ref{examples}, the double mapping cylinder on a
map $p:E\longrightarrow B$ is obtained from the disjoint union 
\begin{equation*}
B_{-}\amalg E\times \left[ -1,1\right] \amalg B_{+}
\end{equation*}%
of $E\times \left[ -1,1\right] $ and two copies of $B$, denoted $B_{-}$ and $%
B_{+}$, by making only the following identifications. For each $e$ in $E$,
let 
\begin{eqnarray*}
\left( e,-1\right) &\sim &p\left( e\right) \in B_{-}, \\
\left( e,1\right) &\sim &p\left( e\right) \in B_{+}.
\end{eqnarray*}

Next we give a brief history of prior work related to these results. It is
not meant to be comprehensive, rather we limit our attention to results that
are specifically relevant to Theorems \ref{main}, \ref{alm nonneg thm}, \ref%
{supplement}, and \ref{examples}.

Combining work of Nash \cite{N} or Poor \cite{Poor} and Fukaya-Yamaguchi 
\cite{FukYam}, gives us a family of almost non-negatively curved metrics
that are also Ricci positive on the exotic spheres that are bundles.

Wraith showed that all exotic spheres bounding parallelizable manifolds
admit metrics with $Ric>0$ \cite{Wr1}. In \cite{BoyGalNak}, Boyer, Galicki,
and Nakamaye showed that such exotic spheres admit $Ric>0$ metrics that are
also Sasaki, provided the dimension is odd.

By combining results of Grove-Ziller \cite{GrovZil1} and Guijarro \cite{Guij}
one gets non-negatively curved metrics on the class \textbf{F}$\mathbb{H}$%
\textbf{P2}. These metrics might make a good starting point for an
alternative argument that the class \textbf{F}$\mathbb{H}$\textbf{P2} has an
almost non-negatively curved family with positive Ricci curvature.

The exotic spheres in $\Sigma ^{7}$ and \textbf{$\Sigma _{BP}^{15}$} that
are not bundles and the entire class \textbf{F}$\mathbb{O}$\textbf{P2} were
not previously known to admit almost non-negative curvature. To the best of
our knowledge, the class \textbf{F}$\mathbb{O}$\textbf{P2} was not
previously known to admit positive Ricci curvature. \label{history}

Theorems \ref{main} and \ref{alm nonneg thm} were already known in the case
when the action is free. Theorem \ref{main} was established for free actions
by Nash in \cite{N} (cf. also \cite{Ber}, \cite{BB}, \cite{GPT}, \cite{Poor}%
). Theorem \ref{alm nonneg thm} was proven by Wei for free actions, with the
additional assumption that the base is non-negatively curved \cite{Wei}. For
actions with only principal orbits, Theorem \ref{alm nonneg thm} follows
from work of Fukaya-Yamaguchi (see Theorem 0.18 in \cite{FukYam}).

Various examples of $G$--manifolds with positive Ricci curvature and
isolated singular orbits are given by Bechtluft-Sachs and Wraith in \cite%
{BS-W1} and by Wraith in \cite{Wr3}.

When $\dim M/G=1,$ and $\pi _{1}\left( M\right) $ is finite, Grove and
Ziller showed that $M$ admits a $G$--invariant metric with positive Ricci
curvature \cite{GrovZil}, and Schwachh\"{o}fer and Tuschmann showed that any
cohomogeneity one manifold admits a $G$--invariant metric with almost
non-negative curvature, regardless of the fundamental group \cite{SchTu2}.
The hypothesis $Ric\left( M/G\right) \geq 1$ in Theorem \ref{main}, implies
that $\dim M/G\geq 2,$ so Theorem \ref{main} does not generalize the result
of \cite{GrovZil}, but Theorem \ref{alm nonneg thm} does extend the result
of \cite{SchTu2}.

To prove our theorems we employ two different methods to improve the metric:
Cheeger deformations and conformal changes. The same methods were combined
in \cite{Bet} to show that $S^{2}\times S^{2}$ admits positive bi-orthogonal
curvature.

In our context, Cheeger deforming a $G$--invariant metric on $M$ will
produce a metric with the desired curvature on any compact subset of the
regular part, $M^{\text{reg}},$ of $M.$ Rather than explicitly elucidating
the aforementioned principle, we have organized the paper to make the proofs
of the main results as clear as possible. Nevertheless, it is omnipresent
and manifested in Proposition \ref{orbital estimate}, Theorems \ref{Step 1}, %
\ref{generic cheeger}, and \ref{step 3}, and Corollary \ref{Step 1} below.

We obtain the desired metric in a neighborhood of the singular strata by
performing the correct $G$--invariant conformal change. There are two key
analytic ideas that make our conformal change work.

The first is based on the universal fact, established in \cite{PetWilh2},
that the Hessian of the distance from any compact Riemannian submanifold $S$
has a prescribed asymptotic behavior at nearby points. It is stated formally
in Lemma \ref{Peter's lemma} below. Informally, let $\Omega $ be a tubular
neighborhood of $S$ on which the closest point map $Pr:\Omega
\longrightarrow S$ is defined. If $\Omega $ is small enough we get an
estimate for the Hessian of $\mathrm{dist}\left( S,\cdot \right) $ on $%
\Omega $ by exploiting the fact that the intrinsic metrics on the fibers of $%
Pr$ are asymptotically Euclidean. This generalizes the known asymptotic
estimate for the Hessian of the distance function from a point, which, in
turn, is based on the fact that a neighborhood of a point in a Riemannian
manifold is asymptotically Euclidean.

The second analytic idea is to perform a conformal change of the metric with
a function of the form $e^{2\rho \left( \mathrm{dist}\left( S,\cdot \right)
\right) },$ where $\rho :\left( 0,\infty \right) \longrightarrow \mathbb{R}$
is $C^{1}$--close to $0,$ but $\rho ^{\prime \prime }\left( t\right) <<-1$
for $t$ very close to $0.$ Our estimates for the Hessian of $\mathrm{dist}%
\left( S,\cdot \right) $ in Lemma \ref{Peter's lemma} coupled with our
choice of conformal factor give that the new metric $\tilde{g}=e^{2\rho
\circ \mathrm{dist}\left( S,\cdot \right) }g$ has a more desirable
curvature. Specifically, given any positive constants $K>0$ and $\varepsilon
>0,$ there is a choice of $\rho $ and a neighborhood $\Omega $ of $S$ so
that $\tilde{g}$ has the following property. For any plane that contains a
vector tangent to the fibers of the the closest point map $Pr:\Omega
\longrightarrow S$ the sectional curvatures of $\tilde{g}$ are bounded from
below by $K,$ and, up to symmetries of the curvature tensor, all other
components of the curvature tensor of $\tilde{g}$ differ from the
corresponding components of the curvature tensor of $g$ by no more than $%
\varepsilon $ (see Theorem \ref{conf-submanf}).

The union of the singular strata of a compact $G$--manifold need not be a
submanifold, but as it is compact and the union of submanifolds, we are able
to push through a generalization of Theorem \ref{conf-submanf} that applies
to the singular strata of a $G$--manifold. This result is Theorem \ref{uber
conf}.

Our conformal change technique will also allow us to show

\begin{theoremlet}
\label{quasi positive}\noindent 1. Given $K,\varepsilon >0$, $\left(
M,g\right) $ a Riemannian $n$--manifold with $Ric_{\left( M,g\right) }\geq
n-1$ and $p\in M,$ there is a metric $\tilde{g}$ on $M$ with 
\begin{equation*}
Ric_{\left( M,\tilde{g}\right) }\geq n-1-\varepsilon \text{ and }\mathrm{sec}%
_{\left( M,\tilde{g}\right) }|_{p}\geq K.
\end{equation*}%
\vspace*{0.1in}

\noindent 2. Given $K>0$ and $\left\{ \left( M,g_{\alpha }\right) \right\}
_{\alpha =1}^{\infty }$ a family of almost non-negatively curved Riemannian $%
n $--manifolds, and $p\in M,$ there is a sequence of almost non-negatively
curved metrics $\tilde{g}_{\alpha }$ on $M$ with 
\begin{equation*}
\mathrm{sec}_{\left( M,\tilde{g}_{\alpha }\right) }|_{p}>K.
\end{equation*}%
\vspace*{0.1in}

\noindent 3. If $\left\{ \left( M,g_{\alpha }\right) \right\} _{\alpha
=1}^{\infty }$ satisfies the hypotheses of Parts 1 and 2, then there is a
sequence of metrics $\tilde{g}_{\alpha }$ on $M$ that satisfies the
conclusions of Parts 1 and 2. \vspace*{0.1in}

\noindent 4. If, in addition, $p$ is a fixed point of an isometric $G$%
--action for $g$ or $g_{\alpha }$, then the metrics $\tilde{g}$ and $\tilde{g%
}_{\alpha }$ can be chosen to be $G$--invariant.
\end{theoremlet}

Recall that $M$ is said to have quasi-positive curvature if it is
non-negatively curved and has positive curvature at a point. Just as the the
set of almost non-negatively curved metrics is an open neighborhood of the
set of non-negatively curved metrics, so too, the condition in the
conclusion of Part 2 defines an open neighborhood of the set of metrics with
quasi-positive curvature.

Since the metrics in Part 2 satisfy $\mathrm{sec}_{\left( M,\tilde{g}%
_{\alpha }\right) }|_{p}>K,$ they also have $\mathrm{sec}_{\left( M,\tilde{g}%
_{\alpha }\right) }>K$ in a neighborhood of $p.$ However, our construction
does not allow us to conclude that this neighborhood is independent of the
metric $\tilde{g}_{\alpha }.$ Part 2 suggests that a more interesting
neighborhood of the quasi-positively curved family is the set of almost
non-negatively curved metrics with $\mathrm{sec}_{\left( M,\tilde{g}_{\alpha
}\right) }>K$ on an open subset of $M$ that is independent of $\alpha .$ The
metrics on the Milnor spheres constructed in \cite{Wilh} are in such a
neighborhood.

The paper is organized as follows. In Section \ref{notation} we fix notation
and review the structure of $G$--manifolds. The discussion of the conformal
change occurs in Section \ref{section cnf change}, where we also prove
Theorem \ref{quasi positive}. Cheeger deformations are discussed in Sections %
\ref{cheeg def} and \ref{Infinitisimal}. Section \ref{cheeg def} reviews the
generalities and also discusses the $A$--tensor of the Cheeger submersion on
compact subsets of $M^{\text{reg}}.$ Section \ref{Infinitisimal} covers the
infinitesimal geometry near the singular orbits, especially as it relates to
Cheeger deformations. In Section \ref{two step}, we analyze the curvature of
a general $G$--manifold after performing a long term Cheeger deformation,
followed by the conformal change of Theorem \ref{uber conf}. Section \ref%
{ricci lift} concludes the proof of Theorem \ref{main}. Section \ref{alm
nonneg lift} finishes the proof of Theorem \ref{alm nonneg thm}, and Section %
\ref{examples section} contains the proof of Theorem \ref{examples}.

The sequence of metric deformations used to prove Theorem \ref{alm nonneg
thm} can also be used to prove Theorem \ref{main}, and hence yields a proof
of Theorem \ref{supplement}. However, the reader who is only interested in
the proof of Theorem \ref{main} can skip Sections \ref{Infinitisimal}, \ref%
{two step}, \ref{alm nonneg lift}, and \ref{examples section}. Similarly,
the reader only interested in the proof of Theorem \ref{alm nonneg thm} can
skip Sections \ref{ricci lift} and \ref{examples section}.

\begin{remarknonum}
If $\pi _{1}\left( G\right) $ is finite, then the hypothesis of Theorem \ref%
{main} that the principal orbits, $G/H,$ have finite fundamental group is
satisfied, but the converse is false. For example, the principal orbits
could be Berger spheres represented as $\left( S^{3}\times S^{1}\right)
/\Delta \left( S^{1}\right) .$ So this is a case where Theorem \ref{main} is
applicable even though $\pi _{1}\left( G\right) $ is infinite.

On the other hand, it would be desirable if the hypothesis that $\pi
_{1}\left( G/H\right) $ is finite could be replaced by $\pi _{1}\left(
M\right) $ is finite. For example the round three sphere admits an isometric
torus action with trivial principal isotropy, hence our method does not
apply to this simple example.
\end{remarknonum}

\textbf{Acknowledgments: }We thank Wilderich Tuschmann and Burkhard Wilking
for stimulating conversations on this paper. We are grateful to Pedro Sol%
\'{o}rzano for numerous discussions with the second author on possible
applications of Theorems \ref{main} and \ref{alm nonneg thm} and to the
referee for a very thorough critique of the manuscript.

Many of the foundational ideas that allow us to understand the effect of
both of our metric deformations were developed in the course of the second
author's work with Peter Petersen in \cite{PetWilh2}; so we are especially
indebted to him for numerous conversations with the second author on
curvature calculations over the years.

\section{Notation, Conventions, and Background\label{notation}}

In this section we will establish notation and review some background
material that we will use in the rest of this paper.

We assume the reader is familiar with the basics of Riemannian submersions
as discussed in \cite{Gray} or \cite{O'Neill}. We adopt the notation of \cite%
{O'Neill} for the $A$ and $T$ tensors.

For $r>0$ and a subset $A$ of a metric space $X$ we set 
\begin{equation*}
B\left( A,r\right) \equiv \left\{ \left. x\in X\text{ }\right\vert \text{ 
\textrm{dist}}\left( x,A\right) <r\right\} .
\end{equation*}

Let $S$ be a compact submanifold of a compact Riemannian manifold $\left(
M,g\right) $, and let $inj\left( S\right) $ be the normal injectivity radius
of $S.$ Let $\Omega $ be an open subset of $B\left( S,\frac{inj\left(
S\right) }{2}\right) ,$ the $\frac{inj\left( S\right) }{2}$--tubular
neighborhood of $S.$

We give $\nu \left( S\right) ,$ the normal bundle of $S,$ the Sasaki metric 
\cite{Sas}. That is, the foot point map $\nu \left( S\right) \longrightarrow
S$ is a Riemannian submersion, the metric on the vertical distribution comes
from $g,$ and the horizontal distribution, $\mathcal{\tilde{H}},$ is
determined by normal parallel transport along $S.$

Let 
\begin{equation*}
\tilde{X}\oplus \mathcal{\tilde{V}}
\end{equation*}%
be the orthogonal decomposition of the vertical distribution of $\nu \left(
S\right) \longrightarrow S,$ where $\tilde{X}$ is the radial, unit field
from the $0$--section$,$ $\nu _{0}\left( S\right) $, and $\mathcal{\tilde{V}}
$ is the orthogonal complement of $\tilde{X}.$ Set 
\begin{eqnarray}
\mathcal{H} &\equiv &d\exp _{S}^{\perp }\left( \mathcal{\tilde{H}}\right) , 
\notag \\
\mathcal{V} &\equiv &d\exp _{S}^{\perp }\left( \mathcal{\tilde{V}}\right) ,
\label{defn of X} \\
X &=&d\exp _{S}^{\perp }\left( \tilde{X}\right) ,  \notag
\end{eqnarray}%
where $\exp _{S}^{\perp }:\nu \left( S\right) \longrightarrow M$ is the
normal exponential map.

Note that $X\oplus \mathcal{V}$ is the tangent space to the fibers of the
closest point map $Pr:\Omega \setminus S\longrightarrow S$, and on $\Omega
\setminus S,$%
\begin{equation*}
X=\mathrm{grad}\left( dist\left( S,\cdot )\right) \right) .
\end{equation*}%
The distribution $\mathcal{H}$ need not be orthogonal to $X\oplus \mathcal{V}%
;$ however, we will show in Proposition \ref{angle tween H and V} that it is
asymptotically orthogonal to $X\oplus \mathcal{V}$ near $S,$ and hence is
very close to $\overline{\mathcal{H}},$ the distribution that is orthogonal
to $span\left\{ X,\mathcal{V}\right\} .$

We use superscripts to denote components of vectors in subspaces. So, for
example, $V^{\mathrm{span}\left\{ X\right\} }$ is the component of $V$ in $%
\mathrm{span}\left\{ X\right\} $ and $V^{\mathcal{V}}$ is the component of $%
V $ in $\mathcal{V}.$

We write conformal metric changes, $\tilde{g}=e^{2f}g.$ We let $\tilde{\nabla%
},$ $\tilde{R},$ $\widetilde{\sec }$ and $\widetilde{Ric}$ denote the
covariant derivative, curvature tensor, sectional curvature and Ricci tensor
of $\tilde{g}.$ We denote $R\left( X,Y,Y,X\right) $ by $\mathrm{curv}\left(
X,Y\right) .$ We write directional derivatives as $D_{V}f,$ and parameterize
geodesics by arc length.

Following \cite{OSY}, we let $\tau :\mathbb{R}^{k}\rightarrow \mathbb{R}_{+}$
be any function that satisfies 
\begin{equation}
\lim_{x_{1},\ldots ,x_{k}\rightarrow 0}\tau \left( x_{1},\ldots
,x_{k}\right) =0.  \label{dfn of tau eqn}
\end{equation}%
When making an estimate with a function $\tau ,$ we implicitly assert the
existence of such a function for which the estimate holds.

\subsection{ The Stratification of G--Manifolds}

Let $G$ act isometrically on $M$ with both $M$ and $G$ compact. For $x\in M,$
we let $G\left( x\right) $ be the orbit of $x,$ $G_{x}$ be the isotropy
subgroup at $x,$ and

\begin{equation*}
\mathfrak{g=g}_{x}\oplus \mathfrak{m}_{x}
\end{equation*}%
be the decomposition of the Lie Algebra of $G$ into $\mathfrak{g}_{x}$, the
Lie Algebra of $G_{x},$ and $\mathfrak{m}_{x}$ the orthogonal complement of $%
\mathfrak{g}_{x}$ with respect to a fixed bi-invariant metric on $G$, $g_{%
\mathrm{bi}}.$

If $G$ acts isometrically on a Riemannian manifold $M$ and $k\in \mathfrak{g}%
,$ we let $k_{M}$ denote the Killing field on $M$ generated by $k.$

Recall that $G\left( x\right) $ is called a \emph{principal }orbit if and
only if for all $y\in M$, there is a $g\in G$ with $G_{x}\subset
gG_{y}g^{-1}.$ An orbit, $G\left( x\right) ,$ is exceptional if and only if $%
G_{x}$ is a finite extension of some principal isotropy subgroup. Otherwise $%
G\left( x\right) $ is called a singular orbit. All of our arguments about
singular orbits apply to exceptional orbits, so for this paper we use the
term \emph{singular orbit }to mean any \emph{non--principal }orbit.

There is a natural stratification of $M$ into smooth submanifolds by orbit
type. The stratum of $x\in M$ is defined to be 
\begin{equation*}
S\left( G_{x}\right) \equiv \left\{ y\in M\text{ }|\text{ }\exists g\in G%
\text{ with }G_{x}=gG_{y}g^{-1}\right\} .
\end{equation*}%
We note that $\overline{S\left( G_{x}\right) }$ is 
\begin{equation*}
\overline{S\left( G_{x}\right) }\equiv \left\{ y\in M\text{ }|\text{ }%
\exists g\in G\text{ with }G_{x}\subset gG_{y}g^{-1}\right\} .
\end{equation*}

Partially order the closed sets $\overline{S\left( G_{x}\right) }$ by
inclusion. If $\overline{S\left( G_{x}\right) }$ is minimal with respect to
this partial order, then $\overline{S\left( G_{x}\right) }=S\left(
G_{x}\right) $ is a closed submanifold.

The union of the principal orbits is called the regular part of $M,$ which
we denote by $M^{\text{reg}}.$ Recall that we have a proper Riemannian
submersion 
\begin{equation*}
\pi ^{\text{reg}}=\pi |_{M^{\text{reg}}}:M^{\text{reg}}\longrightarrow M^{%
\text{reg}}/G.
\end{equation*}

Throughout the paper we assume that $G$ is a compact, connected Lie group
acting isometrically and effectively on a compact Riemannian $n$--manifold $%
\left( M,g\right) $ with singular strata, $S_{1},S_{2},\ldots ,S_{p}.$

\begin{proposition}
\label{Omega_i Dfn}There is a neighborhood $\Omega \equiv \cup \Omega ^{i}$
of the singular strata, $\cup S_{i},$ and for each $i,$ a compact subset $%
\mathcal{C}_{i}\subset S_{i}$. For each $i$, $\Omega ^{i}$ and $\mathcal{C}%
_{i}$ are related as follows.

Let $\mathrm{int}\left( \mathcal{C}_{i}\right) $ be the interior of $%
\mathcal{C}_{i}$ when viewed as a subset of $S_{i}.$ Let $\mathrm{inj}\left( 
\mathcal{C}_{i}\right) $ be the injectivity radius of the normal bundle $\nu
\left( S_{i}\right) |_{\mathcal{C}_{i}},$ and let $\nu _{0}\left(
S_{i}\right) |_{\mathcal{C}_{i}}$ be the image of the zero section of $\nu
\left( S_{i}\right) |_{\mathcal{C}_{i}}\longrightarrow \mathcal{C}_{i}.$ Then%
\begin{equation*}
\Omega ^{i}=\exp _{S_{i}}^{\perp }\left( B\left( \nu _{0}\left( S_{i}\right)
|_{\mathrm{int}\left( \mathcal{C}_{i}\right) },r_{i}\right) \right) ,
\end{equation*}%
where $r_{i}\in \left( 0,\frac{\mathrm{inj}\left( \mathcal{C}_{i}\right) }{2}%
\right) .$
\end{proposition}

\begin{proof}
We define the \emph{Descendant Number}\ of a stratum $S_{i}$ to be the
integer, $\mathcal{D}\left( S_{i}\right) ,$ if there are precisely $\mathcal{%
D}\left( S_{i}\right) $ strata contained in $\overline{S_{i}}.$ Call the
union of the strata with Descendant Number $l,$ $\mathbb{S}^{l}.$ We denote
by $S_{\alpha }^{l}$ the strata so that 
\begin{equation*}
\mathbb{S}^{l}\setminus \mathbb{S}^{l-1}=\cup _{\alpha \in I_{l}}S_{\alpha
}^{l}.
\end{equation*}

The first step to prove the proposition is to establish the following
induction statement.

\noindent \textbf{Induction Statement: }For each $l,$ there are compact
subsets $\mathcal{C}_{l,\alpha }$ of $S_{\alpha }^{l}$ and neighborhoods, $%
U^{l}$ and $V^{l},$ of $\mathbb{S}^{l}$ of the form 
\begin{equation*}
V^{l}\equiv \cup _{k=1}^{l}\cup _{\alpha \in I_{k}}V_{\alpha }^{k}
\end{equation*}%
and 
\begin{equation*}
U^{l}\equiv \cup _{k=1}^{l}\cup _{\alpha \in I_{k}}U_{\alpha }^{k},
\end{equation*}
where 
\begin{eqnarray*}
V_{\alpha }^{l} &=&\exp _{S_{\alpha }^{l}}^{\perp }\left( B\left( \nu
_{0}\left( S_{\alpha }^{l}\right) |_{\mathrm{int}\left( \mathcal{C}%
_{l,\alpha }\right) },r_{l,\alpha }\right) \right) , \\
U_{\alpha }^{l} &=&\exp _{S_{\alpha }^{l}}^{\perp }\left( B\left( \nu
_{0}\left( S_{\alpha }^{l}\right) |_{\mathrm{int}\left( \mathcal{C}%
_{l,\alpha }\right) },\frac{r_{l,\alpha }}{2}\right) \right) ,
\end{eqnarray*}%
and $r_{l,\alpha }\in \left( 0,\frac{\mathrm{inj}\left( \mathcal{C}%
_{l,\alpha }\right) }{2}\right) .$

We prove this statement by induction on the Descendant Number. The strata
with Descendant Number $1$ contain no strata other than themselves and hence
are compact submanifolds. Let $\mathcal{C}_{1,\alpha }=S_{\alpha }^{1}$ and
for $r_{1,\alpha }\in \left( 0,\frac{\mathrm{inj}\left( S_{\alpha
}^{1}\right) }{2}\right) ,$ let $V_{\alpha }^{1}\equiv B\left( S_{\alpha
}^{1},r_{1,\alpha }\right) $ and $U_{\alpha }^{1}\equiv B\left( S_{\alpha
}^{1},\frac{r_{1,\alpha }}{2}\right) .$

Suppose we have constructed $U^{1},\ldots ,U^{l},$ $V^{1},\ldots ,V^{l},$
and $\left\{ \mathcal{C}_{1,\alpha }\right\} _{\alpha \in I_{1}},\ldots
,\left\{ \mathcal{C}_{l,\alpha }\right\} _{\alpha \in I_{l}}$ with the
desired properties. Set $\mathcal{C}_{l+1,\alpha }=S_{\alpha
}^{l+1}\setminus \left\{ U^{l}\cap S_{\alpha }^{l+1}\right\} .$ Note that $%
\mathbb{S}^{l+1}\subset V^{l}\dbigcup \cup _{\alpha \in I_{l+1}}\mathrm{int}%
\left( \mathcal{C}_{l+1,\alpha }\right) .$ For $r_{l+1,\alpha }\in \left( 0,%
\frac{\mathrm{inj}\left( \mathcal{C}_{l+1,\alpha }\right) }{2}\right) $ we
set 
\begin{eqnarray*}
V_{\alpha }^{l+1} &=&\exp _{S_{\alpha }^{l+1}}^{\perp }\left( B\left( \nu
_{0}\left( S_{\alpha }^{l+1}\right) |_{\mathrm{int}\left( \mathcal{C}%
_{l+1,\alpha }\right) },r_{l+1,\alpha }\right) \right) \text{ and} \\
U_{\alpha }^{l+1} &=&\exp _{S_{\alpha }^{l+1}}^{\perp }\left( B\left( \nu
_{0}\left( S_{\alpha }^{l+1}\right) |_{\mathrm{int}\left( \mathcal{C}%
_{l+1,\alpha }\right) },\frac{r_{l+1,\alpha }}{2}\right) \right) .
\end{eqnarray*}

Note that 
\begin{equation*}
V^{l+1}\equiv \cup _{k=1}^{l+1}\cup _{\alpha \in I_{k}}V_{\alpha }^{k}
\end{equation*}%
and 
\begin{equation*}
U^{l+1}\equiv \cup _{k=1}^{l+1}\cup _{\alpha \in I_{k}}U_{\alpha }^{k}
\end{equation*}%
are neighborhoods of $\mathbb{S}^{l+1},$ proving the induction statement.

The proposition follows from the induction statement by re-indexing so that
the $S_{\alpha }^{l}$ become the $S_{i}$ and the $V_{\alpha }^{k}$ become
the $\Omega ^{i}$.
\end{proof}

Notice that for each $\Omega ^{i}$ we have a splitting of $T\left( \Omega
^{i}\setminus \mathcal{S}^{i}\right) $ as in Equation \ref{defn of X}. We
call this splitting 
\begin{equation}
\mathcal{H}^{i}\oplus \mathcal{V}^{i}\oplus \mathrm{span}\left\{
X^{i}\right\} .  \label{HVX splitting}
\end{equation}

\section{Conformal Change\label{section cnf change}}

In this section we establish a universal property of any compact submanifold 
$S$ of any complete Riemannian manifold, $\left( M,g\right) .$ It is stated
formally in Theorem \ref{conf-submanf}, below, which we describe briefly
here.

Given any positive constants $K$ and $\varepsilon ,$ there is a conformal
change $\tilde{g}$ of $g$ and there are neighborhoods $\Omega _{1}\subset
\Omega _{3}\subset B\left( S,\frac{inj\left( S\right) }{2}\right) $ so that
the new metric is $C^{1}$--close to $g$, agrees with $g$ outside of $\Omega
_{3},$ and also has the following property.

For any plane that contains a vector in $\mathrm{span}\left\{ X|_{\Omega
_{1}}\right\} \oplus \mathcal{V}|_{\Omega _{1}}$ the sectional curvatures of 
$\tilde{g}$ are bounded from below by $K,$ and, up to symmetries of the
curvature tensor, all other components of $\tilde{R}$ differ from the
corresponding components of $R$ by no more than $\varepsilon .$ To prove
this we exploit some universal estimates for the asymptotic behavior of $%
Hess_{\mathrm{dist}\left( S,\cdot \right) }$ near $S.$

We then generalize the conformal change result to a neighborhood of the
union of the singular strata of a $G$--action in Theorem \ref{uber conf}.
Since the singular strata are typically non-compact, we first prove an
intermediate result, Theorem \ref{non-compact conf}, that generalizes
Theorem \ref{conf-submanf} to compact subsets of non-compact manifolds. This
will allow us to extend Theorem \ref{conf-submanf} to the union of the
singular strata, in part because each stratum has a compact exhaustion.

\subsection{Conformal Change Around a Compact Submanifold}

\begin{theorem}
\label{conf-submanf}Let $\left( M,g\right) $ be a compact Riemannian $n$%
--manifold. Let $S$ be a compact, smooth submanifold of $\left( M,g\right) $%
. For any $\varepsilon ,K>0$ there are neighborhoods $\Omega _{1}\subset
\Omega _{3}$ of $S$ and a metric $\tilde{g}=e^{2f}g$ with the following
properties.\vspace*{0.1in}

\noindent 1. The metrics $\tilde{g}$ and $g$ coincide on $M\setminus \Omega
_{3}.$\vspace*{0.1in}

\noindent 2. For all $V\in \mathrm{span}\left\{ X\right\} \oplus \mathcal{V}$
and for all $Z\in T\Omega _{1}$ 
\begin{equation}
\widetilde{\mathrm{sec}}\left( V,Z\right) |_{\Omega _{1}}>K.  \label{sec > K}
\end{equation}%
\vspace*{0.1in}

\noindent 3. If $\left\{ E_{1},\ldots ,E_{n}\right\} $ is a local
orthonormal frame for $\Omega _{3}$ with $X=E_{1}$ and \textrm{span}$\left\{
E_{2},\ldots ,E_{r}\right\} =\mathcal{V}$ for $2\leq r\leq n,$ then 
\begin{equation*}
\left\vert \tilde{R}\left( E_{i},E_{j},E_{k},E_{l}\right) -R\left(
E_{i},E_{j},E_{k},E_{l}\right) \right\vert <\varepsilon ,
\end{equation*}%
except if the quadruple corresponds, up to a symmetry of the curvature
tensor, to the sectional curvature of a plane containing a vector $V\in 
\mathrm{span}\left\{ X\right\} \cup \mathcal{V}$.\vspace*{0.1in}

\noindent 4. For all $U,W\in TM.$%
\begin{equation*}
\widetilde{\mathrm{sec}}\left( U,W\right) >\mathrm{sec}\left( U,W\right)
-\varepsilon .
\end{equation*}%
\vspace*{0.1in}

\noindent 5. If $G$ acts isometrically on $\left( M,g\right) $ and $S$ is $G$%
--invariant, then we may choose $\tilde{g}$ to be $G$--invariant.
\end{theorem}

\begin{remark}
While this theorem does not imply that $\Omega _{1}$ is almost
non-negatively curved, we can conclude, with appropriate choices of $%
\varepsilon $ and $K,$ that $Ric_{\tilde{g}}|_{\Omega _{1}}>1.$
\end{remark}

We get Theorem \ref{quasi positive} by applying Theorem \ref{conf-submanf}
in the special case when $S$ is a point.

In our proof of Theorem \ref{conf-submanf}, our conformal factor will have
the form $e^{2f},$ where $f=\rho \circ \mathrm{dist}\left( S,\cdot \right) $%
, $\rho :\left[ 0,\infty \right) \longrightarrow \mathbb{R}$ is $C^{\infty
}, $ satisfies $\rho |_{\left( \frac{inj\left( S\right) }{2},\infty \right)
}\equiv 0,$ and will be further specified later. We set $\tilde{g}=e^{2f}g.$

For ease of notation we set, 
\begin{eqnarray*}
f^{\prime } &\equiv &\rho ^{\prime }\circ \mathrm{dist}\left( S,\cdot
\right) , \\
\mathrm{grad}f &=&f^{\prime }X, \\
f^{\prime \prime } &\equiv &\rho ^{\prime \prime }\circ \mathrm{dist}\left(
S,\cdot \right) .
\end{eqnarray*}%
The main step to prove Theorem \ref{conf-submanf} is the following.

\begin{keylemma}
\label{conformal curv} For every $\varepsilon ,K>0,$ there is a $\delta >0,$
and a $\sigma _{1}\in \left( 0,\frac{injS}{2}\right) $ so that the following
holds.

Suppose that for all $Z\in T\Omega $, for all $V\in \mathrm{span}\left\{ X,%
\mathcal{V}\right\} ,$ and for some $\sigma _{3}\in \left( \sigma _{1},\frac{%
injS}{2}\right) ,$ 
\begin{eqnarray}
R\left( Z,V,V,Z\right) |_{_{B\left( S,\sigma _{1}\right) }} -f^{\prime
\prime }|_{_{B\left( S,\sigma _{1}\right) }}\left\vert Z\right\vert
^{2}\left\vert V^{\mathrm{span}\left\{ X\right\} }\right\vert ^{2} - \frac{%
f^{\prime }}{\mathrm{dist}\left( S,\cdot \right) }|_{_{B\left( S,\sigma
_{1}\right) }} \left\vert V^{\mathcal{V}}\right\vert ^{2} &\left\vert
Z\right\vert ^{2} &  \notag \\
&\geq & \left( K+1\right) \left\vert V\right\vert ^{2}\left\vert
Z\right\vert ^{2}  \label{nasty inequality}
\end{eqnarray}%
\noindent 
\begin{eqnarray*}
f^{\prime } &\leq &0, \\
f^{\prime \prime }|_{B\left( S,\sigma _{1}\right) } &\leq &0,
\end{eqnarray*}%
\begin{eqnarray*}
\left\vert f\right\vert +\left\vert f^{\prime }\right\vert &<&\delta , \\
f^{\prime \prime } &<&\delta , \\
f|_{M\setminus B\left( S,\sigma _{3}\right) } &\equiv &0.
\end{eqnarray*}%
Then

\begin{enumerate}
\item 
\begin{equation}
\widetilde{\mathrm{sec}}\left( V,Z\right) |_{B\left( S,\sigma _{1}\right)
}>K.  \label{sec > K-key}
\end{equation}

\item If $\left\{ E_{1},\ldots ,E_{n}\right\} $ is a local orthonormal frame
for $B\left( S,\sigma _{3}\right) $ with $X=E_{1}$ and \textrm{span}$\left\{
E_{2},\ldots ,E_{r}\right\} =\mathcal{V}$ for $2\leq r\leq n,$ then 
\begin{equation}
\left\vert \tilde{R}\left( E_{i},E_{j},E_{k},E_{l}\right) -R\left(
E_{i},E_{j},E_{k},E_{l}\right) \right\vert <\varepsilon ,  \label{R est}
\end{equation}%
except if the quadruple corresponds, up to a symmetry of the curvature
tensor, to the sectional curvature of a plane containing a vector $V\in 
\mathrm{span}\left\{ X\right\} \cup \mathcal{V}$.

\item For all $Z,W\in TM.$ 
\begin{equation}
\widetilde{\mathrm{sec}}\left( Z,W\right) >\mathrm{sec}\left( Z,W\right)
-\varepsilon .  \label{not foulded your father}
\end{equation}
\end{enumerate}
\end{keylemma}

Recall from page 144 of \cite{Wals} 
\begin{eqnarray*}
e^{-2f}\tilde{R}\left( V,Y,Z,U\right) &=&R\left( V,Y,Z,U\right) -g\left(
V,U\right) \mathrm{Hess}_{f}\left( Y,Z\right) -g\left( Y,Z\right) \mathrm{%
Hess}_{f}\left( V,U\right) \\
&&+g\left( V,Z\right) \mathrm{Hess}_{f}\left( Y,U\right) +g\left( Y,U\right) 
\mathrm{Hess}_{f}\left( V,Z\right) \\
&&+g\left( V,U\right) D_{Y}fD_{Z}f+g\left( Y,Z\right) D_{V}fD_{U}f \\
&&-g\left( Y,U\right) D_{V}fD_{Z}f-g\left( V,Z\right) D_{Y}fD_{U}f \\
&&-g\left( Y,Z\right) g\left( V,U\right) \left\vert \mathrm{grad}%
f\right\vert ^{2}+g\left( V,Z\right) g\left( Y,U\right) \left\vert \mathrm{%
grad}f\right\vert ^{2}.
\end{eqnarray*}%
Since we assume $\left\vert f^{\prime }\right\vert <\delta $ this becomes 
\begin{eqnarray}
e^{-2f}\tilde{R}\left( V,Y,Z,U\right) &=&R\left( V,Y,Z,U\right) -g\left(
V,U\right) \mathrm{Hess}_{f}\left( Y,Z\right) -g\left( Y,Z\right) \mathrm{%
Hess}_{f}\left( V,U\right)  \notag  \label{confffcurv} \\
&&+g\left( V,Z\right) \mathrm{Hess}_{f}\left( Y,U\right) +g\left( Y,U\right) 
\mathrm{Hess}_{f}\left( V,Z\right)  \label{conf
curv tensor} \\
&&\pm O\left( \delta ^{2}\right) \left\vert V\right\vert \left\vert
Y\right\vert \left\vert Z\right\vert \left\vert U\right\vert .  \notag
\end{eqnarray}

So to prove the Key Lemma we need an understanding of $\mathrm{Hess}_{f},$
which will be addressed in the next subsection.

\subsection{Universal Infinitesimal Geometry of Tubular Neighborhoods}

\begin{proposition}
\label{Jacobi Rep}Let $X,\mathcal{V},$ and $\mathcal{H}$ be as in \ref{defn
of X}. Along a unit speed geodesic, $\gamma $ in $\Omega ,$ that leaves $S$
orthogonally at $\gamma \left( 0\right) $ we have the following.\vspace*{%
0.1in}

\noindent 1. At $\gamma \left( t\right) ,$ any vector in $\mathcal{V}$ has
the form $J\left( t\right) $ where $J$ is a Jacobi field along $\gamma $
that satisfies 
\begin{eqnarray}
J\left( 0\right) &=&0,  \notag \\
J^{\prime }\left( 0\right) &\in &\nu _{\gamma \left( 0\right) }\left(
S\right) \cap \gamma ^{\prime }\left( 0\right) ^{\perp }.\text{\label{Js
spanning V}}
\end{eqnarray}%
\vspace*{0.1in}\noindent 2. At $\gamma \left( t\right) ,$ any vector in $%
\mathcal{H}$ has the form $J\left( t\right) $ where $J$ is a Jacobi field
along $\gamma $ that satisfies 
\begin{equation}
J\left( 0\right) ,J^{\prime }\left( 0\right) \in T_{\gamma \left( 0\right)
}S.  \label{Js spanning H-II}
\end{equation}%
\vspace*{0.1in}\noindent 3. Let $Sh_{\gamma ^{\prime }\left( 0\right) }$ be
the shape operator of $S$ at $\gamma \left( 0\right) $ in the direction of $%
\gamma ^{\prime }\left( 0\right) .$ That is 
\begin{equation*}
Sh_{\gamma ^{\prime }\left( 0\right) }:T_{\gamma \left( 0\right)
}S\longrightarrow T_{\gamma \left( 0\right) }S
\end{equation*}
is $Sh_{\gamma ^{\prime }\left( 0\right) }\left( v\right) \equiv \left(
\nabla _{v}Z\right) ^{T_{\gamma \left( 0\right) }S}$ where $Z$ is any
extension of $\gamma ^{\prime }\left( 0\right) $ to a field in $\nu \left(
S\right) .$ Then the Jacobi fields in Part 2 also satisfy 
\begin{equation*}
J^{\prime }\left( 0\right) =Sh_{\gamma ^{\prime }\left( 0\right) }\left(
J\left( 0\right) \right) .
\end{equation*}%
\vspace*{0.1in}\noindent 4. The distribution 
\begin{equation}
\mathcal{\bar{V}}_{\gamma \left( t\right) }\equiv \left\{ \left. \mathcal{V}%
_{\gamma \left( t\right) }\right\vert \text{ }t>0\right\} \cup \left\{ \nu
_{\gamma \left( 0\right) }\left( S\right) \cap \gamma ^{\prime }\left(
0\right) ^{\perp }\right\}  \label{extension of V}
\end{equation}%
is smooth along $\gamma .$
\end{proposition}

\begin{proof}
A vector in $\mathcal{\tilde{V}}$ is a value of a variation field of a
variation of lines leaving the origin in a normal fiber. Since $\mathcal{V}%
\equiv d\exp _{S}^{\perp }\left( \mathcal{\tilde{V}}\right) ,$ a vector in $%
\mathcal{V}$ is tangent to a variation of geodesics that leave $S$
orthogonally from a single point, and Part 1 follows.

With respect to the Sasaki metric $\nu \left( S\right) \longrightarrow S$ is
a Riemannian submersion whose horizontal spaces $\mathcal{\tilde{H}}$ are
given by normal parallel transport of vectors in $\nu \left( S\right) $
along curves in $S.$ That is, if $Z:\left[ a,b\right] \longrightarrow \nu
\left( S\right) $ is a horizontal lift of a curve $c:\left[ a,b\right]
\longrightarrow S,$ then 
\begin{equation*}
\left( \nabla _{c^{\prime }}Z\right) ^{\nu \left( S\right) }=0.
\end{equation*}

Exponentiating all real multiples of such a field, $Z,$ produces a variation
of geodesics whose tangent field is $X.$ Along the geodesic $t\longmapsto
\exp _{c\left( 0\right) }tZ\left( 0\right) $ the variation field, $J,$
satisfies 
\begin{equation*}
J\left( 0\right) =c^{\prime }\left( 0\right) \in T_{\gamma \left( 0\right)
}S.
\end{equation*}%
Since $J^{\prime }\left( 0\right) =\nabla _{J\left( 0\right) }Z=\nabla
_{c^{\prime }\left( 0\right) }Z$ and $\left( \nabla _{c^{\prime }}Z\right)
^{\nu \left( S\right) }=0,$ it follows that 
\begin{equation*}
J^{\prime }\left( 0\right) \in T_{\gamma \left( 0\right) }S,
\end{equation*}%
proving Part 2, and also Part 3 since 
\begin{equation*}
Sh_{\gamma ^{\prime }\left( 0\right) }\left( J\left( 0\right) \right)
=\nabla _{c^{\prime }\left( 0\right) }Z=J^{\prime }\left( 0\right) .
\end{equation*}

Combining the proofs of Parts 1 and 2, we see that together the families of
Jacobi fields that span $\mathcal{H}\oplus \mathcal{V}$ come from variations
of geodesics that leave $S$ orthogonally. In particular, they form an $%
\left( n-1\right) $--dimensional family of Jacobi fields on which the
Riccati operator is self--adjoint \cite{Wilk}.

Let $\mathcal{J}^{V}$ be the family of Jacobi fields along $\gamma $ from
Part $1.$ That is 
\begin{equation*}
\mathcal{J}^{V}\equiv \left\{ \left. J\right\vert J\left( 0\right)
=0,J^{\prime }\left( 0\right) \in \nu _{\gamma \left( 0\right) }\left(
S\right) \cap \gamma ^{\prime }\left( 0\right) ^{\perp }\right\} .
\end{equation*}%
It follows from Part 1, that for $t\in \left( 0,inj\left( S\right) \right) ,$%
\begin{equation*}
\left\{ \left. \mathcal{V}_{\gamma \left( t\right) }\right\vert t>0\right\} =%
\mathrm{span}\left\{ \left. J\left( t\right) \right\vert J\in \mathcal{J}%
^{V}\right\} .
\end{equation*}%
For $t=0,$ we have that 
\begin{equation*}
\nu _{\gamma \left( 0\right) }\left( S\right) \cap \gamma ^{\prime }\left(
0\right) ^{\perp }=\mathrm{span}\left\{ \left. J^{\prime }\left( 0\right)
\right\vert J\in \mathcal{J}^{V}\right\} .
\end{equation*}%
On the other hand, given a nonzero $J\in \mathcal{J}^{V},$ then for all $%
t\in \left( 0,inj\left( S\right) \right) ,$ $J\left( t\right) \neq 0;$ so%
\begin{equation*}
\mathrm{span}\left\{ \left. J^{\prime }\left( t\right) \right\vert J\in 
\mathcal{J}^{V},J\left( t\right) =0\right\} =\mathrm{span}\left\{ \left.
J^{\prime }\left( 0\right) \right\vert J\in \mathcal{J}^{V}\right\} .
\end{equation*}

Therefore for $t\in \left( -inj\left( S\right) ,inj\left( S\right) \right) $ 
\begin{equation}
\mathcal{\bar{V}}_{\gamma \left( t\right) }=\mathrm{span}\left\{ \left.
J\left( t\right) \right\vert J\in \mathcal{J}^{V}\right\} \oplus \mathrm{span%
}\left\{ \left. J^{\prime }\left( t\right) \right\vert J\in \mathcal{J}%
^{V},J\left( t\right) =0\right\} .  \label{wilking decomp--2}
\end{equation}

As asserted on page 1300 of \cite{Wilk}, $\mathcal{\bar{V}}_{\gamma \left(
t\right) }$ depends smoothly on $t$, cf. Lemma 1.7.1 in \cite{GromWals}.
This proves Part 4.
\end{proof}

\begin{remark}
Note that the first summand in \ref{wilking decomp--2} vanishes only at $t=0$
and the second summand is only nonzero at $t=0$.
\end{remark}

\begin{lemma}
\label{Jacobi Bounds}Let $S$ be a compact submanifold of a Riemannian $n$%
--manifold $M.$ There are constants $C_{1},C_{2}$ so that if $\gamma :\left[
0,l\right] \longrightarrow \Omega \subset B\left( S,\frac{inj\left( S\right) 
}{2}\right) $ is any unit speed geodesic that leaves $S$ orthogonally and $J$
is any Jacobi field along $\gamma $ as in \ref{Js spanning H-II}, then 
\begin{equation*}
\left\vert J\left( t\right) \right\vert \geq C_{1}\left\vert J\left(
0\right) \right\vert \text{ and }\left\vert J^{\prime }\left( t\right)
\right\vert \leq C_{2}\left\vert J\left( 0\right) \right\vert .
\end{equation*}
\end{lemma}

\begin{proof}
All constants that we discuss in this proof are independent of $\gamma .$
For simplicity we also suppose that $\left\vert J\left( 0\right) \right\vert
=1.$

Let $\tilde{V}$ be the variation of lines in $\nu \left( S\right) $ that
corresponds to $J.$ Then $\tilde{V}\left( t,0\right) =t\gamma ^{\prime
}\left( 0\right) $ and the variation field $\frac{\partial }{\partial s}%
\tilde{V}|_{s=0}$ consists of lifts of $J\left( 0\right) $ to the normal
bundle, $\nu \left( S\right) ,$ along $t\gamma ^{\prime }\left( 0\right) .$
In particular, $\left\vert \frac{\partial }{\partial s}\tilde{V}%
|_{s=0}\right\vert \equiv \left\vert J\left( 0\right) \right\vert =1.$ Since 
$J$ is the image of $\frac{\partial }{\partial s}\tilde{V}|_{s=0}$ under $%
d\exp ,$ it follows from compactness that there is a constant $C_{3}>0$ so
that 
\begin{equation}
\left\vert J\left( t\right) \right\vert \leq C_{3}.  \label{J upper}
\end{equation}%
Also since $\gamma \left( t\right) \subset B\left( S,\frac{inj\left(
S\right) }{2}\right) ,$ $J\left( t\right) \neq 0,$ so there is a constant $%
C_{1}>0$ so that 
\begin{equation*}
\left\vert J\left( t\right) \right\vert \geq C_{1}=C_{1}\left\vert J\left(
0\right) \right\vert .
\end{equation*}%
Since $J^{\prime }\left( 0\right) =Sh_{\gamma ^{\prime }\left( 0\right)
}\left( J\left( 0\right) \right) ,$ by continuity of the shape operator and
compactness of the unit normal bundle of $S$ there is a constant $C_{4}>0$
so that 
\begin{equation}
\left\vert J^{\prime }\left( 0\right) \right\vert \leq C_{4}=C_{4}\left\vert
J\left( 0\right) \right\vert .  \label{J' upper B}
\end{equation}%
Let $\left\{ E_{i}\right\} _{i=1}^{n}$ be an orthonormal parallel frame
along $\gamma $ with $E_{1}\left( t\right) =\gamma ^{\prime }\left( t\right) 
$ and write 
\begin{equation*}
J\left( t\right) =\Sigma _{i=2}^{n}e^{i}\left( t\right) E_{i}.
\end{equation*}%
Then 
\begin{equation*}
J^{\prime }\left( t\right) =\Sigma _{i=2}^{n}\left( e^{i}\right) ^{\prime
}\left( t\right) E_{i}\text{ and }-R\left( J,\gamma ^{\prime }\right) \gamma
^{\prime }=J^{\prime \prime }\left( t\right) =\Sigma _{i=2}^{n}\left(
e^{i}\right) ^{\prime \prime }\left( t\right) E_{i}.
\end{equation*}

So 
\begin{eqnarray*}
\left\vert \Sigma _{i=2}^{n}\left( e^{i}\right) ^{\prime \prime }\left(
t\right) E_{i}\right\vert &\leq &\left\vert R\right\vert \left\vert
J\right\vert \\
&\leq &C_{3}\left\vert R\right\vert \left\vert J\left( 0\right) \right\vert ,%
\text{ by }\ref{J upper} \\
&\leq &C_{5}\left\vert J\left( 0\right) \right\vert ,\text{ by compactness
of }M,
\end{eqnarray*}%
for some constant $C_{5}>0$. Combining this with Inequality \ref{J' upper B}
and the Fundamental Theorem of Calculus completes the proof.
\end{proof}

The following is from a revised version of \cite{PetWilh2}.

\begin{lemma}
\label{Peter's lemma}There is a constant $C>0$ so that on $\Omega \setminus
S $ we have

\begin{enumerate}
\item $\left( \mathrm{Hess}_{\mathrm{dist}\left( S,\cdot \right) }\right)
\left( X,\cdot \right) =0.$

\item For $Z\in \mathcal{H}$ and $Y\in \mathcal{V}\oplus \mathcal{H}$ 
\begin{equation*}
\left\vert \left( \mathrm{Hess}_{\mathrm{dist}\left( S,\cdot \right)
}\right) \left( Y,Z\right) \right\vert <C\left\vert Y\right\vert \left\vert
Z\right\vert .
\end{equation*}

\item For $V\in \mathcal{V}$ and $W\in \mathcal{V}\oplus \mathcal{H}$ 
\begin{equation}
\left\vert \left( \mathrm{Hess}_{\mathrm{dist}\left( S,\cdot \right) }\left(
V,W\right) -\frac{1}{\mathrm{dist}\left( S,\cdot \right) }g\left( V,W\right)
\right) \right\vert <O\left( \mathrm{dist}\left( S,\cdot \right) \right)
\left\vert V\right\vert \left\vert W\right\vert .
\end{equation}
\end{enumerate}
\end{lemma}

\begin{proof}
Recall that $X=\mathrm{grad}\left( \mathrm{dist}\left( S,\cdot \right)
\right) .$ So $\nabla _{X}X=0,$ and therefore $\left( \mathrm{Hess}_{\mathrm{%
dist}\left( S,\cdot \right) }\right) \left( X,\cdot \right) =0.$

To prove the estimates in Parts 2 and 3, we first focus on a fixed geodesic $%
\gamma :\left[ 0,l\right] \longrightarrow \Omega $ that leaves $S$
orthogonally at time $0.$

For the second estimate, we let $J$ be a Jacobi field along $\gamma $ as in %
\ref{Js spanning H-II}. Then for $Y\in \mathcal{V}_{\gamma \left( t\right)
}\oplus \mathcal{H}_{\gamma \left( t\right) }$ 
\begin{eqnarray}
\mathrm{Hess}_{\mathrm{dist}\left( S,\cdot \right) }\left( J\left( t\right)
,Y\right) &=&g\left( \nabla _{J\left( t\right) }X,Y\right)  \notag \\
&=&g\left( J^{\prime }\left( t\right) ,Y\right) .  \label{rough Hess}
\end{eqnarray}%
By Lemma \ref{Jacobi Bounds}, there are constants $C_{1},C_{2}$ so that $%
\left\vert J\left( t\right) \right\vert \geq C_{1}\left\vert J\left(
0\right) \right\vert $ and $\left\vert J^{\prime }\left( t\right)
\right\vert \leq C_{2}\left\vert J\left( 0\right) \right\vert .$ So for $%
Y\in \mathcal{V}_{\gamma \left( t\right) }\oplus \mathcal{H}_{\gamma \left(
t\right) }$ and $Z\in \mathcal{H}_{\gamma \left( t\right) }$ with $Z=J\left(
t\right) $ as in \ref{Js spanning H-II} we get from \ref{rough Hess} 
\begin{eqnarray*}
\left\vert \left( \mathrm{Hess}_{\mathrm{dist}\left( S,\cdot \right)
}\right) \left( Z,Y\right) \right\vert &\leq &\left\vert J^{\prime }\left(
t\right) \right\vert \left\vert Y\right\vert \\
&\leq &C_{2}\left\vert J\left( 0\right) \right\vert \left\vert Y\right\vert
\\
&\leq &\frac{C_{2}}{C_{1}}\left\vert J\left( t\right) \right\vert \left\vert
Y\right\vert \\
&\leq &C\left\vert Y\right\vert \left\vert Z\right\vert ,
\end{eqnarray*}%
for $C=C_{2}/C_{1}$, proving Part 2, along $\gamma .$

Similarly, for $J$ as in \ref{Js spanning V} and $Y\in \mathcal{V}\oplus 
\mathcal{H}$, we have $\mathrm{Hess}_{\mathrm{dist}\left( S,\cdot \right)
}\left( J\left( t\right) ,Y\right) =g\left( J^{\prime }\left( t\right)
,Y\right) .$ So 
\begin{eqnarray*}
\mathrm{Hess}_{\mathrm{dist}\left( S,\cdot \right) }\left( \frac{J\left(
t\right) }{\left\vert J\left( t\right) \right\vert },Y\right) &=&\frac{1}{%
\left\vert J\left( t\right) \right\vert }g\left( J^{\prime }\left( t\right)
,Y\right) \text{ and} \\
\mathrm{Hess}_{\mathrm{dist}\left( S,\cdot \right) }\left( \frac{J\left(
t\right) }{\left\vert J\left( t\right) \right\vert },\frac{J\left( t\right) 
}{\left\vert J\left( t\right) \right\vert }\right) &=&\frac{1}{\left\vert
J\left( t\right) \right\vert ^{2}}g\left( J^{\prime }\left( t\right)
,J\left( t\right) \right) .
\end{eqnarray*}

Write $J\left( t\right) =\sum\limits_{i=2}^{n}e^{i}E_{i},$ where $\left\{
E_{i}\right\} _{i=1}^{n}$ is an orthonormal parallel frame along $\gamma $
with $E_{1}\left( t\right) =\gamma ^{\prime }\left( t\right) $ and $%
E_{2}\left( 0\right) =J^{\prime }\left( 0\right) .$ Then 
\begin{eqnarray*}
e^{i}\left( 0\right) &=&0\text{ for }i\geq 2, \\
\left( e^{2}\right) ^{\prime }\left( 0\right) &=&1,\text{ }\left(
e^{i}\right) ^{\prime }\left( 0\right) =0\text{ for }i\geq 3,
\end{eqnarray*}

and since $0=-R\left( J,\dot{\gamma}\right) \dot{\gamma}|_{0}=J^{\prime
\prime }\left( 0\right) =\sum\limits_{i=2}^{n}\left( e^{i}\right) ^{\prime
\prime }\left( 0\right) E_{i}\left( 0\right) ,$ 
\begin{equation*}
\left( e^{i}\right) ^{\prime \prime }\left( 0\right) =0\text{ for }i\geq 2.
\end{equation*}%
So%
\begin{eqnarray}
J\left( t\right) &=&\left( t+O\left( t^{3}\right) \right) E_{2}\left(
t\right) +\sum\limits_{i=3}^{n}O\left( t^{3}\right) E_{i}\text{,\label{J
Taylor}} \\
J^{\prime }\left( t\right) &=&\left( 1+O\left( t^{2}\right) \right)
E_{2}\left( t\right) +\sum\limits_{i=3}^{n}O\left( t^{2}\right) E_{i}  \notag
\\
\left\vert J\left( t\right) \right\vert &=&t+O\left( t^{3}\right)  \notag \\
\left\vert J\left( t\right) \right\vert ^{2} &=&t^{2}+O\left( t^{4}\right) .
\notag
\end{eqnarray}%
So 
\begin{eqnarray*}
\mathrm{Hess}_{\mathrm{dist}\left( S,\cdot \right) }\left( \frac{J\left(
t\right) }{\left\vert J\left( t\right) \right\vert },\frac{J\left( t\right) 
}{\left\vert J\left( t\right) \right\vert }\right) &=&\frac{1}{\left\vert
J\left( t\right) \right\vert ^{2}}g\left( J^{\prime }\left( t\right)
,J\left( t\right) \right) \\
&=&\frac{t+O\left( t^{3}\right) }{t^{2}+O\left( t^{4}\right) } \\
&=&\frac{1}{t}+O\left( t\right) ,
\end{eqnarray*}%
and Part 3 holds along $\gamma $ when $V=W.$

For $Y\in \mathcal{V}\oplus \mathcal{H}$ with $Y\perp J\left( t\right) ,$ $%
\left\vert Y\right\vert =1,$ we write 
\begin{equation*}
Y\left( t\right) =\sum\limits_{i=2}^{n}\alpha _{i}E_{i}.
\end{equation*}%
Since $\left\vert Y\right\vert =1$, $\left\vert \alpha _{i}\right\vert \leq
1.$ Combining this with $Y\perp J\left( t\right) $ we get 
\begin{equation*}
\alpha _{2}=O\left( t^{2}\right) .
\end{equation*}

So 
\begin{eqnarray*}
\mathrm{Hess}_{\mathrm{dist}\left( S,\cdot \right) }\left( \frac{J\left(
t\right) }{\left\vert J\left( t\right) \right\vert },Y\right) &=&\frac{1}{%
\left\vert J\left( t\right) \right\vert }g\left( J^{\prime }\left( t\right)
,Y\right) \text{ } \\
&=&\frac{O\left( t^{2}\right) }{t+O\left( t^{3}\right) } \\
&=&O\left( t\right) ,
\end{eqnarray*}%
and Part 3 holds along $\gamma .$

The result follows in general from continuity and the compactness of the
unit normal bundle of $S.$
\end{proof}

The distributions $\mathcal{H}$ and $\mathcal{V}$ are not orthogonal, but
they are asymptotically orthogonal to a high order as $t\rightarrow 0,$ as
we show in the following proposition.

\begin{proposition}
\label{angle tween H and V}There is a constant $C>0$ with the following
property. Let $\gamma $ be a unit speed geodesic in $\Omega ,$ that leaves $%
S $ orthogonally at $\gamma \left( 0\right) $. Let $J_{1}$ and $J_{2}$ be
Jacobi fields along $\gamma $ with 
\begin{eqnarray*}
J_{1}\left( 0\right) &=&0,J_{1}^{\prime }\left( 0\right) \in \nu _{\gamma
\left( 0\right) }\left( S\right) ,\text{ }\left\vert J_{1}^{\prime }\left(
0\right) \right\vert =1\text{ and} \\
J_{2}\left( 0\right) ,J_{2}^{\prime }\left( 0\right) &\in &T_{\gamma \left(
0\right) }S,\text{ }\left\vert J_{2}\left( 0\right) \right\vert =1.
\end{eqnarray*}

Then 
\begin{equation*}
\left\vert g\left( J_{1}\left( \gamma \left( t\right) \right) ,J_{2}\left(
\gamma \left( t\right) \right) \right) \right\vert \leq Ct^{3}.
\end{equation*}
\end{proposition}

\begin{remark}
Since $J_{1}$ satisfies Conditions \ref{Js spanning V} and $J_{2}$ satisfies
Conditions \ref{Js spanning H-II} this tells us that near $S$ the
distributions $\mathcal{H}$ and $\mathcal{V}$ are almost orthogonal.
\end{remark}

\begin{proof}
Just notice that%
\begin{eqnarray*}
g\left( J_{1},J_{2}\right) \left( 0\right) &=&0,\text{ since }J_{1}\left(
0\right) =0, \\
g\left( J_{1},J_{2}\right) ^{\prime }\left( 0\right) &=&g\left(
J_{1}^{\prime },J_{2}\right) \left( 0\right) +g\left( J_{1},J_{2}^{\prime
}\right) \left( 0\right) \\
&=&0,\text{ since }J_{1}\left( 0\right) =0,J_{2}\left( 0\right) \in
T_{\gamma \left( 0\right) }S,\text{ and }J_{1}^{\prime }\left( 0\right) \in
\nu _{\gamma \left( 0\right) }\left( S\right) ,
\end{eqnarray*}

and%
\begin{eqnarray*}
g\left( J_{1},J_{2}\right) ^{\prime \prime }\left( 0\right) &=&g\left(
J_{1}^{\prime \prime },J_{2}\right) \left( 0\right) +2g\left( J_{1}^{\prime
},J_{2}^{\prime }\right) +g\left( J_{1},J_{2}^{\prime \prime }\right) \left(
0\right) \\
&=&0,
\end{eqnarray*}%
since $J_{1}^{\prime }\left( 0\right) \in \nu _{\gamma \left( 0\right)
}\left( S\right) ,J_{2}^{\prime }\left( 0\right) \in T_{\gamma \left(
0\right) }S,$ $J_{1}\left( 0\right) =0,$ and $J_{1}^{\prime \prime }\left(
0\right) =-R\left( J_{1},\dot{\gamma}\right) \dot{\gamma}|_{0}=0.$

This gives us the desired estimate for any particular choice of $\gamma ,$ $%
J_{1}$ and $J_{2}.$ We then get the result with a uniform constant $C$ from
compactness of the unit tangent bundle of $M$ along $S.$
\end{proof}

Proposition \ref{angle tween H and V} allows us to estimate the entire
Hessian of $f$ near $S$ by estimating its values on vectors in $X\cup 
\mathcal{V}\cup \mathcal{H}.$

\begin{lemma}
\label{Hessian Est}On $\Omega \setminus S$

\begin{enumerate}
\item $\mathrm{Hess}_{f}\left( X,X\right) =f^{\prime \prime }.$

\item For $Y\in \mathrm{span}\left\{ X,\mathcal{V},\mathcal{H}\right\} $ and 
$Z\in \mathcal{H}$ and $\delta >0$ as in the Key Lemma%
\begin{equation*}
\left\vert \mathrm{Hess}_{f}\left( Y,Z\right) \right\vert <O\left( \delta
\right) \left\vert Y\right\vert \left\vert Z\right\vert .
\end{equation*}

\item For $Y\in \mathcal{V}$ and $Z\in \mathrm{span}\left\{ X,\mathcal{V}%
\right\} $ 
\begin{equation*}
\left\vert \mathrm{Hess}_{f}\left( Y,Z\right) -\frac{f^{\prime }}{\mathrm{%
dist}\left( S,\cdot \right) }g\left( Y,Z\right) \right\vert \leq \delta
O\left( \mathrm{dist}\left( S,\cdot \right) \right) \left\vert Y\right\vert
\left\vert Z\right\vert
\end{equation*}
\end{enumerate}
\end{lemma}

\begin{proof}
Recall our notational shorthand $f=\rho \circ \mathrm{dist}\left( S,\cdot
\right) ,$ $f^{\prime }\equiv \rho ^{\prime }\circ \mathrm{dist}\left(
S,\cdot \right) ,$ and $f^{\prime \prime }\equiv \rho ^{\prime \prime }\circ 
\mathrm{dist}\left( S,\cdot \right) .$ So $\mathrm{grad}f=f^{\prime }X$ and $%
\mathrm{grad}f^{\prime }=f^{\prime \prime }X.$ Thus%
\begin{eqnarray}
\mathrm{Hess}_{f}\left( Y,Z\right) &=&g\left( \nabla _{Y}f^{\prime
}X,Z\right)  \notag \\
&=&\left( D_{Y}f^{\prime }\right) g\left( X,Z\right) +f^{\prime }g\left(
\nabla _{Y}X,Z\right)  \notag \\
&=&g\left( Y,\mathrm{grad}f^{\prime }\right) g\left( X,Z\right) +f^{\prime }%
\mathrm{Hess}_{\mathrm{dist}\left( S,\cdot \right) }\left( Y,Z\right)  \notag
\\
&=&f^{\prime \prime }g\left( Y,X\right) g\left( X,Z\right) +f^{\prime }%
\mathrm{Hess}_{\mathrm{dist}\left( S,\cdot \right) }\left( Y,Z\right) .
\label{hess f III}
\end{eqnarray}

The lemma follows from Lemma \ref{Peter's lemma}, Equation \ref{hess f III},
and our hypothesis that $\left\vert f^{\prime }\right\vert <\delta $.
\end{proof}

Combining the previous two results gives us the following.

\begin{corollary}
\label{Hess H bar}For $\bar{Y}\in \mathrm{span}\left\{ X,\mathcal{V},%
\overline{\mathcal{H}}\right\} $, $\bar{Z}\in \overline{\mathcal{H}},$ with
footpoint in $\Omega \setminus S$ sufficiently close to $S,$ and for $\delta 
$ as in the Key Lemma%
\begin{equation*}
\left\vert \mathrm{Hess}_{f}\left( \bar{Y},\bar{Z}\right) \right\vert
<O\left( \delta \right) \left\vert \bar{Y}\right\vert \left\vert \bar{Z}%
\right\vert .
\end{equation*}
\end{corollary}

\begin{proof}
Write $\bar{Y}=Y+V$ with $Y\in \mathcal{H}$ and $V\in \mathrm{span}\left\{ X,%
\mathcal{V}\right\} ,$ and write $\bar{Z}=Z+W$ with $Z\in \mathcal{H}$ and $%
W\in \mathrm{span}\left\{ \mathcal{V}\right\} .$ So 
\begin{equation*}
\left\vert \mathrm{Hess}_{f}\left( \bar{Y},\bar{Z}\right) \right\vert \leq
\left\vert \mathrm{Hess}_{f}\left( Y,Z\right) \right\vert +\left\vert 
\mathrm{Hess}_{f}\left( Y,W\right) \right\vert +\left\vert \mathrm{Hess}%
_{f}\left( V,Z\right) \right\vert +\left\vert \mathrm{Hess}_{f}\left(
V,W\right) \right\vert
\end{equation*}%
By Proposition \ref{angle tween H and V}, we have $\left\vert W\right\vert
=\left\vert \bar{Z}\right\vert O\left( \mathrm{dist}\left( S,\cdot \right)
^{2}\right) .$ Combining this with Lemma \ref{Hessian Est} and our
hypothesis that $\Omega \subset B\left( S,\frac{inj\left( S\right) }{2}%
\right) $ gives us 
\begin{equation*}
\left\vert \mathrm{Hess}_{f}\left( \bar{Y},\bar{Z}\right) \right\vert \leq
O\left( \delta \right) \left\vert \bar{Y}\right\vert \left\vert \bar{Z}%
\right\vert .
\end{equation*}
\end{proof}

We are now in a position to prove the Key Lemma.

\begin{proof}[Proof of the Key Lemma]
From Equation \ref{conf curv tensor}, we have for $Y\perp U$ 
\begin{eqnarray}
e^{-2f}\tilde{R}\left( U,Y,Y,U\right) &\geq &R\left( U,Y,Y,U\right)  \notag
\\
&&-g\left( U,U\right) \mathrm{Hess}_{f}\left( Y,Y\right) -g\left( Y,Y\right) 
\mathrm{Hess}_{f}\left( U,U\right) -\left\vert O\left( \delta ^{2}\right)
\right\vert \left\vert Y\right\vert ^{2}\left\vert U\right\vert ^{2}.
\label{gen conf}
\end{eqnarray}

Combining this with Lemma \ref{Hessian Est} we have for $Z\in T\Omega $ and
all $V\in \mathrm{span}\left\{ X,\mathcal{V}\right\} $ with $Z\perp V$%
\begin{eqnarray*}
e^{-2f}\tilde{R}\left( Z,V,V,Z\right) |_{_{B\left( S,\sigma _{1}\right) }}
&\geq &R\left( Z,V,V,Z\right) |_{_{B\left( S,\sigma _{1}\right) }}-f^{\prime
\prime }|_{_{B\left( S,\sigma _{1}\right) }}\left( \left\vert Z\right\vert
^{2}\left\vert V^{\mathrm{span}\left\{ X\right\} }\right\vert
^{2}+\left\vert Z^{\mathrm{span}\left\{ X\right\} }\right\vert
^{2}\left\vert V\right\vert ^{2}\right) \\
&&-\frac{f^{\prime }}{\mathrm{dist}\left( S,\cdot \right) }|_{_{B\left(
S,\sigma _{1}\right) }}\left( \left\vert V^{\mathcal{V}}\right\vert
^{2}\left\vert Z\right\vert ^{2}+\left\vert Z^{\mathcal{V}}\right\vert
^{2}\left\vert V\right\vert ^{2}\right) -\left\vert O\left( \delta \right)
\right\vert \left\vert Z\right\vert ^{2}\left\vert V\right\vert ^{2}
\end{eqnarray*}
provided $\sigma _{1}$ is sufficiently small. Since we assumed that 
\begin{equation*}
R\left( Z,V,V,Z\right) |_{_{B\left( S,\sigma _{1}\right) }}-f^{\prime \prime
}|_{_{B\left( S,\sigma _{1}\right) }}\left\vert Z\right\vert ^{2}\left\vert
V^{\mathrm{span}\left\{ X\right\} }\right\vert ^{2}-\frac{f^{\prime }}{%
\mathrm{dist}\left( S,\cdot \right) }|_{_{B\left( S,\sigma _{1}\right)
}}\left\vert V^{\mathcal{V}}\right\vert ^{2}\left\vert Z\right\vert ^{2}\geq
\left( K+1\right) \left\vert V\right\vert ^{2}\left\vert Z\right\vert ^{2},
\end{equation*}%
$f^{\prime }|_{_{B\left( S,\sigma _{1}\right) }}\leq 0,$ $f^{\prime \prime
}|_{_{B\left( S,\sigma _{1}\right) }}\leq 0,$ and $\left\vert f\right\vert
<\delta $ we obtain 
\begin{equation*}
\widetilde{\mathrm{sec}}\left( V,Z\right) |_{B\left( S,\sigma _{1}\right)
}>K,
\end{equation*}%
provided $\delta $ is sufficiently small.

Now consider, not necessarily distinct, orthonormal vectors $E,Y,Z,U\in 
\mathrm{span}\left\{ X\right\} \cup \mathcal{V}\cup \mathcal{\bar{H}}.$ Then%
\begin{eqnarray}
e^{-2f}\tilde{R}\left( E,Y,Z,U\right) &=&R\left( E,Y,Z,U\right)  \notag \\
&&-g\left( E,U\right) \mathrm{Hess}_{f}\left( Y,Z\right) -g\left( Y,Z\right) 
\mathrm{Hess}_{f}\left( E,U\right)  \notag \\
&&+g\left( E,Z\right) \mathrm{Hess}_{f}\left( Y,U\right) +g\left( Y,U\right) 
\mathrm{Hess}_{f}\left( E,Z\right)  \notag \\
&&\pm O\left( \delta ^{2}\right) \left\vert E\right\vert \left\vert
Y\right\vert \left\vert Z\right\vert \left\vert U\right\vert .
\end{eqnarray}

If we further assume that $R\left( E,Y,Z,U\right) $ does not correspond, up
to a symmetry of the curvature tensor, to the sectional curvature of a plane
containing a vector $V\in \mathrm{span}\left\{ X\right\} \cup \mathcal{V}$,
it then follows from Lemma \ref{Hessian Est} and Corollary \ref{Hess H bar}
that all four Hessian terms are bounded from above by $O\left( \delta
\right) .$ So 
\begin{equation*}
e^{-2f}\tilde{R}\left( E,Y,Z,U\right) =R\left( E,Y,Z,U\right) \pm O\left(
\delta \right) \left\vert E\right\vert \left\vert Y\right\vert \left\vert
Z\right\vert \left\vert U\right\vert .
\end{equation*}%
We then get Inequality \ref{R est} by choosing $\delta $ to be sufficiently
small.

On $M\setminus B\left( S,\sigma _{3}\right) $ Inequality \ref{not foulded
your father} follows from the hypothesis that $f|_{M\setminus B\left(
S,\sigma _{3}\right) }\equiv 0.$ We get Inequality \ref{not foulded your
father} on $B\left( S,\sigma _{3}\right) $ by combining Inequalities \ref{R
est}, and \ref{gen conf} with Lemma \ref{Hessian Est}, Corollary \ref{Hess H
bar}, and the hypothesis that $\left\vert f^{\prime }\right\vert +$ $%
f^{\prime \prime }<2\delta .$
\end{proof}

Now we prove Theorem \ref{conf-submanf}.

\begin{proof}[Proof of Theorem \protect\ref{conf-submanf}]
Given $\varepsilon ,K>0,$ choose $\delta $ and $\sigma _{1}$ as in the Key
Lemma. Let $\sigma _{2},\sigma _{3},$ and $\sigma _{4}$ be such that $\sigma
_{1}<\sigma _{2}<<\sigma _{3}<\sigma _{4}<\min \left\{ \frac{inj\left(
S\right) }{2},\frac{1}{4}\right\} ,$ and let $\rho :\left[ 0,\infty \right)
\longrightarrow \mathbb{R}$ satisfy the following conditions.

\begin{enumerate}
\item All derivatives of $\rho $ of odd order at $0$ are equal to $0.$

\item $K+2>-\rho ^{\prime \prime }\left( t\right) |_{\left[ 0,\sigma _{1}%
\right] }+\min \sec _{g}>K+1.$

\item $\rho ^{\prime \prime }\left( t\right) |_{\left[ 0,\sigma _{2}\right]
}\leq 0,$ $\rho ^{\prime }\left( t\right) \leq 0.$

\item $0\leq \rho ^{\prime \prime }|_{\left( \sigma _{2},\infty \right)
}<\delta .$

\item $\left\vert \rho ^{\prime }\right\vert +\left\vert \rho \right\vert
<\delta .$

\item $\rho |_{\left[ \sigma _{3},\infty \right) }\equiv 0.$
\end{enumerate}

Since $f=\rho \circ \mathrm{dist}\left( S,\cdot \right) ,$ Condition 1 gives
us that our conformal factor $e^{2f}$ is a smooth function on $M.$

The Fundamental Theorem of Calculus and Condition 2 give 
\begin{eqnarray}
-\rho ^{\prime }\left( t\right) |_{\left[ 0,\sigma _{1}\right] } &>&\left(
K+1-\min \sec _{g}\right) t,\text{ so}  \notag \\
-\frac{\rho ^{\prime }\left( t\right) |_{\left[ 0,\sigma _{1}\right] }}{t}%
+\min \sec _{g} &>&\left( K+1\right) .  \label{lambda'}
\end{eqnarray}

For $V\in \mathrm{span}\left\{ X,\mathcal{V}\right\} ,$ write $V=$ $V^{%
\mathrm{span}\left\{ X\right\} }+V^{\mathcal{V}}.$ Then Condition 2 gives 
\begin{equation*}
-\rho ^{\prime \prime }\left( t\right) |_{\left[ 0,\sigma _{1}\right]
}\left\vert V^{\mathrm{span}\left\{ X\right\} }\right\vert ^{2}+\min \sec
_{g}\left\vert V^{\mathrm{span}\left\{ X\right\} }\right\vert ^{2}>\left(
K+1\right) \left\vert V^{\mathrm{span}\left\{ X\right\} }\right\vert ^{2},
\end{equation*}%
and Inequality \ref{lambda'} gives 
\begin{equation*}
-\frac{\rho ^{\prime }\left( t\right) |_{\left[ 0,\sigma _{1}\right] }}{t}%
\left\vert V^{\mathcal{V}}\right\vert ^{2}+\min \sec _{g}\left\vert V^{%
\mathcal{V}}\right\vert ^{2}>\left( K+1\right) \left\vert V^{\mathcal{V}%
}\right\vert ^{2}.
\end{equation*}%
Adding the previous two inequalities we get 
\begin{equation*}
-\rho ^{\prime \prime }\left( t\right) |_{\left[ 0,\sigma _{1}\right]
}\left\vert V^{\mathrm{span}\left\{ X\right\} }\right\vert ^{2}-\frac{\rho
^{\prime }\left( t\right) |_{\left[ 0,\sigma _{1}\right] }}{t}\left\vert V^{%
\mathcal{V}}\right\vert ^{2}+\min \sec _{g}\left\vert V\right\vert
^{2}>\left( K+1\right) \left\vert V\right\vert ^{2}.
\end{equation*}

Let $t=\mathrm{dist}\left( S,\cdot \right) ,$ then $f^{\prime }\equiv \rho
^{\prime }\left( t\right) $ and $f^{\prime \prime }\equiv \rho ^{\prime
\prime }\left( t\right) .$ Making these substitutions, multiplying both
sides by $\left\vert Z\right\vert ^{2}$, and using $R\left( Z,V,V,Z\right)
|_{_{B\left( S,\sigma _{1}\right) }}\geq \min \sec _{g}\left\vert
V\right\vert ^{2}\left\vert Z\right\vert ^{2}$ gives%
\begin{eqnarray*}
&R\left( Z,V,V,Z\right) |_{_{B\left( S,\sigma _{1}\right) }}-f^{\prime
\prime }|_{_{B\left( S,\sigma _{1}\right) }}\left\vert Z\right\vert
^{2}\left\vert V^{\mathrm{span}\left\{ X\right\} }\right\vert ^{2}-\frac{%
f^{\prime }}{\mathrm{dist}\left( S,\cdot \right) }|_{_{B\left( S,\sigma
_{1}\right) }}\left\vert V^{\mathcal{V}}\right\vert ^{2}\left\vert
Z\right\vert ^{2} \\
&\geq \left( K+1\right) \left\vert V\right\vert ^{2}\left\vert Z\right\vert
^{2}.
\end{eqnarray*}

This establishes Inequality \ref{nasty inequality} of the Key Lemma. The
other hypotheses of the Key Lemma follow from the properties of $\rho $
(numbered 3--6, above). We then apply the Key Lemma to obtain the curvature
bounds of Theorem \ref{conf-submanf}. Finally, if $G$ acts isometrically on $%
M$ and $S$ is $G$--invariant, then $\tilde{g}$ is as well, since $f=\rho
\circ \mathrm{dist}\left( S,\cdot \right) .$
\end{proof}

\begin{remark}
\label{c-1 close}Given $\varepsilon $ and $K,$ if the Key Lemma holds for $%
\delta =\delta _{0},$ then it also holds for all $\delta \in \left( 0,\delta
_{0}\right) .$ Since $\left\vert \rho ^{\prime }\right\vert +\left\vert \rho
\right\vert <\delta ,$ and $f=\rho \circ \mathrm{dist}\left( S,\cdot \right)
,$ our conformal factor, $e^{2f}$ can be as close as we please in the $C^{1}$%
--topology to $1.$
\end{remark}

\subsection{Conformal Change Near a Compact Subset of a Non-compact
Submanifold}

Since the strata can be non-compact manifolds, we will need to generalize
Theorem \ref{conf-submanf}. Let $\left( M,g\right) $ be a compact Riemannian 
$n$--manifold. Let $S$ be a smooth submanifold of $\left( M,g\right) $. Let $%
\mathcal{C}_{1}$ be a compact subset of $S$. Let $inj\left( \mathcal{C}%
_{1}\right) $ be the injectivity radius of the normal bundle $\nu \left(
S\right) |_{\mathcal{C}_{1}}.$ Let $\nu _{0}\left( S\right) |_{\mathcal{C}%
_{1}}$ be the image of the zero section of $\nu \left( S\right) |_{\mathcal{C%
}_{1}}\longrightarrow \mathcal{C}_{1}.$ Let 
\begin{equation*}
\Omega \equiv \exp _{S}^{\perp }\left( B\left( \nu _{0}\left( S\right) |_{%
\mathcal{C}_{1}},\frac{inj\left( \mathcal{C}_{1}\right) }{2}\right) \right) ,
\end{equation*}%
and let 
\begin{equation*}
X\oplus \mathcal{V}\oplus \mathcal{H}
\end{equation*}%
be the splitting of $T\Omega $ given in \ref{defn of X}.

\begin{theorem}
\label{non-compact conf}Let $\left( M,g\right) ,$ $S,$ $\mathcal{C}_{1}$, $%
X, $ and $\mathcal{V}$ be as above, and let $\mathcal{C}_{3}$ be any compact
subset of $S$ with $\mathcal{C}_{1}\subset \mathrm{Int}\left( \mathcal{C}%
_{3}\right) .$ For any $\varepsilon ,K>0$ there are numbers $\sigma
_{1},\sigma _{3}$ with $0<\sigma _{1}<\sigma _{3}<\frac{inj\left( \mathcal{C}%
_{3}\right) }{2}$ and a metric $\tilde{g}=e^{2f}g$ with the following
properties.

\begin{enumerate}
\item Setting $\Omega _{1}\equiv \exp _{S}^{\perp }B\left( \nu _{0}(S)|_{%
\mathcal{C}_{1}},\sigma _{1}\right) $ and $\Omega _{3}\equiv \exp
_{S}^{\perp }B\left( \nu _{0}(S)|_{\mathcal{C}_{3}},\sigma _{3}\right) ,$
the metrics $\tilde{g}$ and $g$ coincide on $M\setminus \Omega _{3}.$

\item For all $Z\in T\Omega _{1}$ and all $V\in \mathrm{span}\left\{ X,%
\mathcal{V}\right\} $ 
\begin{equation}
\widetilde{\mathrm{sec}}\left( V,Z\right) |_{\Omega _{1}}>K.
\label{big curv}
\end{equation}

\item If $\left\{ E_{1},\ldots ,E_{n}\right\} $ is a local orthonormal frame
for $\Omega _{3}$ with $X=E_{1}$ and \textrm{span}$\left\{ E_{2},\ldots
,E_{r}\right\} =\mathcal{V}$ for $2\leq r\leq n,$ then 
\begin{equation}
\left\vert \tilde{R}\left( E_{i},E_{j},E_{k},E_{l}\right) -R\left(
E_{i},E_{j},E_{k},E_{l}\right) \right\vert <\varepsilon ,
\label{not too
Badd-II}
\end{equation}%
except if the quadruple corresponds, up to a symmetry of the curvature
tensor, to the sectional curvature of a plane containing a vector $V\in 
\mathrm{span}\left\{ X\right\} \cup \mathcal{V}$.

\item 
\begin{equation}
\widetilde{\mathrm{sec}}\left( V,W\right) >\mathrm{sec}\left( V,W\right)
-\varepsilon  \label{not bad curv}
\end{equation}%
for all $V,W\in TM.$
\end{enumerate}

Moreover, if $G$ acts isometrically on $\left( M,g\right) $ and $S$ and $%
\mathcal{C}_{1}$ are $G$--invariant, then we may choose $\tilde{g}$ to be $G$%
--invariant.
\end{theorem}

\begin{remark}
As was the case for Theorem \ref{conf-submanf}, with appropriate choices of $%
\varepsilon $ and $K,$ $Ric_{\tilde{g}}|_{\Omega _{1}}>1.$
\end{remark}

\begin{proof}[Proof of Theorem \protect\ref{non-compact conf}]
Let $\mathcal{C}_{2}$ and $\mathcal{C}_{4}$ be compact subsets of $S$ with $%
\mathcal{C}_{1}\subset \mathrm{Int}\left( \mathcal{C}_{2}\right) $, $%
\mathcal{C}_{2}\subset \mathrm{Int}\left( \mathcal{C}_{3}\right) ,$ and $%
\mathcal{C}_{3}\subset \mathrm{Int}\left( \mathcal{C}_{4}\right) .$ Let $%
inj\left( \mathcal{C}_{4}\right) $ be the injectivity radius of the normal
bundle $\nu \left( S\right) |_{\mathcal{C}_{4}}.$ Let $\bar{\varphi}%
:S\longrightarrow \left[ 0,1\right] $ be $C^{\infty }$ and satisfy%
\begin{equation*}
\bar{\varphi}=\left\{ 
\begin{array}{cc}
1 & \text{on }\mathcal{C}_{2} \\ 
0 & S\setminus \mathcal{C}_{3}%
\end{array}%
\right. .
\end{equation*}%
Given $\sigma _{4}\in \left( 0,\frac{inj\left( \mathcal{C}_{4}\right) }{2}%
\right) ,$ extend $\bar{\varphi},$ by exponentiation, to a function $\varphi
,$ defined on $\exp _{S}^{\perp }B\left( \nu _{0}\left( S\right) |_{\mathcal{%
C}_{4}},\sigma _{4}\right) $ by setting 
\begin{equation*}
\varphi \left( x\right) =\bar{\varphi}\left( \mathrm{footpoint}\left( \left(
\exp _{S}^{\perp }\right) ^{-1}\left( x\right) \right) \right) .
\end{equation*}%
Our conformal factor is $e^{2f},$ where 
\begin{equation*}
f\left( x\right) \equiv \left\{ 
\begin{array}{ccc}
\left( \rho \circ \mathrm{dist}\left( S,x\right) \right) \cdot \varphi
\left( x\right) &  & \text{for }x\in \exp _{S}^{\perp }B\left( \nu
_{0}\left( S\right) |_{\mathcal{C}_{4}},\sigma _{4}\right) \\ 
0 &  & \text{for }x\in \text{ }M\setminus \exp _{S}^{\perp }B\left( \nu
_{0}\left( S\right) |_{\mathcal{C}_{3}},\sigma _{3}\right)%
\end{array}%
\right.
\end{equation*}%
and $\rho $ is as in the proof of Theorem \ref{conf-submanf}$.$ Since $%
\left( \rho \circ \mathrm{dist}\left( S,x\right) \right) \cdot \varphi
\left( x\right) $ is $0$ on $\left( \exp _{S}^{\perp }B\left( \nu _{0}\left(
S\right) |_{\mathcal{C}_{4}},\sigma _{4}\right) \right) \setminus \left(
\exp _{S}^{\perp }B\left( \nu _{0}\left( S\right) |_{\mathcal{C}_{3}},\sigma
_{3}\right) \right) ,$ $f$ is a well defined $C^{\infty }$ function.

Setting $\tilde{f}\equiv \rho \circ \mathrm{dist}\left( S,\cdot \right) ,$
we have that on $\exp _{S}^{\perp }B\left( \nu _{0}\left( S\right) |_{%
\mathcal{C}_{4}},\sigma _{4}\right) ,$ 
\begin{equation*}
f=\varphi \cdot \tilde{f},
\end{equation*}%
\begin{equation}
\text{grad}\left( f\right) =\varphi \text{grad}\left( \tilde{f}\right) +%
\tilde{f}\text{grad}\left( \varphi \right) .  \label{grad phi f}
\end{equation}%
Since $\left\vert \tilde{f}\right\vert ,\left\vert \mathrm{grad}\left( 
\tilde{f}\right) \right\vert <\delta $ and $\left\vert \varphi \right\vert
\leq 1,$ we see from Equation \ref{grad phi f} that if $\delta $ is
sufficiently small compared to $\left\vert \text{grad}\varphi \right\vert ,$
then 
\begin{equation}
\left\vert \text{grad}\left( f\right) \right\vert <O\left( \delta \right) .
\label{grad f}
\end{equation}%
Since grad$\left( f\right) =\varphi $grad$\left( \tilde{f}\right) +\tilde{f}$%
grad$\left( \varphi \right) ,$ 
\begin{eqnarray*}
\mathrm{Hess}_{f}\left( V,W\right) &=&g\left( \nabla _{V}\left( \varphi 
\text{grad}\left( \tilde{f}\right) +\tilde{f}\text{grad}\left( \varphi
\right) \right) ,W\right) \\
&=&\left( D_{V}\varphi \right) g\left( \text{grad}\left( \tilde{f}\right)
,W\right) +\varphi g\left( \nabla _{V}\left( \text{grad}\left( \tilde{f}%
\right) \right) ,W\right) \\
&&+\left( D_{V}\tilde{f}\right) g\left( \text{grad}\left( \varphi \right)
,W\right) +\tilde{f}g\left( \nabla _{V}\text{grad}\left( \varphi \right)
,W\right) \\
&=&\left( D_{V}\varphi \right) D_{W}\tilde{f}+\varphi \mathrm{Hess}_{\tilde{f%
}}\left( V,W\right) \\
&&+\left( D_{V}\tilde{f}\right) D_{W}\varphi +\tilde{f}\mathrm{Hess}%
_{\varphi }\left( V,W\right) .
\end{eqnarray*}

Using $\left\vert \tilde{f}\right\vert ,\left\vert \mathrm{grad}\left( 
\tilde{f}\right) \right\vert <\delta $ and choosing $\delta $ small compared
to both $\left\vert \text{grad}\varphi \right\vert $ and $\left\vert \mathrm{%
Hess}_{\varphi }\right\vert $ gives us%
\begin{equation}
\mathrm{Hess}_{f}\left( V,W\right) =\varphi \mathrm{Hess}_{\tilde{f}}\left(
V,W\right) +O\left( \delta \right) \left\vert V\right\vert \left\vert
W\right\vert .  \label{Hess f}
\end{equation}

Inequality \ref{grad f} and Equation \ref{Hess f} allow us to argue, as in
the proof of Theorem \ref{conf-submanf}, to obtain the curvature estimates
in \ref{big curv}, \ref{not too Badd-II}, and \ref{not bad curv}.

If $S$ and $\mathcal{C}_{1}$ are $G$--invariant, we take $\mathcal{C}_{2},%
\mathcal{C}_{3},$ and $\mathcal{C}_{4}$ to be metric neighborhoods of $%
\mathcal{C}_{1}$ within $S.$ Let $\bar{\varphi}$ have the form $\bar{\varphi}%
=\psi \circ \mathrm{dist}\left( \mathcal{C}_{1},\cdot \right) $ where $\psi :%
\mathbb{R}\longrightarrow \mathbb{R}.$ Such functions, $\bar{\varphi},$ are $%
G$--invariant and have $G$--invariant smoothings, using the Riemannian
convolution technique of, for example, \cite{GrWu1}, \cite{GrWu2}, \cite%
{GrovShio}. Extending $\bar{\varphi}$ by exponentiation as above then gives
a smooth, $G$--invariant $\varphi ,$ and hence a $G$--invariant $\tilde{g}.$
\end{proof}

\begin{remark}
As was the case for Theorem \ref{conf-submanf}, the conformal factor, $%
e^{2f},$ can be as close to $1$ as we please in the $C^{1}$--topology.
\end{remark}

\subsection{Conformal change in a Neighborhood of the Entire Singular Strata}

The conformal change that we actually use to prove Theorems \ref{main} and %
\ref{alm nonneg thm} is the one obtained from the following theorem.

\begin{theorem}
\label{uber conf}Let $G$ be a compact, connected Lie group acting
isometrically and effectively on a compact Riemannian $n$--manifold $\left(
M,g\right) $ with singular strata, $S_{1},S_{2},\ldots ,S_{p}.$

For any $\varepsilon ,K>0$ there are neighborhoods $\Omega _{1}\subset
\Omega _{3}$ of $S_{1}\cup S_{2}\cup \cdots \cup S_{p},$ and a $G$%
--invariant metric $\tilde{g}=e^{2f}g$ which have the following properties.

\begin{enumerate}
\item For each $S_{i}$ there is a compact subset $\mathcal{C}_{i}\subset $ $%
S_{i}$ and tubular neighborhoods $\Omega _{1}^{i}\subset \Omega _{3}^{i}$ as
in Theorem \ref{non-compact conf} with $\Omega _{1}=\cup \Omega _{1}^{i}$
and $\Omega _{3}=\cup \Omega _{3}^{i}.$ Let $T\Omega _{3}^{i}=\mathrm{span}%
\left\{ X_{i}\right\} \oplus \mathcal{V}_{i}\oplus \mathcal{H}_{i}$ be the
splitting as in \ref{defn of X}.

\item The metrics $\tilde{g}$ and $g$ coincide on $M\setminus \Omega _{3}.$

\item For all $i\in \left\{ 1,\ldots ,p\right\} ,$ all $Z\in T\Omega
_{1}^{i} $ and all $V\in \mathrm{span}\left\{ X_{i}\right\} \oplus \mathcal{V%
}_{i}$ 
\begin{equation}
\widetilde{\mathrm{sec}}\left( V,Z\right) |_{\Omega _{1}^{i}}>K.
\label{V-Curv}
\end{equation}

\item If $\left\{ E_{1},\ldots ,E_{n}\right\} $ is a local orthonormal frame
for $\Omega _{3}^{i}$ with $X=E_{1}$ and \textrm{span}$\left\{ E_{2},\ldots
,E_{r}\right\} =\mathcal{V}_{i}$ for $2\leq r\leq n,$ then%
\begin{equation}
\left\vert \tilde{R}\left( E_{i},E_{j},E_{k},E_{l}\right) -R\left(
E_{i},E_{j},E_{k},E_{l}\right) \right\vert <\varepsilon ,
\label{not
fouled your mother}
\end{equation}%
except if the quadruple corresponds, up to a symmetry of the curvature
tensor, to the sectional curvature of a plane containing a vector $V\in 
\mathrm{span}\left\{ X_{i}\right\} \cup \mathcal{V}_{i}$.

\item 
\begin{equation}
\widetilde{\mathrm{sec}}\left( V,W\right) >\mathrm{sec}\left( V,W\right)
-\varepsilon  \label{not fouled--II}
\end{equation}%
for all $V,W\in TM.$
\end{enumerate}
\end{theorem}

\begin{proof}
Our proof is by induction on Descendant Number, which was defined in the
proof of Proposition \ref{Omega_i Dfn}. A stratum with Descendant Number $1$
contains no stratum other than itself and hence is a compact submanifold. We
apply Theorem \ref{conf-submanf} to obtain a $G$--invariant conformal change
and neighborhoods, $\Omega _{1}^{1}$ and $\Omega _{3}^{1}$, of all the
strata with Descendant Number $1$ on which the Inequalities \ref{V-Curv}, %
\ref{not fouled your mother}, and \ref{not fouled--II} hold.

Now suppose we have such a $G$--invariant conformal change and neighborhoods 
$\Omega _{1}^{l}$ and $\Omega _{3}^{l}$ for all strata whose Descendant
Number is $l.$ For each stratum, $S,$ with Descendant Number $l+1$, we
choose a compact subset $\mathcal{C}$ of $S$ so that $\bar{S}\subset 
\mathcal{C}\cup \Omega _{1}^{l}.$ Applying Theorem \ref{non-compact conf} to
each such $S,$ yields a $G$--invariant conformal change and neighborhoods $%
\Omega _{1}^{l+1}$ and $\Omega _{3}^{l+1}$ of all the strata with Descendant
Number $l+1$ that satisfy Inequalities \ref{V-Curv}, \ref{not fouled your
mother}, and \ref{not fouled--II}.
\end{proof}

\begin{remark}
\label{uber c-1 close}Since our conformal factor comes from repeated
applications of Theorem \ref{non-compact conf}, it can be as close as we
please in the $C^{1}$--topology to $1.$
\end{remark}

\section{Cheeger Deformations\label{cheeg def}}

In the presence of a group of isometries, $G,$ a method for deforming the
metric on a manifold, $M$, of non-negative sectional curvature is given in 
\cite{Cheeg}. It is based on the Gray-O'Neill principle that Riemannian
submersions do not decrease the curvatures of horizontal planes. We briefly
review the basics of this construction here, largely following the
exposition from \cite{PetWilh1}.

Let $G$ be a compact group of isometries of $\left( M,g_{M}\right) ,$ $g_{%
\mathrm{bi}}$ a bi-invariant metric on $G,$ and consider the one parameter
family $l^{2}g_{\mathrm{bi}}+g_{M}$ of metrics on $G\times M.$ Then $G$ acts
on $\left( G\times M,l^{2}g_{\mathrm{bi}}+g_{M}\right) $ via 
\begin{equation}
g(p,m)=(pg^{-1},gm).  \label{skew action}
\end{equation}

Modding out by the action \ref{skew action} we obtain a one parameter family 
$g_{l}$ of metrics on $M\cong \left( G\times M\right) /G.$ As $l\rightarrow
\infty ,$ $\left( M,g_{l}\right) $ converges to $g_{M}$ \cite{PetWilh1}.

The quotient map for the action (\ref{skew action}) is 
\begin{equation*}
q_{G\times M}:(g,m)\mapsto gm.
\end{equation*}%
The vertical space for $q_{G\times M}$ at $(g,m)\in G\times M$ is 
\begin{equation}
\mathcal{V}_{q_{G\times M}}=\{(-k_{G}\left( g\right) ,k_{M}\left( m\right)
)\ |\ k\in \mathfrak{g}\},  \label{cheeg vertt}
\end{equation}%
where we are employing the convention that for $k\in \mathfrak{g,}$ $k_{G}$
is the Killing field on $G$ generated by $k$ and $k_{M}$ is the Killing
field on $M$ generated by $k\mathfrak{.}$

We recall from \cite{Cheeg}, \cite{PetWilh1} that there is a
reparametrization of the tangent space, that we call the Cheeger
reparametrization. We denote it by 
\begin{equation*}
Ch_{l}:TM\rightarrow TM.
\end{equation*}%
It is defined by 
\begin{equation*}
Ch_{l}\left( v\right) =d\left( q_{G\times M}\right) \left( \hat{v}%
_{l}\right) ,
\end{equation*}%
where $\hat{v}_{l}\in TG\times TM$ is the horizontal vector for $q_{G\times
M}:\left( G\times M,l^{2}g_{\mathrm{bi}}+g_{M}\right) \longrightarrow \left(
M,g_{l}\right) $ that maps to $v$ under the projection $d\pi _{2}:$ $T\left(
G\times M\right) \longrightarrow TM.$

Note that every $G$--orbit in $G\times M$ has a point of the form $\left(
e,x\right) .$ At such a point, we let $\kappa _{v}$ be the element of $TG$
so that when $l=1$ 
\begin{equation*}
\hat{v}_{1}=\left( \kappa _{v},v\right) .
\end{equation*}%
Because $\hat{v}_{l}$ is $q_{G\times M}$--horizontal, $\kappa _{v}$ is
orthogonal to the Lie Algebra of the isotropy at $x,$ that is, $\kappa
_{v}\in \mathfrak{m}_{x}.$ For any $l,$ we then have 
\begin{equation*}
\hat{v}_{l}=\left( \frac{\kappa _{v}}{l^{2}},v\right) .
\end{equation*}%
For simplicity we will write $\hat{v}$ for $\hat{v}_{l}.$

Although $\kappa _{v}$ is completely determined by $v,$ $g_{\mathrm{bi}},$ $%
g_{M},$ and the $G$--action, we will not give its explicit formula since it
is somewhat unpleasant. Instead, we develop some key abstract properties in
the following proposition.

\begin{proposition}
\label{bound on kappa}

\begin{enumerate}
\item There is a constant $C_{1}>0$ so that for all unit vectors $V\in TM,$%
\begin{equation*}
\left\vert \kappa _{V}\right\vert _{g_{\mathrm{bi}}}\leq C_{1}.
\end{equation*}

\item For any compact subset $\mathcal{K}\subset M^{\text{reg }}$ there is a
constant $C_{2}>0$ so that for all $x\in \mathcal{K}$ and all unit $V\in
TG\left( x\right) ,$ 
\begin{equation*}
\left\vert \kappa _{V}\right\vert _{g_{\mathrm{bi}}}\geq C_{2}.
\end{equation*}

\item The maps $T_{p}M\longrightarrow T_{\left( e,p\right) }\left( G\times
M\right) ,$ $v\longmapsto \hat{v}_{l}$ and $T_{p}M\longrightarrow T_{p}M,$ $%
v\longmapsto Ch_{l}\left( v\right) $ are linear.
\end{enumerate}
\end{proposition}

\begin{proof}
By definition, $\left( \kappa _{V},V\right) $ is $q_{G\times M}$--horizontal
with respect to $\left( g_{\mathrm{bi}}+g_{M}\right) ,$ so for all $k\in 
\mathfrak{g}$ with $\left\vert k_{G}\right\vert _{g_{\mathrm{bi}}}=1,$ 
\begin{equation*}
0=\left( g_{\mathrm{bi}}+g_{M}\right) \left( \left( \kappa _{V},V\right)
,\left( -k_{G},k_{M}\right) \right) .
\end{equation*}%
So%
\begin{eqnarray*}
\left\vert \kappa _{V}\right\vert _{g_{\mathrm{bi}}} &=&\max \left\{
\left\vert g_{\mathrm{bi}}\left( \kappa _{V},-k_{G}\right) \right\vert \text{
s.t. }k\in \mathfrak{g}\text{ with }\left\vert k_{G}\right\vert _{g_{\mathrm{%
bi}}}=1\right\} \\
&=&\left\vert g_{\mathrm{bi}}\left( \kappa _{V},-\left( k_{\max }\right)
_{G}\right) \right\vert ,\text{ for some }k_{\max }\in \mathfrak{g}\text{
with }\left\vert \left( k_{\max }\right) _{G}\right\vert _{g_{\mathrm{bi}}}=1
\\
&=&\left\vert g_{M}\left( V,\left( k_{\max }\right) _{M}\right) \right\vert 
\text{ } \\
&\leq &\max \left\{ \left\vert k_{M}\right\vert _{g_{M}}\text{ s.t. }k\in 
\mathfrak{g}\text{ with }\left\vert k_{G}\right\vert _{g_{\mathrm{bi}%
}}=1\right\} ,\text{ since }\left\vert V\right\vert _{g_{M}}=1.
\end{eqnarray*}%
By compactness, the right hand side is bounded from above by some constant $%
C_{1}>0$, proving Part 1.

For Part 2, take $x\in \mathcal{K}\subset M^{\text{reg }}$ and choose $k\in 
\mathfrak{m}_{x}$ with $\left\vert k_{G}\right\vert _{g_{\mathrm{bi}}}=1$
and $\frac{k_{M}}{\left\vert k_{M}\right\vert _{g_{M}}}=V,$ then%
\begin{equation*}
0=\left( g_{\mathrm{bi}}+g_{M}\right) \left( \left( \kappa _{V},V\right)
,\left( -k_{G},k_{M}\right) \right) .
\end{equation*}%
So 
\begin{eqnarray*}
\left\vert \kappa _{V}\right\vert _{g_{\mathrm{bi}}} &\geq &\left\vert g_{%
\mathrm{bi}}\left( \kappa _{V},-k_{G}\right) \right\vert \\
&=&\left\vert g_{M}\left( V,k_{M}\right) \right\vert \\
&=&\left\vert k_{M}\right\vert _{g_{M}}.
\end{eqnarray*}%
Since $k\in \mathfrak{m}_{x}$ and $\left\vert k_{G}\right\vert _{g_{\mathrm{%
bi}}}=1,$ $\left\vert k_{M}\left( x\right) \right\vert _{g_{M}}>0$. By
compactness, there is a positive constant $C_{2}$ so that 
\begin{equation*}
\min_{x\in \mathcal{K}}\left\{ \min_{k\in \mathfrak{m}_{x},\left\vert
k_{G}\right\vert _{g_{\mathrm{bi}}}=1}\left\vert k_{M}\left( x\right)
\right\vert _{g_{M}}\right\} >C_{2}>0,
\end{equation*}%
and Part 2 follows.

Part 3 is an immediate consequence of the definitions of $\hat{v}_{l}$ and $%
Ch_{l}\left( v\right) .$
\end{proof}

Next we bound the sectional curvatures of $g_{l}$ from below.

\begin{proposition}
\label{orbital estimate}If $\left\{ V,W\right\} $ is $g_{M}$--orthonormal,
then 
\begin{eqnarray*}
\mathrm{sec}_{g_{l}}\left( Ch_{l}\left( W\right) ,Ch_{l}\left( V\right)
\right) &\geq &\mathrm{sec}_{l^{2}g_{\mathrm{bi}}+g_{M}}\left( \hat{W},\hat{V%
}\right) \\
&\geq &\max \left\{ -1,-\frac{l^{2}}{\left\vert \kappa _{V}\right\vert _{g_{%
\mathrm{bi}}}^{2}},-\frac{l^{2}}{\left\vert \kappa _{W}\right\vert _{g_{%
\mathrm{bi}}}^{2}}\right\} \left\vert \mathrm{sec}_{g_{M}}\left( V,W\right)
\right\vert
\end{eqnarray*}%
Moreover, if $V$ and $W$ are perpendicular to the orbits of $G,$ that is, if 
$V,W\in \left( TG\left( x\right) \right) ^{\perp },$ then 
\begin{equation}
\mathrm{sec}_{g_{l}}\left( Ch_{l}\left( W\right) ,Ch_{l}\left( V\right)
\right) \geq \mathrm{sec}_{g_{M}}\left( V,W\right) ,  \label{horiz curv}
\end{equation}%
and if $\mathrm{sec}_{g_{M}}\left( V,W\right) >0,$ then $\mathrm{sec}%
_{g_{l}}\left( Ch_{l}\left( W\right) ,Ch_{l}\left( V\right) \right) >0.$
\end{proposition}

\begin{proof}
From the Gray-O'Neill Horizontal Curvature Equation we have 
\begin{eqnarray}
\mathrm{curv}_{g_{l}}\left( Ch_{l}\left( V\right) ,Ch_{l}\left( W\right)
\right) &\geq &\mathrm{curv}_{l^{2}g_{\mathrm{bi}}+g_{M}}\left( \hat{W},\hat{%
V}\right)  \notag \\
&=&\mathrm{curv}_{l^{2}g_{\mathrm{bi}}}\left( \frac{\kappa _{V}}{l^{2}},%
\frac{\kappa _{W}}{l^{2}}\right) +\mathrm{curv}_{g_{M}}\left( V,W\right) 
\notag \\
&\geq &\mathrm{curv}_{g_{M}}\left( V,W\right)  \notag \\
&=&\mathrm{sec}_{g_{M}}\left( V,W\right)  \label{cheeg's idea eqn}
\end{eqnarray}

For $U\in \left( TG\left( x\right) \right) ^{\perp }$, $\hat{U}=\left(
0,U\right) .$ So for $U,V\in \left( TG\left( x\right) \right) ^{\perp },$ $%
\left\vert Ch_{l}\left( U\wedge V\right) \right\vert _{g_{l}}=\left\vert
U\wedge V\right\vert _{g_{M}}$ $,$ and Inequality \ref{horiz curv} follows
from Equation \ref{cheeg's idea eqn}. It also follows that $\mathrm{sec}%
_{g_{l}}\left( Ch_{l}\left( W\right) ,Ch_{l}\left( V\right) \right) >0,$ if $%
\mathrm{sec}_{g_{M}}\left( V,W\right) >0.$

In general we have 
\begin{eqnarray*}
\left\vert Ch_{l}\left( V\right) \right\vert _{g_{l}}^{2} &=&\left\vert 
\frac{\kappa _{V}}{l^{2}}\right\vert _{l^{2}g_{\mathrm{bi}}}^{2}+\left\vert
V\right\vert _{g_{M}}^{2} \\
&=&\frac{1}{l^{2}}\left\vert \kappa _{V}\right\vert _{g_{\mathrm{bi}}}^{2}+1
\\
\left\vert Ch_{l}\left( W\right) \right\vert _{g_{l}}^{2} &=&\frac{1}{l^{2}}%
\left\vert \kappa _{W}\right\vert _{g_{\mathrm{bi}}}^{2}+1,\text{ and} \\
g_{l}\left( Ch_{l}\left( V\right) ,Ch_{l}\left( W\right) \right) ^{2}
&=&l^{4}g_{\mathrm{bi}}\left( \frac{\kappa _{V}}{l^{2}},\frac{\kappa _{W}}{%
l^{2}}\right) ^{2} \\
&=&\frac{1}{l^{4}}g_{\mathrm{bi}}\left( \kappa _{V},\kappa _{W}\right) ^{2}
\end{eqnarray*}%
So%
\begin{eqnarray*}
\left\vert Ch_{l}\left( V\right) \wedge Ch_{l}\left( W\right) \right\vert
_{g_{l}}^{2} &=&\left( \frac{1}{l^{2}}\left\vert \kappa _{V}\right\vert _{g_{%
\mathrm{bi}}}^{2}+1\right) \left( \frac{1}{l^{2}}\left\vert \kappa
_{W}\right\vert _{g_{\mathrm{bi}}}^{2}+1\right) \\
&&-\frac{1}{l^{4}}g_{\mathrm{bi}}\left( \kappa _{V},\kappa _{W}\right) ^{2}
\\
&=&\frac{1}{l^{4}}\left\vert \kappa _{V}\wedge \kappa _{W}\right\vert _{g_{%
\mathrm{bi}}}^{2}+\frac{1}{l^{2}}\left\vert \kappa _{V}\right\vert _{g_{%
\mathrm{bi}}}^{2}+\frac{1}{l^{2}}\left\vert \kappa _{W}\right\vert _{g_{%
\mathrm{bi}}}^{2}+1 \\
&\geq &\frac{1}{l^{2}}\left\vert \kappa _{V}\right\vert _{g_{\mathrm{bi}%
}}^{2}+\frac{1}{l^{2}}\left\vert \kappa _{W}\right\vert _{g_{\mathrm{bi}%
}}^{2}+1 \\
&\geq &\max \left\{ \frac{1}{l^{2}}\left\vert \kappa _{V}\right\vert _{g_{%
\mathrm{bi}}}^{2},\frac{1}{l^{2}}\left\vert \kappa _{W}\right\vert _{g_{%
\mathrm{bi}}}^{2},1\right\}
\end{eqnarray*}%
\begin{eqnarray*}
-\left\vert Ch_{l}\left( V\right) \wedge Ch_{l}\left( W\right) \right\vert
_{g_{l}}^{2} &\leq &-\max \left\{ \frac{1}{l^{2}}\left\vert \kappa
_{V}\right\vert _{g_{\mathrm{bi}}}^{2},\frac{1}{l^{2}}\left\vert \kappa
_{W}\right\vert _{g_{\mathrm{bi}}}^{2},1\right\} \\
&=&\mathrm{min}\left\{ -\frac{1}{l^{2}}\left\vert \kappa _{V}\right\vert
_{g_{\mathrm{bi}}}^{2},-\frac{1}{l^{2}}\left\vert \kappa _{W}\right\vert
_{g_{\mathrm{bi}}}^{2},-1\right\}
\end{eqnarray*}%
and 
\begin{eqnarray*}
-\frac{1}{\left\vert Ch_{l}\left( V\right) \wedge Ch_{l}\left( W\right)
\right\vert _{g_{l}}^{2}} &\geq &\frac{1}{\mathrm{min}\left\{ -\frac{1}{l^{2}%
}\left\vert \kappa _{V}\right\vert _{g_{\mathrm{bi}}}^{2},-\frac{1}{l^{2}}%
\left\vert \kappa _{W}\right\vert _{g_{\mathrm{bi}}}^{2},-1\right\} } \\
&=&\max \left\{ -\frac{l^{2}}{\left\vert \kappa _{V}\right\vert _{g_{\mathrm{%
bi}}}^{2}},-\frac{l^{2}}{\left\vert \kappa _{W}\right\vert _{g_{\mathrm{bi}%
}}^{2}},-1\right\} .
\end{eqnarray*}%
Combining this with Equation \ref{cheeg's idea eqn}%
\begin{eqnarray*}
\mathrm{sec}_{g_{l}}\left( Ch_{l}\left( V\right) ,Ch_{l}\left( W\right)
\right) &\geq &-\frac{1}{\left\vert Ch_{l}\left( V\right) \wedge
Ch_{l}\left( W\right) \right\vert _{g_{l}}^{2}}\left\vert \mathrm{sec}%
_{g_{M}}\left( V,W\right) \right\vert \\
&\geq &\max \left\{ -\frac{l^{2}}{\left\vert \kappa _{V}\right\vert _{g_{%
\mathrm{bi}}}^{2}},-\frac{l^{2}}{\left\vert \kappa _{W}\right\vert _{g_{%
\mathrm{bi}}}^{2}},-1\right\} \left\vert \mathrm{sec}_{g_{M}}\left(
V,W\right) \right\vert ,
\end{eqnarray*}%
as desired.
\end{proof}

\begin{remark}
In particular, Proposition \ref{orbital estimate} shows that the family $%
\left\{ \left( M,g_{l}\right) \right\} _{l>0}$ has a uniform lower curvature
bound. Since $\left\{ \left( M,g_{l}\right) \right\} _{l>0}$ converges to $%
M/G$ in the Gromov--Hausdorff topology, this provides a simple proof that
any $G$--manifold collapses with a lower curvature bound to $M/G,$ as
remarked in Example 1.2(c) of \cite{Yam}.
\end{remark}

\subsection{The A--Tensor for the Cheeger Deformation on the Regular Set}

In this subsection, we show that the curvature of a horizontal plane for $%
\pi ^{\text{reg}}:\left( M^{\text{reg}},g_{l}\right) \longrightarrow M^{%
\text{reg}}/G$ converges to the curvature of its projection in $M^{\text{reg}%
}/G$ as $l\rightarrow 0.$

\begin{proposition}
\label{Cheeg A--tensor}Let $A^{Ch}$ denote the $A$--tensor of the Cheeger
submersion $q_{G\times M}:G\times M\longrightarrow M,$ and let $A^{\text{reg}%
}$ denote the $A$--tensor of the Riemannian submersion 
\begin{equation*}
\pi ^{\text{reg}}:M^{\text{reg}}\longrightarrow M^{\text{reg}}/G.
\end{equation*}

Given any compact subset $\mathcal{K}\subset M^{\text{reg}}$, $x\in \mathcal{%
K}$ and any unit vectors $Z_{1},Z_{2}\in T_{x}G\left( x\right) ^{\perp },$ 
\begin{equation*}
\left\vert A_{\hat{Z}_{1}}^{Ch}\hat{Z}_{2}\right\vert _{l^{2}g_{\mathrm{bi}%
}+g_{M}}\longrightarrow \left\vert A_{{Z}_{1}}^{^{\text{reg}}}{Z}%
_{2}\right\vert _{g_{M}}\text{ uniformly on }\mathcal{K}\text{ as }%
l\rightarrow 0.
\end{equation*}
\end{proposition}

\begin{remark}
\label{product field}Let $\left( -k_{G},k_{M}\right) \in T\left( G\times
M\right) ,$ as in \ref{cheeg vertt}, be vertical for $q_{G\times M}.$ Notice
that 
\begin{equation*}
\left( -k_{G},k_{M}\right) =\left( -k_{G},0\right) +\left( 0,k_{M}\right)
\end{equation*}%
is the sum of a vector field $\left( -k_{G},0\right) $ that only depends on
the $G$--coordinate of the foot point and a vector field $\left(
0,k_{M}\right) $ that only depends on the $M$--coordinate of the foot point.
In the proof below, we exploit the fact that $k_{G}$ does not depend on the $%
M$--coordinate.
\end{remark}

\begin{proof}
Notice that $Z_{1},Z_{2}\in T_{x}G\left( x\right) ^{\perp }$ implies that $%
\hat{Z}_{1}=\left( 0,Z_{1}\right) $ and $\hat{Z}_{2}=\left( 0,Z_{2}\right) $
for all $l.$ So 
\begin{eqnarray*}
\left( l^{2}g_{\mathrm{bi}}+g_{M}\right) \left( A_{\hat{Z}_{1}}^{Ch}\hat{Z}%
_{2},\left( -k_{G},k_{M}\right) \right) &=&-\left( l^{2}g_{\mathrm{bi}%
}+g_{M}\right) \left( \left( 0,Z_{2}\right) ,A_{\left( 0,Z_{1}\right)
}^{Ch}\left( -k_{G},k_{M}\right) \right) \\
&=&-\left( l^{2}g_{\mathrm{bi}}+g_{M}\right) \left( \left( 0,Z_{2}\right)
,\left( 0,\nabla _{Z_{1}}^{g_{M}}k_{M}\right) \right) ,\text{ by Remark \ref%
{product field}} \\
&=&-g_{M}\left( Z_{2},A_{Z_{1}}^{\text{reg}}k_{M}\right) \\
&=&g_{M}\left( A_{Z_{1}}^{\text{reg}}Z_{2},k_{M}\right)
\end{eqnarray*}%
For $k\in \mathfrak{g}_{x}$, $k_{M}\left( x\right) =0,$ so $\left( l^{2}g_{%
\mathrm{bi}}+g_{M}\right) \left( A_{\hat{Z}_{1}}^{Ch}\hat{Z}_{2},\left(
-k_{G},k_{M}\right) \right) =g_{M}\left( A_{Z_{1}}^{\text{reg}%
}Z_{2},k_{M}\right) =0.$

So we may assume that $k\in \mathfrak{m}_{x}.$ In this case, divide both
sides of the previous display by $\left\vert \left( -k_{G},k_{M}\right)
\right\vert _{l^{2}g_{\mathrm{bi}}+g_{M}}$ and observe that for $k\in 
\mathfrak{m}_{x}$ with $\left\vert k_{G}\right\vert _{g_{bi}}=1,$ $%
\left\vert k_{M}\right\vert _{g_{M}}$ is uniformly bounded from below on $%
\mathcal{K},$ so 
\begin{equation*}
\frac{1}{\left\vert \left( -k_{G},k_{M}\right) \right\vert _{l^{2}g_{\mathrm{%
bi}}+g_{M}}}=\frac{1}{\sqrt{l^{2}+\left\vert k_{M}\right\vert _{g_{M}}^{2}}}%
\longrightarrow \frac{1}{\left\vert k_{M}\right\vert _{g_{M}}}
\end{equation*}%
uniformly on $\mathcal{K}$ as $l\rightarrow 0,$ and the result follows.
\end{proof}

\section{Infinitesimal Geometry Near the Singular Orbits\label{Infinitisimal}%
}

This section culminates with the proof of Lemma \ref{H---A--tensor}, which
shows that for planes that are orthogonal to both the orbits of $G$ and to
the fibers of the metric projection to the singular strata, the sectional
curvatures of $g_{l}$ converge to the curvatures of their images under $d\pi
^{\mathrm{reg}}$ as $l$ approaches $0.$ The other crucial result in this
section is Corollary \ref{kappa not 0}. It gives a lower bound for $%
\left\vert \kappa _{v}\right\vert $ for certain vectors $v$ whose footpoints
are near the singular strata.

Recall from Proposition \ref{Omega_i Dfn} that the singular strata have a
neighborhood $\Omega =\cup \Omega ^{i}$ so that for each $i,$ we have a
splitting of $T\left( \Omega ^{i}\setminus S_{i}\right) ,$ $\mathcal{H}%
^{i}\oplus X^{i}\oplus \mathcal{V}^{i}.$ Recall that $\mathcal{H}^{i}$ need
not be orthogonal to $span\left\{ X^{i},\mathcal{V}^{i}\right\} .$ So we
define $\mathcal{\bar{H}}^{i}$ to be the distribution that is orthogonal to $%
span\left\{ X^{i},\mathcal{V}^{i}\right\} .$ We let $Pr_{i}:B\left( \mathcal{%
C}_{i},\frac{inj\left( \mathcal{C}_{i}\right) }{2}\right) \longrightarrow 
\mathcal{C}_{i}$ be 
\begin{equation*}
Pr_{i}\left( x\right) =\mathrm{footpt}\left( \left( \exp _{S_{i}}^{\perp
}\right) ^{-1}\left( x\right) \right) .
\end{equation*}%
That is, $Pr_{i}$ is the metric projection map onto $\mathcal{C}_{i}\subset
S_{i}.$

\begin{proposition}
\label{orphan}For $x\in \exp _{S_{i}}^{\perp }B\left( \nu _{0}(S_{i})|_{%
\mathcal{C}_{i}},\frac{inj\left( \mathcal{C}_{i}\right) }{2}\right) ,$ 
\begin{equation}
\mathcal{V}_{x}^{i}\supset T_{x}G_{Pr_{i}\left( x\right) }\left( x\right) .%
\text{ \label{orphan contain}}
\end{equation}%
In fact, 
\begin{equation*}
\mathcal{V}_{x}^{i}\cap T_{x}G\left( x\right) =T_{x}G_{Pr_{i}\left( x\right)
}\left( x\right) .
\end{equation*}
\end{proposition}

\begin{proof}
Let $\gamma $ be the minimal geodesic from $Pr_{i}\left( x\right) $ to $x.$
Then $G_{Pr_{i}\left( x\right) }\left( \gamma \right) $ is a family of
minimal geodesics emanating from $Pr_{i}\left( x\right) ,$ normal to $S_{i}.$
In particular, 
\begin{equation}
G_{Pr_{i}\left( x\right) }\left( \dot{\gamma}\left( 0\right) \right) \subset
\nu _{Pr_{i}\left( x\right) }\left( S_{i}\right) .  \label{isotropy act}
\end{equation}

Since $\left\{ \tilde{X}^{i},\mathcal{\tilde{V}}^{i}\right\} $ spans the
vertical distribution of $\nu \left( S\right) \longrightarrow S,$ and 
\begin{eqnarray*}
\mathcal{V}^{i} &\equiv &d\exp _{S}^{\perp }\left( \mathcal{\tilde{V}}%
^{i}\right) , \\
X^{i} &=&d\exp _{S}^{\perp }\left( \tilde{X}^{i}\right) ,
\end{eqnarray*}%
exponentiating \ref{isotropy act} then gives 
\begin{equation*}
TG_{Pr_{i}\left( x\right) }\left( \gamma \right) \subset \mathrm{span}%
\left\{ X^{i},\mathcal{V}^{i}\right\} .
\end{equation*}%
Since $X^{i}\left( x\right) \perp T_{x}G_{Pr_{i}\left( x\right) }\left(
x\right) $ and $X^{i}\perp \mathcal{V}^{i},$ we get that for any small fixed
number $t_{0}>0,$ $T_{\gamma \left( t_{0}\right) }G_{Pr_{i}\left( x\right)
}\left( \gamma \left( t_{0}\right) \right) \subset span\left\{ \mathcal{V}%
_{\gamma \left( t_{0}\right) }^{i}\right\} $, proving \ref{orphan contain}.

On the other hand, if $k\in \mathfrak{g}\setminus \mathfrak{g}_{Pr_{i}\left(
x\right) },$ then%
\begin{equation*}
\frac{\partial }{\partial s}\left. \exp _{G}\left( sk\right) \cdot \gamma
\left( 0\right) \right\vert _{s=0}\neq 0,
\end{equation*}%
so from Part 1 of Proposition \ref{Jacobi Rep} we have for any small fixed
number $t_{0}>0$ 
\begin{equation*}
\frac{\partial }{\partial s}\left. \exp _{G}\left( sk\right) \cdot \gamma
\left( t_{0}\right) \right\vert _{s=0}\notin \mathcal{V}_{x}^{i}.
\end{equation*}%
So%
\begin{equation*}
\mathcal{V}_{x}^{i}\cap T_{x}G\left( x\right) =T_{x}G_{Pr_{i}\left( x\right)
}\left( x\right) .
\end{equation*}
\end{proof}

\begin{proposition}
For $i\in \left\{ 1,2,\ldots ,p\right\} $, let $U$ be a neighborhood of the
vectors 
\begin{equation*}
\cup _{x\in \Omega ^{i}\setminus \mathcal{C}_{i}}\mathrm{span}\left\{ 
\overline{\mathcal{H}}_{x}^{i}\cap T_{x}G\left( x\right) ^{\perp },X_{x}^{i},%
\mathcal{V}_{x}^{i}\right\} \bigcup \cup _{x\in \mathcal{C}_{i}}\mathrm{span}%
\left\{ T_{x}\left( S_{i}\right) \cap T_{x}G\left( x\right) ^{\perp },\nu
_{x}\left( S_{i}\right) \right\} .
\end{equation*}%
There is a constant $C>0$ so that for all $x\in \Omega ^{i}$ that are close
enough to $\mathcal{C}_{i}$ and all $v\in \left( T\Omega ^{i}\right)
\setminus U,$ there is a $k\in \mathfrak{g}$ with $\left\vert k\right\vert
_{g_{\mathrm{bi}}}=1$ so that 
\begin{equation*}
g_{M}\left( v,k_{M}\right) \geq C\left\vert v\right\vert _{g_{M}}.
\end{equation*}
\end{proposition}

\begin{remark}
It follows from Part 4 of Proposition \ref{Jacobi Rep} that along any
geodesic $\gamma $ that leaves $S_{i}$ orthogonally at $\gamma \left(
0\right) ,$ 
\begin{eqnarray}
&&\lim_{t\rightarrow 0}\mathrm{span}\left\{ \overline{\mathcal{H}}_{\gamma
\left( t\right) }^{i}\cap T_{\gamma \left( t\right) }G\left( \gamma \left(
t\right) \right) ^{\perp },X_{\gamma \left( t\right) }^{i},\mathcal{V}%
_{\gamma \left( t\right) }^{i}\right\}  \notag \\
&=&\mathrm{span}\left\{ T_{\gamma \left( 0\right) }\left( S_{i}\right) \cap
T_{\gamma \left( 0\right) }G\left( \gamma \left( 0\right) \right) ^{\perp
},\nu _{\gamma \left( 0\right) }\left( S_{i}\right) \right\} .
\label{cove of
H + V +X}
\end{eqnarray}
\end{remark}

\begin{proof}
Since $S_{i}$ is $G$--invariant, $T_{x}\left( S_{i}\right) =\left\{
T_{x}\left( S_{i}\right) \cap T_{x}G\left( x\right) ^{\perp }\right\} \oplus
T_{x}G\left( x\right) .$ So for $x\in \mathcal{C}_{i}$ we have the
orthogonal splitting 
\begin{equation*}
T_{x}M=\left\{ T_{x}\left( S_{i}\right) \cap T_{x}G\left( x\right) ^{\perp
}\right\} \oplus \nu _{x}\left( S_{i}\right) \oplus T_{x}G\left( x\right) .
\end{equation*}%
So if $v\in T_{x}M$ is not in the span of the first two summands, its
projection to $TG\left( x\right) $ is nonzero.

Combining this with Equation \ref{cove of H + V +X} and continuity, we have
that for $x\in \Omega ^{i}$, close enough to $\mathcal{C}_{i},$ if 
\begin{equation*}
v\notin \left[ \overline{\mathcal{H}}_{x}^{i}\cap T_{x}G\left( x\right)
^{\perp }\right] \oplus X_{x}^{i}\oplus \mathcal{V}_{x}^{i},
\end{equation*}%
then its projection to $TG\left( x\right) $ is nonzero.

The result then follows from compactness of the unit vectors in $TM\setminus
U.$
\end{proof}

\begin{corollary}
\label{kappa not 0}For $i\in \left\{ 1,2,\ldots ,p\right\} $, let $U$ be a
neighborhood of the vectors 
\begin{equation*}
\cup _{x\in \Omega ^{i}}\mathrm{span}\left\{ \overline{\mathcal{H}}%
_{x}^{i}\cap T_{x}G\left( x\right) ^{\perp },X_{x}^{i},\mathcal{V}%
_{x}^{i}\right\} \bigcup \cup _{x\in \mathcal{C}_{i}}\mathrm{span}\left\{
T_{x}\left( S_{i}\right) \cap T_{x}G\left( x\right) ^{\perp },\nu _{x}\left(
S_{i}\right) \right\} .
\end{equation*}%
There is a constant $C>0$ so that for all $x\in \Omega ^{i},$ and all $v\in
\left( T_{x}\Omega ^{i}\right) \setminus U,$ 
\begin{equation*}
\left\vert \kappa _{v}\right\vert _{g_{\mathrm{bi}}}\geq C\left\vert
v\right\vert _{g_{M}}.
\end{equation*}
\end{corollary}

\begin{proof}
By definition of $\kappa _{v}$ we have 
\begin{equation*}
0=\left( g_{\mathrm{bi}}+g_{M}\right) \left( \left( \kappa _{v},v\right)
,\left( -k_{G},k_{M}\right) \right) =-g_{\mathrm{bi}}\left( \kappa
_{v},k_{G}\right) +g_{M}\left( v,k_{M}\right)
\end{equation*}%
for all $k\in \mathfrak{g}.$ From the previous proposition we have a
constant $C>0$ and a $k\in \mathfrak{g}$ with $\left\vert k\right\vert _{g_{%
\mathrm{bi}}}=1$ so that 
\begin{equation*}
g_{M}\left( v,k_{M}\right) \geq C\left\vert v\right\vert _{g_{M}}.
\end{equation*}

The result follows by combining the previous two displays.
\end{proof}

\subsection{The A--Tensor Near the Singular Orbits}

In this subsection we prove Lemma \ref{H---A--tensor}, which shows that for
planes that are orthogonal to both the orbits of $G$ and to the fibers of
the metric projection to the singular strata, the sectional curvatures of $%
g_{l}$ converge to the curvatures of their images under $d\pi ^{\mathrm{reg}%
} $ as $l$ approaches $0.$

\begin{lemma}
\label{nabla k}For each $i\in \left\{ 1,2,\ldots ,p\right\} $, there is a
constant $C_{i}>0$ so that the following hold.

\begin{enumerate}
\item If our foot point, $x,$ is in $\Omega ^{i}$, $k\in \mathfrak{m}%
_{Pr_{i}\left( x\right) },$ and $Y,Z\in \mathrm{span}\left\{ TG\left(
x\right) ^{\perp }\cap \overline{\mathcal{H}}_{x}^{i}\right\} ,$ 
\begin{equation}
\left\vert g\left( A_{Z}^{\mathrm{reg}}Y,k_{M}\right) \right\vert \leq
C_{i}\left\vert k_{M}\right\vert \left\vert Y\right\vert \left\vert
Z\right\vert .
\end{equation}

\item If our foot point, $x,$ is in $\Omega ^{i}$, $k\in \mathfrak{g}%
_{Pr_{i}\left( x\right) },$ $Y\in \mathrm{span}\left\{ TG\left( x\right)
^{\perp }\cap \overline{\mathcal{H}}_{x}^{i}\right\} ,$ and $Z\in \mathrm{%
span}\left\{ X^{i},\text{ }TG\left( x\right) ^{\perp }\cap \overline{%
\mathcal{H}}_{x}^{i}\right\} ,$%
\begin{equation*}
\left\vert g\left( A_{Z}^{\mathrm{reg}}Y,k_{M}\right) \right\vert \leq
C_{i}\left\vert Z\right\vert \left\vert Y\right\vert \left\vert
k_{M}\right\vert \mathrm{dist}\left( \mathcal{C}_{i},x\right) .
\end{equation*}
\end{enumerate}
\end{lemma}

\begin{proof}
First we prove both inequalities for the special case of a fixed $k\in 
\mathfrak{g}$ and a fixed geodesic $\gamma :\left[ -l,l\right]
\longrightarrow \Omega ^{i}$ with $\gamma \left( 0\right) \in \mathcal{C}%
_{i},$ $\gamma ^{\prime }\left( 0\right) \in \nu _{\gamma \left( 0\right)
}S_{i}.$

If $k\in \mathfrak{m}_{\gamma \left( 0\right) }\setminus \left\{ 0\right\} ,$
then 
\begin{equation*}
t\longmapsto \left. \frac{\left\vert \nabla _{\centerdot }\,k_{M}\right\vert 
}{\left\vert k_{M}\right\vert }\right\vert _{\gamma \left( t\right) }
\end{equation*}%
is continuous on $\left[ -l,l\right] $ and hence has a maximum, proving Part
1 for a fixed $k\in \mathfrak{g}$ and a fixed geodesic $\gamma .$

For Part 2, we first consider the case when both $Y,Z$ are in $T_{\gamma
\left( t\right) }G\left( \gamma \left( t\right) \right) ^{\perp }\cap 
\overline{\mathcal{H}}_{\gamma \left( t\right) }^{i}.$ Since $\overline{%
\mathcal{H}}^{i}$ is the orthogonal complement of $span\left\{ X^{i},%
\mathcal{V}^{i}\right\} ,$ it follows from Part 4 of Lemma \ref{Jacobi Rep}
that $\cup _{t\in \left[ -l,l\right] \setminus \left\{ 0\right\} }\overline{%
\mathcal{H}}_{\gamma \left( t\right) }^{i}$ has an extension to a smooth
distribution along $\gamma .$ Let $\overline{\mathcal{H}}_{\gamma \left(
0\right) }^{i}$ be the vectors at $\gamma \left( 0\right) $ that are in this
distribution. From Part 4 of Lemma \ref{Jacobi Rep}, it follows that 
\begin{equation*}
\overline{\mathcal{H}}_{\gamma \left( 0\right) }^{i}=T_{\gamma \left(
0\right) }S_{i}.
\end{equation*}

By the Slice Theorem we have 
\begin{equation*}
T_{\gamma \left( 0\right) }G\left( \gamma \left( 0\right) \right) ^{\perp
}\cap T_{\gamma \left( 0\right) }S_{i}\subset T_{\gamma \left( 0\right) }%
\mathrm{Fix}\left( M;G_{\gamma \left( 0\right) }\right) ,
\end{equation*}%
where $\mathrm{Fix}\left( M;G_{\gamma \left( 0\right) }\right) $ is the
fixed point set of the $G_{\gamma \left( 0\right) }$-action on $M.$ On the
other hand, $G_{\gamma \left( 0\right) }$ acts on $\nu _{\gamma \left(
0\right) }S_{i}$ without fixed points, so 
\begin{equation*}
T_{\gamma \left( 0\right) }\mathrm{Fix}\left( M;G_{\gamma \left( 0\right)
}\right) \subset T_{\gamma \left( 0\right) }S_{i}.
\end{equation*}%
Combining the previous three displays gives 
\begin{equation*}
T_{\gamma \left( 0\right) }G\left( \gamma \left( 0\right) \right) ^{\perp
}\cap \overline{\mathcal{H}}_{\gamma \left( 0\right) }^{i}\subset T\mathrm{%
Fix}\left( M;G_{\gamma \left( 0\right) }\right) \subset TS_{i}.
\end{equation*}

As components of fixed point sets are totally geodesic, it follows that for
\linebreak $Y,Z\in T_{\gamma \left( 0\right) }G\left( \gamma \left( 0\right)
\right) ^{\perp }\cap \overline{\mathcal{H}}_{\gamma \left( 0\right) }^{i},$ 
$\nabla _{Z}Y\in T\mathrm{Fix}\left( M;G_{\gamma \left( 0\right) }\right)
\subset TS_{i}.$ In particular, for any vector $W$ normal to $S_{i},$ 
\begin{equation}
g\left( \nabla _{Z}Y,W\right) =0.  \label{FIXED!!}
\end{equation}

Combining Part 4 of Lemma \ref{Jacobi Rep}, Proposition \ref{orphan}, and
Equation \ref{FIXED!!} yields 
\begin{equation*}
\left\vert g\left( A_{Z}^{\mathrm{reg}}Y,\frac{k_{M}}{\left\vert
k_{M}\right\vert }\right) \right\vert =\left\vert g\left( \nabla _{Z}Y,\frac{%
k_{M}}{\left\vert k_{M}\right\vert }\right) \right\vert \leq C\cdot \mathrm{%
dist}\left( \mathcal{C}_{i},\gamma \left( t\right) \right)
\end{equation*}%
for unit $Y,Z\in \mathrm{span}\left\{ TG\left( \gamma \left( t\right)
\right) ^{\perp }\cap \overline{\mathcal{H}}_{\gamma \left( t\right)
}^{i}\right\} $, $k\in \mathfrak{g}_{Pr_{i}\left( \gamma \left( t\right)
\right) }$ and some $C>0.$

For the case when $Z=X$ and $Y\in \mathrm{span}\left\{ TG\left( \gamma
\left( t\right) \right) ^{\perp }\cap \overline{\mathcal{H}}_{\gamma \left(
t\right) }^{i}\right\} $ is unit we write\linebreak\ $Y=Y^{\mathcal{H}}+Y^{%
\mathcal{V}}.$ For $k\in \mathfrak{g}_{Pr_{i}\left( \gamma \left( t\right)
\right) }$ by Proposition \ref{orphan} we then have%
\begin{eqnarray*}
0 &=&g\left( Y,\frac{k_{M}}{\left\vert k_{M}\right\vert }\right) \\
&=&g\left( Y^{\mathcal{H}},\frac{k_{M}}{\left\vert k_{M}\right\vert }\right)
+g\left( Y^{\mathcal{V}},\frac{k_{M}}{\left\vert k_{M}\right\vert }\right) .
\end{eqnarray*}%
By Propositions \ref{angle tween H and V} and \ref{orphan}, $\left\vert
g\left( Y^{\mathcal{H}},\frac{k_{M}}{\left\vert k_{M}\right\vert }\right)
\right\vert \leq C\left( t^{2}\right) .$ It follows that $\left\vert g\left(
Y^{\mathcal{V}},\frac{k_{M}}{\left\vert k_{M}\right\vert }\right)
\right\vert \leq C\left( t^{2}\right) .$

From Part 3 of Lemma \ref{Peter's lemma} we then conclude that 
\begin{eqnarray*}
\left\vert g\left( A_{Y}^{\mathrm{reg}}X,\frac{k_{M}}{\left\vert
k_{M}\right\vert }\right) \right\vert &\leq &\left\vert g\left( \nabla _{Y^{%
\mathcal{H}}}X,\frac{k_{M}}{\left\vert k_{M}\right\vert }\right) \right\vert
+\left\vert g\left( \nabla _{Y^{\mathcal{V}}}X,\frac{k_{M}}{\left\vert
k_{M}\right\vert }\right) \right\vert \\
&\leq &C\mathrm{dist}\left( \mathcal{C}_{i},\gamma \left( t\right) \right) ,
\end{eqnarray*}%
for some $C>0.$ So Part 2 follows for a fixed $k\in \mathfrak{g}$ and a
fixed geodesic $\gamma .$

Now observe that in both cases, the left hand side depends continuously on
the choice of $k\in \mathfrak{g}$ and the choice of normal geodesic $\gamma
. $ Thus the theorem follows in general from the compactness of the unit
sphere in $\mathfrak{g}$ and the unit normal bundle of $\mathcal{C}_{i}.$
\end{proof}

\begin{proposition}
\label{nearly orthog killing}Let $\gamma $ be a geodesic that leaves $S_{i}$
orthogonally from a point of $\mathcal{C}_{i}.$ For $k^{1}\in \mathfrak{g}%
_{\gamma \left( 0\right) }$ and $k^{2}\in \mathfrak{m}_{\gamma \left(
0\right) }$ 
\begin{equation*}
g_{M}\left( k_{M}^{1}\left( \gamma \left( t\right) \right) ,k_{M}^{2}\left(
\gamma \left( t\right) \right) \right) =O\left( t\right) \left\vert
k_{M}^{1}\left( \gamma \left( t\right) \right) \right\vert \left\vert
k_{M}^{2}\left( \gamma \left( t\right) \right) \right\vert .
\end{equation*}

In particular, the angles between the subspaces 
\begin{equation*}
\left\{ \left. \left( -k_{G},k_{M}\right) \right\vert k\in \mathfrak{g}%
_{\gamma \left( 0\right) }\right\} \text{ and }\left\{ \left. \left(
-k_{G},k_{M}\right) \right\vert k\in \mathfrak{m}_{\gamma \left( 0\right)
}\right\}
\end{equation*}%
and the subspaces%
\begin{equation*}
\left\{ \left. k_{M}\right\vert k\in \mathfrak{g}_{\gamma \left( 0\right)
}\right\} \text{ and }\left\{ \left. k_{M}\right\vert k\in \mathfrak{m}%
_{_{\gamma \left( 0\right) }}\right\}
\end{equation*}%
are both $\frac{\pi }{2}\pm O\left( t\right) .$
\end{proposition}

\begin{proof}
The action of the circle generated by $k^{1}$ on $\gamma $ produces a
variation of $\gamma $ by geodesics that leave $S_{i}$ orthogonally from $%
\gamma \left( 0\right) ,$ and hence shows that $\nabla _{\gamma ^{\prime
}\left( 0\right) }k_{M}^{1}\in \nu _{\gamma \left( 0\right) }S_{i}.$

Since $k_{M}^{1}\left( \gamma \left( 0\right) \right) =0,$ 
\begin{equation*}
g_{M}\left( k_{M}^{1}\left( \gamma \left( 0\right) \right) ,k_{M}^{2}\left(
\gamma \left( 0\right) \right) \right) =0,
\end{equation*}%
and 
\begin{eqnarray*}
\left. \frac{d}{dt}g_{M}\left( k_{M}^{1}\left( \gamma \left( t\right)
\right) ,k_{M}^{2}\left( \gamma \left( t\right) \right) \right) \right\vert
_{t=0} &=&g_{M}\left( \nabla _{\gamma ^{\prime }\left( 0\right)
}k_{M}^{1},k_{M}^{2}\right) +g_{M}\left( k_{M}^{1},\nabla _{\gamma ^{\prime
}\left( 0\right) }k_{M}^{2}\right) \\
&=&0,
\end{eqnarray*}%
since $\nabla _{\gamma ^{\prime }\left( 0\right) }k_{M}^{1}\in \nu _{\gamma
\left( 0\right) }S_{i}$ and $k_{M}^{2}\in TS_{i}$ and $k_{M}^{1}\left(
\gamma \left( 0\right) \right) =0.$ On the other hand, 
\begin{eqnarray*}
\left. \frac{d^{2}}{dt^{2}}g_{M}\left( k_{M}^{1}\left( \gamma \left(
t\right) \right) ,k_{M}^{2}\left( \gamma \left( t\right) \right) \right)
\right\vert _{t=0} &=&g_{M}\left( \nabla _{\gamma ^{\prime }\left( 0\right)
}\nabla _{\gamma ^{\prime }\left( 0\right) }k_{M}^{1},k_{M}^{2}\right) \\
&&+2g_{M}\left( \nabla _{\gamma ^{\prime }\left( 0\right) }k_{M}^{1},\nabla
_{\gamma ^{\prime }\left( 0\right) }k_{M}^{2}\right) +g_{M}\left(
k_{M}^{1},\nabla _{\gamma ^{\prime }\left( 0\right) }\nabla _{\gamma
^{\prime }\left( 0\right) }k_{M}^{2}\right) \\
&=&-\left. R\left( k_{M}^{1},\gamma ^{\prime },\gamma ^{\prime
},k_{M}^{2}\right) \right\vert _{t=0}+2g_{M}\left( \nabla _{\gamma ^{\prime
}\left( 0\right) }k_{M}^{1},\nabla _{\gamma ^{\prime }\left( 0\right)
}k_{M}^{2}\right) \\
&=&2g_{M}\left( \nabla _{\gamma ^{\prime }\left( 0\right) }k_{M}^{1},\nabla
_{\gamma ^{\prime }\left( 0\right) }k_{M}^{2}\right)
\end{eqnarray*}

\noindent since $k_{M}^{1}\left( \gamma \left( 0\right) \right) =0.$ So%
\begin{eqnarray}
\left\vert g_{M}\left( k_{M}^{1}\left( \gamma \left( t\right) \right)
,k_{M}^{2}\left( \gamma \left( t\right) \right) \right) \right\vert
&=&\left\vert g_{M}\left( \nabla _{\gamma ^{\prime }\left( 0\right)
}k_{M}^{1},\nabla _{\gamma ^{\prime }\left( 0\right) }k_{M}^{2}\right)
\right\vert t^{2}+O\left( t^{3}\right)  \notag \\
&\leq &2\left\vert \nabla _{\gamma ^{\prime }\left( 0\right)
}k_{M}^{1}\right\vert _{g_{M}}\left\vert \nabla _{\gamma ^{\prime }\left(
0\right) }k_{M}^{2}\right\vert _{g_{M}}t^{2}.  \label{2nd tay}
\end{eqnarray}

Since $k_{M}^{1}\left( \gamma \left( 0\right) \right) =0$ 
\begin{equation}
\left\vert k_{M}^{1}\left( \gamma \left( t\right) \right) \right\vert
_{g_{M}}=\left\vert \nabla _{\gamma ^{\prime }\left( 0\right)
}k_{M}^{1}\right\vert _{g_{M}}t+O\left( t^{2}\right) .  \label{K_1}
\end{equation}%
By compactness of $\mathcal{C}_{i}$ and the unit sphere in $\mathfrak{m}%
_{\gamma \left( 0\right) }$ 
\begin{equation*}
\left\vert \nabla _{\gamma ^{\prime }\left( 0\right) }k_{M}^{2}\right\vert
_{g_{M}}\leq C_{1}\left\vert k_{G}^{2}\right\vert _{g_{\mathrm{bi}}}
\end{equation*}%
for some constant $C_{1}>0.$ By compactness of the unit sphere in $\mathfrak{%
m}_{\gamma \left( 0\right) },$%
\begin{equation*}
\left\vert k_{G}^{2}\right\vert _{g_{\mathrm{bi}}}\leq C_{2}\left\vert
k_{M}^{2}\left( \gamma \left( 0\right) \right) \right\vert _{g_{M}},
\end{equation*}%
for some constant $C_{2}>0,$ and for all sufficiently small $t$ 
\begin{equation*}
\left\vert k_{M}^{2}\left( \gamma \left( 0\right) \right) \right\vert
_{g_{M}}\leq 2\left\vert k_{M}^{2}\left( \gamma \left( t\right) \right)
\right\vert _{g_{M}}.
\end{equation*}%
Combining the previous three inequalities gives 
\begin{equation*}
\left\vert \nabla _{\gamma ^{\prime }\left( 0\right) }k_{M}^{2}\right\vert
_{g_{M}}\leq C\left\vert k_{M}^{2}\left( \gamma \left( t\right) \right)
\right\vert _{g_{M}}
\end{equation*}%
for all sufficiently small $t>0$ and some $C>0.$ Together with Inequality %
\ref{2nd tay} and Equation \ref{K_1} this gives 
\begin{equation*}
\left\vert g_{M}\left( k_{M}^{1}\left( \gamma \left( t\right) \right)
,k_{M}^{2}\left( \gamma \left( t\right) \right) \right) \right\vert =O\left(
t\right) \left\vert k_{M}^{1}\left( \gamma \left( t\right) \right)
\right\vert \left\vert k_{M}^{2}\left( \gamma \left( t\right) \right)
\right\vert ,
\end{equation*}%
for all sufficiently small $t.$

So the angle between 
\begin{equation*}
\left\{ \left. k_{M}\right\vert k\in \mathfrak{g}_{\gamma \left( 0\right)
}\right\} \text{ and }\left\{ \left. k_{M}\right\vert k\in \mathfrak{m}%
_{\gamma \left( 0\right) }\right\}
\end{equation*}%
is $\frac{\pi }{2}\pm O\left( t\right) ,$ and the angle between 
\begin{equation*}
\left\{ \left. \left( -k_{G},k_{M}\right) \right\vert k\in \mathfrak{g}%
_{\gamma \left( 0\right) }\right\} \text{ and }\left\{ \left. \left(
-k_{G},k_{M}\right) \right\vert k\in \mathfrak{m}_{\gamma \left( 0\right)
}\right\}
\end{equation*}%
is only closer to $\frac{\pi }{2}.$
\end{proof}

\begin{lemma}
\label{H---A--tensor}

For all $x\in \cup _{i}\Omega ^{i},$ $Y\in \mathrm{span}\left\{ TG\left(
x\right) ^{\perp }\cap \overline{\mathcal{H}}_{x}^{i}\right\} $ and $Z\in 
\mathrm{span}\left\{ X^{i},\text{ }TG\left( x\right) ^{\perp }\cap \overline{%
\mathcal{H}}_{x}^{i}\right\} $ 
\begin{equation*}
\left\vert |A_{\hat{Y}}^{Ch}\hat{Z}|-\left\vert A_{Y}^{\mathrm{reg}%
}Z\right\vert \right\vert \leq \left[ O\left( l\right) +O\left( \mathrm{dist}%
\left( \mathcal{C}_{1}\cup \cdots \cup \mathcal{C}_{p},x\right) \right) %
\right] \left\vert Z\right\vert \left\vert Y\right\vert ,
\end{equation*}%
where $A^{Ch}$ is the $A$--tensor of the Riemannian submersion $\left(
G\times M,l^{2}g_{\mathrm{bi}}+g_{M}\right) \longrightarrow \left(
M,g_{l}\right) .$
\end{lemma}

\begin{proof}
Since the splittings 
\begin{equation*}
\left\{ \left. \left( -k_{G},k_{M}\right) \right\vert k\in \mathfrak{g}%
_{Pr_{i}\left( x\right) }\right\} \oplus \left\{ \left. \left(
-k_{G},k_{M}\right) \right\vert k\in \mathfrak{m}_{Pr_{i}\left( x\right)
}\right\}
\end{equation*}%
and%
\begin{equation*}
\left\{ \left. k_{M}\right\vert k\in \mathfrak{g}_{Pr_{i}\left( x\right)
}\right\} \oplus \left\{ \left. k_{M}\right\vert k\in \mathfrak{m}%
_{Pr_{i}\left( x\right) }\right\}
\end{equation*}%
are nearly orthogonal, its is enough to compare the projections of $A_{\hat{Y%
}}^{Ch}\hat{Z}$ and $A_{Y}^{\mathrm{reg}}Z$ onto the corresponding subspaces.

For all $Y\in \mathrm{span}\left\{ TG\left( x\right) ^{\perp }\cap \overline{%
\mathcal{H}}_{x}^{i}\right\} $ and $Z\in \mathrm{span}\left\{ X^{i},\text{ }%
TG\left( x\right) ^{\perp }\cap \overline{\mathcal{H}}_{x}^{i}\right\} ,$ as
in the proof of Proposition \ref{Cheeg A--tensor}, we have 
\begin{eqnarray}
\left( l^{2}g_{\mathrm{bi}}+g_{M}\right) \left( A_{\hat{Y}}^{Ch}\hat{Z}%
,\left( -k_{G},k_{M}\right) \right) &=&-g_{M}\left( Z,\nabla _{Y}k_{M}\right)
\notag \\
&=&g_{M}\left( A_{Y}^{\mathrm{reg}}Z,k_{M}\right)  \label{stupid}
\end{eqnarray}

The set of real numbers 
\begin{equation*}
\left\{ \left. \left\vert k_{M}\left( x\right) \right\vert _{g_{M}}\text{ }%
\right\vert \text{ }x\in \Omega ^{i},\text{ }k\in \mathfrak{m}_{Pr_{i}\left(
x\right) },\text{ }\left\vert k_{G}\right\vert _{g_{bi}}=1\right\}
\end{equation*}%
has a positive lower bound. So for $x\in \Omega ^{i}$ and $k\in \mathfrak{m}%
_{Pr_{i}\left( x\right) }$ with $\left\vert k_{G}\right\vert _{g_{bi}}=1$ we
have 
\begin{equation*}
\frac{1}{\left\vert \left( -k_{G},k_{M}\right) \right\vert _{l^{2}g_{\mathrm{%
bi}}+g_{M}}^{2}}=\frac{1}{l^{2}\left\vert k_{G}\right\vert _{g_{\mathrm{bi}%
}}^{2}+\left\vert k_{M}\right\vert _{g_{M}}^{2}}\longrightarrow \frac{1}{%
\left\vert k_{M}\right\vert _{g_{M}}^{2}}\text{ as }l\rightarrow 0,
\end{equation*}%
uniformly on $\Omega ^{i}.$ So for $k\in \mathfrak{m}_{Pr_{i}\left( x\right)
}$ and unit $Y$ and $Z$ we have 
\begin{equation*}
\frac{\left( l^{2}g_{\mathrm{bi}}+g_{M}\right) \left( A_{\hat{Y}}^{Ch}\hat{Z}%
,\left( -k_{G},k_{M}\right) \right) }{\left\vert \left( -k_{G},k_{M}\right)
\right\vert _{l^{2}g_{\mathrm{bi}}+g_{M}}}\longrightarrow g_{M}\left( A_{Y}^{%
\mathrm{reg}}Z,\frac{k_{M}}{\left\vert k_{M}\right\vert _{g_{M}}}\right) 
\text{ as }l\rightarrow 0,
\end{equation*}%
uniformly on $\Omega ^{i}.$

On the other hand, if $k\in \mathfrak{g}_{Pr_{i}\left( x\right) },$ we
combine Equation \ref{stupid} with Part 2 of Lemma \ref{nabla k} to get 
\begin{equation*}
\left( l^{2}g_{\mathrm{bi}}+g_{M}\right) \left( A_{\hat{Y}}^{Ch}\hat{Z}%
,\left( -k_{G},k_{M}\right) \right) ^{2}\leq C\left\vert Z\right\vert
_{g_{M}}^{2}\left\vert Y\right\vert _{g_{M}}^{2}\left\vert k_{M}\right\vert
_{g_{M}}^{2}\mathrm{dist}\left( \mathcal{C}_{1}\cup \cdots \cup \mathcal{C}%
_{p},x\right) ^{2},
\end{equation*}%
for $C=\max \left\{ C_{i}\right\} .$ Dividing both sides by $\left\vert
\left( -k_{G},k_{M}\right) \right\vert _{l^{2}g_{\mathrm{bi}}+g_{M}}^{2}$ we
have 
\begin{eqnarray*}
\frac{\left( l^{2}g_{\mathrm{bi}}+g_{M}\right) \left( A_{\hat{Y}}^{Ch}\hat{Z}%
,\left( -k_{G},k_{M}\right) \right) ^{2}}{\left\vert \left(
-k_{G},k_{M}\right) \right\vert _{l^{2}g_{\mathrm{bi}}+g_{M}}^{2}} &\leq
&C\left\vert Z\right\vert _{g_{M}}^{2}\left\vert Y\right\vert _{g_{M}}^{2}%
\mathrm{dist}\left( \mathcal{C}_{1}\cup \cdots \cup \mathcal{C}_{p},x\right)
^{2}\frac{\left\vert k_{M}\right\vert _{g_{M}}^{2}}{l^{2}\left\vert
k_{G}\right\vert _{g_{\mathrm{bi}}}^{2}+\left\vert k_{M}\right\vert
_{g_{M}}^{2}} \\
&\leq &C\left\vert Z\right\vert _{g_{M}}^{2}\left\vert Y\right\vert
_{g_{M}}^{2}\mathrm{dist}\left( \mathcal{C}_{1}\cup \cdots \cup \mathcal{C}%
_{p},x\right) ^{2}.
\end{eqnarray*}%
Finally, by Part 2 of Lemma \ref{nabla k} 
\begin{equation*}
\frac{g_{M}\left( A_{Y}^{\mathrm{reg}}Z,k_{M}\right) ^{2}}{\left\vert
k_{M}\right\vert _{g_{M}}^{2}}\leq C\left\vert Z\right\vert
_{g_{M}}^{2}\left\vert Y\right\vert _{g_{M}}^{2}\mathrm{dist}\left( \mathcal{%
C}_{1}\cup \cdots \cup \mathcal{C}_{p},x\right) ^{2},
\end{equation*}%
for $C=\max \left\{ C_{i}\right\} ,$ and the result follows.
\end{proof}

\section{Two Steps To Better Curvature\label{two step}}

In this section, we prove two results that are the first two steps in the
proofs of Theorems \ref{main}, \ref{alm nonneg thm}, and \ref{supplement}.
They track the effects of first Cheeger deforming $g_{M}$ and then
performing the conformal change of Theorem \ref{uber conf}, allowing us to
improve the curvature of $M$.

One can also omit the first step and still prove Theorem \ref{main}. So the
reader who is only interested in the proof of Theorem \ref{main} can skip
this section.

Recall that $g_{l}$ is the metric on $M$ induced by the Riemannian
submersion 
\begin{equation*}
q_{G\times M}:\left( G\times M,l^{2}g_{\mathrm{bi}}+g\right) \longrightarrow
M,
\end{equation*}
and $d\pi ^{\text{reg }}\left( g_{l}\right) $ is the Riemannian metric on $%
\left( M/G\right) ^{\text{reg }}$induced by the Riemannian submersion $\pi ^{%
\text{reg }}:\left( M,g\right) \longrightarrow M/G.$

\begin{theorem}[\textbf{Step 1}]
\label{Step 1} Let $G$ be a compact Lie group acting isometrically on a
Riemannian $n$--manifold $\left( M,g\right) .$ For any $\varepsilon >0$
there is a neighborhood $\Omega ^{\prime }$ of $S_{1}\cup S_{2}\cup \cdots
\cup S_{p}$ as in Proposition \ref{Omega_i Dfn} and a Cheeger parameter $%
l_{1}$ such that for all $l\in \left( 0,l_{1}\right) $

\begin{equation}
\left\vert \mathrm{sec}_{g_{l}}(Y,Z)-\mathrm{sec}_{d\pi \left( g_{l}\right)
}(d\pi ^{\mathrm{reg}}\left( Y\right) ,d\pi ^{\mathrm{reg}}\left( Z\right)
)\right\vert <\frac{\varepsilon }{2}\text{ \label{Alm on H}}
\end{equation}%
if either $Y,Z\in TG(x)^{\perp }|_{M\setminus \Omega ^{\prime }}$ or $Y,Z\in
\left\{ TG(x)^{\perp }\cap \overline{\mathcal{H}}^{i}\right\} |_{\Omega
^{\prime }\setminus S_{1}\cup S_{2}\cup \cdots \cup S_{p}}$ for some $i\in
\left\{ 1,2,\ldots ,p\right\} $.

Moreover, $d\pi ^{\text{reg }}\left( g_{l}\right) $ is independent of $l$
and is equal to $d\pi ^{\text{reg }}\left( g\right) .$
\end{theorem}

\begin{proof}
For orthonormal $Y,Z\in TG(x)^{\perp }$ we have $\hat{Y}=\left( 0,Y\right) $
and $\hat{Z}=\left( 0,Z\right) ,$ so using the Horizontal Curvature Equation
we obtain%
\begin{eqnarray*}
\mathrm{sec}_{g_{l}}\left( Y,Z\right) &=&\mathrm{sec}_{l^{2}g_{\mathrm{bi}%
}+g_{M}}\left( \left( 0,Y\right) ,\left( 0,Z\right) \right) +3\left\vert A_{%
\hat{Y}}^{Ch}\hat{Z}\right\vert _{l^{2}g_{\mathrm{bi}}+g_{M}}^{2} \\
&=&\mathrm{sec}_{g_{M}}\left( Y,Z\right) +3\left\vert A_{\hat{Y}}^{Ch}\hat{Z}%
\right\vert _{l^{2}g_{\mathrm{bi}}+g_{M}}^{2}.
\end{eqnarray*}%
On the other hand, for $\pi ^{\text{reg}}:M\longrightarrow M/G,$ the
Horizontal Curvature Equation becomes 
\begin{equation*}
\mathrm{\sec }\left( d\pi ^{\text{reg}}Y,d\pi ^{\text{reg}}Z\right) =\mathrm{%
sec}_{g_{M}}\left( Y,Z\right) +3\left\vert A_{Y}^{\text{reg}}Z\right\vert
_{g_{M}}^{2}.
\end{equation*}

From Lemma \ref{H---A--tensor} we have 
\begin{equation*}
\left\vert \left\vert A_{\hat{Y}}^{Ch}\hat{Z}\right\vert _{l^{2}g_{\mathrm{bi%
}}+g_{M}}^{2}-\left\vert A_{Y}^{\text{reg}}Z\right\vert
_{g_{M}}^{2}\right\vert <O\left( l\right) +O\left( \mathrm{dist}\left( 
\mathcal{C}_{1}\cup \cdots \cup \mathcal{C}_{p},x\right) \right) ,
\end{equation*}%
for $Y,Z\in \left\{ TG(x)^{\perp }\cap \overline{\mathcal{H}}^{i}\right\}
|_{\Omega ^{\prime }\setminus S_{1}\cup S_{2}\cup \cdots \cup S_{p}},$ and
from Proposition \ref{Cheeg A--tensor} we have 
\begin{equation*}
\left\vert \left\vert A_{\hat{Y}}^{Ch}\hat{Z}\right\vert _{l^{2}g_{\mathrm{bi%
}}+g_{M}}^{2}-\left\vert A_{Y}^{\text{reg}}Z\right\vert
_{g_{M}}^{2}\right\vert <\frac{\varepsilon }{6},
\end{equation*}%
for all $Y,Z\in TG(x)^{\perp }|_{M\setminus \Omega ^{\prime }},$ provided $l$
is sufficiently small.

Inequality \ref{Alm on H} follows by combining the previous four displays.

Finally, since a Cheeger deformation does not change the metric on the
distribution that is orthogonal to the orbits $d\pi ^{\text{reg}}\left(
g_{l}\right) $ is independent of $l$ and equal to $d\pi ^{\text{reg }}\left(
g\right) .$
\end{proof}

Next we apply Theorem \ref{uber conf} to the metrics $g_{l}$ and obtain the
following.

\begin{theorem}[\textbf{Step 2}]
\label{Step 2} Let $M$ and $G$ be as in Theorem \ref{Step 1}. For any $%
\varepsilon >0,$ let $g_{l}$ be a metric that satisfies the conclusion of
Theorem \ref{Step 1}$.$ For any $K>0$, there is a neighborhood $\Omega _{1}$
of $S_{1}\cup S_{2}\cup \cdots \cup S_{p}$ and a $G$--invariant metric $%
\widetilde{g_{l}}=e^{2f}g_{l}$ so that if $V\in \mathrm{span}\left\{ 
\mathcal{V}^{i},X^{i}\right\} |_{\Omega _{1}}$ for some $i\in \left\{
1,\ldots ,p\right\} ,$ then%
\begin{equation*}
\mathrm{sec}_{\widetilde{g_{l}}}(V,W)\geq K
\end{equation*}%
for all $W\in T\Omega _{1}.$ Moreover, 
\begin{equation*}
\mathrm{sec}_{\widetilde{g_{l}}}(V,W)\geq \mathrm{sec}_{g_{l}}(V,W)-\frac{%
\varepsilon }{2}\text{ \label{Alm on Omega''}}
\end{equation*}%
for all $V,W\in TM.$
\end{theorem}

\section{Lifting Positive Ricci Curvature\label{ricci lift}}

In this section we prove Theorem \ref{main}. For convenience, re-scale so
that $Ric_{M^{\text{reg}}/G}\geq 2.$

To the best of our knowledge, the most efficient metric construction to
prove Theorem \ref{main} is to perform the conformal change of Theorem \ref%
{uber conf}, and then to Cheeger deform the resulting metric. In contrast,
to prove Theorem \ref{alm nonneg thm} we first Cheeger deform, then perform
a conformal change, and then further Cheeger deform. Consequently we also
use this $3$--step deformation to prove Theorem \ref{supplement}.

If $\left( M,g\right) $ satisfies the hypotheses of Theorem \ref{main}, and $%
g_{l}$ is a Cheeger deformation of $g,$ then $\left( M,g_{l}\right) $ also
satisfies the hypotheses of Theorem \ref{main}. So for simplicity of
notation we will write $g$ for $g_{l}$ in this section. On the one hand,
this points to the most efficient path to proving Theorem \ref{main}, on the
other hand, since $g_{l}$ satisfies the hypotheses of Theorem \ref{main} we
will simultaneously verify the positive Ricci curvature portion of the
conclusion of Theorem \ref{supplement}.

We obtain Theorem \ref{main} by combining the conformal change of Theorem %
\ref{uber conf} and the following two results, both of which are proven in
this section. The first result shows that all Cheeger deformations of $%
\tilde{g}$ have positive Ricci curvature on $\Omega _{1},$ where $\tilde{g}$
and $\Omega _{1}$ are as in Theorem \ref{uber conf}. The second result shows
that there are Cheeger deformations of $\tilde{g}$ that have positive Ricci
curvature on $M\setminus \Omega _{1}.$

\begin{theorem}
\label{Ric_omega >0}Let $\left( M,g\right) $ be as in Theorem \ref{main}.
Given $\varepsilon ,K>0,$ let $\widetilde{g}$ be the $G$--invariant metric
on $M$ from Theorem \ref{uber conf}. If $\varepsilon $ is sufficiently small
and $K$ is sufficiently large$,$ then for all $\lambda \in \left( 0,\infty
\right) $ 
\begin{equation*}
Ric_{\tilde{g}_{\lambda }}|_{\Omega _{1}}>0,
\end{equation*}%
where $\tilde{g}_{\lambda }$ is the metric on $M$ induced by the Riemannian
submersion 
\begin{equation*}
q_{G\times M}:\left( G\times M,\lambda ^{2}g_{\mathrm{bi}}+\tilde{g}\right)
\longrightarrow M,
\end{equation*}
and $\Omega _{1}$ is as in Theorem \ref{uber conf}.
\end{theorem}

\begin{theorem}
\label{generic cheeger}Given $\left( M,g\right) $ as in Theorem \ref{main},
let $\tilde{g}_{\lambda }$ be the metric on $M$ from Theorem \ref{Ric_omega
>0}, and let $\Omega _{1}$ be as in Theorem \ref{uber conf}. Then 
\begin{equation*}
Ric_{\tilde{g}_{\lambda }}|_{M\setminus \Omega _{1}}>0,
\end{equation*}%
provided $\lambda $ is sufficiently small.
\end{theorem}

Before proceeding with the proofs, we record the following result, which is
obtained by taking the trace of the Horizontal Curvature Equation.

\begin{proposition}
\label{Gray-O'Neill}Let $\pi :\left( E,g\right) \longrightarrow B$ be a
Riemannian submersion with horizontal distribution $H.$ Using the
superscript $^{\mathrm{Horiz}}$ to denote the $H$--component of a vector,
for $x,y,z\in H$ we define 
\begin{eqnarray*}
Ric^{\mathrm{Horiz}}\left( x,y\right) &\equiv &Trace\left( z\mapsto \left\{
R\left( z,x\right) y\right\} ^{\mathrm{Horiz}}\right) , \\
R^{A}\left( z,x\right) y &\equiv &2A_{y}A_{z}x-A_{z}A_{x}y-A_{x}A_{y}z,\text{
and} \\
Ric^{A}\left( x,y\right) &\equiv &Trace\left( z\mapsto R^{A}\left(
z,x\right) y\right) .
\end{eqnarray*}%
Extend $Ric^{\mathrm{Horiz}}$ and $Ric^{A}$ to be $\left( 0,2\right) $%
--tensors on $M$ by setting $Ric^{\mathrm{Horiz}}\left( v,\cdot \right)
=Ric^{A}\left( v,\cdot \right) =0,$ if $v$ is vertical.

Then 
\begin{equation}
\pi ^{\ast }\left( Ric_{B}\right) =Ric^{\mathrm{Horiz}}+3Ric^{A}.
\label{Ricc O'neill}
\end{equation}
\end{proposition}

\begin{remark}
Let $\left\{ e_{i}\right\} _{i=2}^{\mathrm{\dim B}}$ be an extension of $x$
to an orthonormal basis for the horizontal distribution. Then we have 
\begin{eqnarray*}
Ric^{A}\left( x,x\right) &=&\Sigma _{i=2}^{\mathrm{\dim B}}g\left(
R^{A}\left( e_{i},x\right) x,e_{i}\right) \\
&=&\Sigma _{i=2}^{\mathrm{\dim B}}\left( 2g\left(
A_{x}A_{e_{i}}x,e_{i}\right) -g\left( A_{x}A_{x}e_{i},e_{i}\right) \right) \\
&=&\Sigma _{i=2}^{\mathrm{\dim B}}3g\left( A_{x}e_{i},A_{x}e_{i}\right) \\
&\geq &0.
\end{eqnarray*}%
Combined with Equation \ref{Ricc O'neill}, this yields $\pi ^{\ast }\left(
Ric_{B}\right) \geq Ric^{\mathrm{Horiz}}.$ In contrast, the inequality $\pi
^{\ast }\left( Ric_{B}\right) \geq Ric_{\left( E,g\right) }$ does not hold
for all Riemannian submersions \cite{ProWilh2}.
\end{remark}

\begin{proof}[Proof of Theorem \protect\ref{Ric_omega >0}]
Recall that $\Omega _{1}=\cup \Omega _{1}^{i}$ where $\Omega _{1}^{i}$ is as
in Proposition \ref{Omega_i Dfn}, and for each $\Omega _{1}^{i}$ we have a
splitting 
\begin{equation*}
T\left( \Omega _{1}^{i}\right) =\mathcal{H}^{i}\oplus \mathcal{V}^{i}\oplus 
\mathrm{span}\left\{ X^{i}\right\} ,
\end{equation*}%
as in \ref{HVX splitting}. For simplicity, throughout this proof, we will
write $X$ for any of the $X^{i}$s.

Let $W\in T\Omega _{1}$ be any vector with $\left\vert Ch_{\lambda }\left(
W\right) \right\vert _{\tilde{g}_{\lambda }}=1.$ Since $X\in TG\left(
x\right) ^{\perp }$, we have $\tilde{g}\left( X,V\right) =0$ if and only if $%
\tilde{g}_{\lambda }\left( Ch_{\lambda }\left( X\right) ,Ch_{\lambda }\left(
V\right) \right) =0.$ So we may write 
\begin{equation*}
Ch_{\lambda }\left( W\right) =Ch_{\lambda }\left( X\right) \cos \sigma
+Ch_{\lambda }\left( V\right) \sin \sigma
\end{equation*}%
with $X\perp V$ and $\left\vert Ch_{\lambda }\left( V\right) \right\vert =1.$

Choose $\left\{ E_{i}\right\} _{i=2}^{n}\subset T\Omega _{1}$ so that $%
\left\{ Ch_{\lambda }\left( W\right) ,Ch_{\lambda }\left( E_{2}\right)
,\left\{ Ch_{\lambda }\left( E_{i}\right) \right\} _{i=3}^{n}\right\} $ is
an orthonormal basis with $E_{2}\in \mathrm{span}\left\{ X,V\right\} .$ By
the Horizontal Curvature Equation and Theorem \ref{uber conf}%
\begin{eqnarray}
\mathrm{sec}_{\tilde{g}_{\lambda }}\left( Ch_{\lambda }\left( W\right)
,Ch_{\lambda }\left( E_{2}\right) \right) &=&\mathrm{sec}_{\tilde{g}%
_{\lambda }}\left( Ch_{\lambda }\left( X\right) ,Ch_{\lambda }\left(
V\right) \right)  \notag \\
&\geq &\mathrm{sec}_{\lambda ^{2}g_{\mathrm{bi}}+\tilde{g}}\left( \hat{X},%
\hat{V}\right)  \notag \\
&=&\mathrm{curv}_{\lambda ^{2}g_{\mathrm{bi}}+\tilde{g}}\left( \left(
0,X\right) ,\left( \frac{\kappa _{V}}{\lambda ^{2}},V\right) \right)  \notag
\\
&=&\mathrm{curv}_{\tilde{g}}\left( X,V\right)  \notag \\
&\geq &K\left\vert V\right\vert _{\tilde{g}}^{2}.  \label{sec(W,E_2)}
\end{eqnarray}

For $i\geq 3$ we have%
\begin{eqnarray}
\mathrm{sec}_{\tilde{g}_{\lambda }}\left( Ch_{\lambda }\left( W\right)
,Ch_{\lambda }\left( E_{i}\right) \right) &\geq &\mathrm{curv}_{\lambda
^{2}g_{\mathrm{bi}}+\tilde{g}}\left( \hat{W},\hat{E}_{i}\right)  \notag \\
&\geq &\mathrm{curv}_{\tilde{g}}\left( W,E_{i}\right)  \label{curv(W, E_i)}
\\
&=&\mathrm{curv}_{\tilde{g}}\left( X\cos \sigma +V\sin \sigma ,E_{i}\right) 
\notag \\
&\geq &\cos ^{2}\sigma K\left\vert E_{i}\right\vert ^{2}+2\sin \sigma \cos
\sigma R^{\tilde{g}}\left( X,E_{i},E_{i},V\right)  \notag \\
&&-2\sin ^{2}\sigma \left\vert \mathrm{minsec}_{g}\right\vert \left\vert
V\right\vert ^{2}\left\vert E_{i}\right\vert ^{2},  \notag
\end{eqnarray}%
where we applied Part 5 of Theorem \ref{uber conf} to replace $-\left\vert 
\mathrm{minsec}_{\tilde{g}}\right\vert $ by $-2\left\vert \mathrm{minsec}%
_{g}\right\vert .$

By the antisymmetry of $R^{\tilde{g}},$ $R^{\tilde{g}}\left(
X,E_{i},E_{i},V\right) =R^{\tilde{g}}\left( X,E_{i},E_{i},V^{\perp
,E_{i}}\right) ,$ where $V^{\perp ,E_{i}}$ is the component of $V$ that is
perpendicular to $E_{i}.$ For $i\geq 3,$ $X\perp E_{i},$ so $X\perp V^{\perp
,E_{i}}.$ Combining this with Lemma \ref{Hessian Est} and Equation \ref%
{confffcurv} we conclude%
\begin{eqnarray*}
-\left\vert R^{\tilde{g}}\left( X,E_{i},E_{i},V\right) \right\vert
&=&-\left\vert R^{\tilde{g}}\left( X,E_{i},E_{i},V^{\perp ,E_{i}}\right)
\right\vert \\
&\geq &-\left\vert R^{g}\left( X,E_{i},E_{i},V^{\perp ,E_{i}}\right)
\right\vert -\left\vert E_{i}\right\vert _{g}^{2}\left\vert V\right\vert _{g}
\\
&\geq &-(\left\vert R^{g}\right\vert +1)\left\vert E_{i}\right\vert
_{g}^{2}\left\vert V\right\vert _{g}.
\end{eqnarray*}%
Thus%
\begin{eqnarray*}
\mathrm{sec}_{\tilde{g}_{\lambda }}\left( Ch_{\lambda }\left( W\right)
,Ch_{\lambda }\left( E_{i}\right) \right) &\geq &\cos ^{2}\sigma K\left\vert
E_{i}\right\vert ^{2}-2\sin \sigma \cos \sigma \left( \left\vert
R^{g}\right\vert +1\right) \left\vert E_{i}\right\vert _{g}^{2}\left\vert
V\right\vert _{g} \\
&&-2\sin ^{2}\sigma \left\vert \mathrm{minsec}_{g}\right\vert \left\vert
V\right\vert _{g}^{2}\left\vert E_{i}\right\vert _{g}^{2}.
\end{eqnarray*}%
Combining this with Equation \ref{sec(W,E_2)} we have 
\begin{eqnarray*}
Ric\left( W,W\right) &\geq &K\left\vert V\right\vert
_{g}^{2}+\sum_{i=3}^{n}\left( \cos ^{2}\sigma K-2\sin \sigma \cos \sigma
\left( \left\vert R^{g}\right\vert +1\right) \left\vert V\right\vert
_{g}\right) \left\vert E_{i}\right\vert _{g}^{2} \\
&&-\sum_{i=3}^{n}\left( 2\sin ^{2}\sigma \left\vert \mathrm{minsec}%
_{g}\right\vert \left\vert V\right\vert _{g}^{2}\right) \left\vert
E_{i}\right\vert _{g}^{2}.
\end{eqnarray*}

Since $\left\vert E_{i}\right\vert _{g}\leq 1$ and we can choose $K$ to be
as large as we please compared to $\left\vert \mathrm{minsec}_{g}\right\vert 
$, we have 
\begin{eqnarray}
Ric\left( W,W\right) &\geq &\frac{K}{2}\left\vert V\right\vert
_{g}^{2}+\sum_{i=3}^{n}\left( \cos ^{2}\sigma K-2\sin \sigma \cos \sigma
\left( \left\vert R^{g}\right\vert +1\right) \left\vert V\right\vert
_{g}\right) \left\vert E_{i}\right\vert _{g}^{2}  \notag \\
&\geq &\sum_{i=3}^{n}\left( \cos ^{2}\sigma K-2\cos \sigma \left( \left\vert
R^{g}\right\vert +1\right) \left\vert V\right\vert _{g}+\frac{K}{2\left(
n-2\right) }\left\vert V\right\vert _{g}^{2}\right) \left\vert
E_{i}\right\vert _{g}^{2}  \notag \\
&=&\sum_{i=3}^{n}\left( \left( \cos \sigma -\left( \left\vert
R^{g}\right\vert +1\right) \left\vert V\right\vert _{g}\right) ^{2}-\cos
^{2}\sigma -\left( \left\vert R^{g}\right\vert +1\right) ^{2}\left\vert
V\right\vert _{g}^{2}\right) \left\vert E_{i}\right\vert _{g}^{2}  \notag \\
&&+\sum_{i=3}^{n}\left( \cos ^{2}\sigma K+\frac{K}{2\left( n-2\right) }%
\left\vert V\right\vert _{g}^{2}\right) \left\vert E_{i}\right\vert _{g}^{2}
\notag \\
&\geq &\sum_{i=3}^{n}\left( \cos ^{2}\sigma \frac{K}{2}+\frac{K}{3\left(
n-2\right) }\left\vert V\right\vert _{g}^{2}\right) \left\vert
E_{i}\right\vert _{g}^{2},  \label{ricci lower 2}
\end{eqnarray}%
since $K$ can be arbitrarily large. This gives us a positive lower bound for
the Ricci curvature on the regular part of $\Omega _{1}.$ Since the regular
part of $\Omega _{1}$ is not compact we also need to see that this bound is
uniformly positive. For this we analyze the norms $\left\vert V\right\vert
_{g}$ and $\left\vert E_{i}\right\vert _{g}.$

Writing $Y$ for either $\frac{V}{\left\vert V\right\vert _{\tilde{g}}}$ or $%
\frac{E_{i}}{\left\vert E_{i}\right\vert _{\tilde{g}}},$ we have%
\begin{equation*}
\left\vert \hat{Y}\right\vert _{\lambda ^{2}g_{\mathrm{bi}}+\tilde{g}}^{2}=%
\frac{\left\vert \kappa _{Y}\right\vert _{g_{\mathrm{bi}}}^{2}}{\lambda ^{2}}%
+\left\vert Y\right\vert _{\tilde{g}}^{2}
\end{equation*}%
Since $\left\vert Y\right\vert _{\tilde{g}}=1,$ by Part 1 of Proposition \ref%
{bound on kappa} we have%
\begin{equation*}
\left\vert \hat{Y}\right\vert _{\lambda ^{2}g_{\mathrm{bi}}+\tilde{g}%
}^{2}\leq \frac{C}{\lambda ^{2}}+1
\end{equation*}%
for some $C>0.$ So when $Y=\frac{E_{i}}{\left\vert E_{i}\right\vert _{\tilde{%
g}}}$ we have 
\begin{eqnarray*}
\frac{\left\vert \hat{E}_{i}\right\vert _{\lambda ^{2}g_{\mathrm{bi}}+\tilde{%
g}}^{2}}{\left\vert E_{i}\right\vert _{\tilde{g}}^{2}} &=&\left\vert \hat{Y}%
\right\vert _{\lambda ^{2}g_{\mathrm{bi}}+\tilde{g}}^{2} \\
&\leq &\frac{C}{\lambda ^{2}}+1.
\end{eqnarray*}%
Since $\left\vert \hat{E}_{i}\right\vert _{\lambda ^{2}g_{\mathrm{bi}}+%
\tilde{g}}^{2}=1,$ we conclude 
\begin{equation*}
\frac{\lambda ^{2}}{C+\lambda ^{2}}=\frac{1}{\frac{C}{\lambda ^{2}}+1}\leq
\left\vert E_{i}\right\vert _{g}^{2}.
\end{equation*}%
The same argument gives us 
\begin{equation*}
\frac{\lambda ^{2}}{C+\lambda ^{2}}\leq \left\vert V\right\vert _{g}^{2}.
\end{equation*}

Combining the previous two displays with Inequality \ref{ricci lower 2}
gives a uniform positive lower bound for $Ric_{\Omega _{1}}.$
\end{proof}

\begin{remark}
The positive lower bound on $Ric_{\tilde{g}_{\lambda }}|_{\Omega _{1}}$
above is far from optimal. In fact, the lower bound on $\left\vert
E_{i}\right\vert _{g}^{2}$ is minimal on the vectors $E_{i}$ for which $%
\left\vert \kappa _{E_{i}}\right\vert _{g_{\mathrm{bi}}}^{2}$ is maximal,
and one of the non-negative terms that we dropped in Inequality \ref{curv(W,
E_i)} is very large for these vectors, if $\lambda $ is small, cf.
Proposition \ref{vertical Ricci}.
\end{remark}

We will make use of Theorem 1 of \cite{Ber}, which is a consequence of the
proof of Proposition 3.4 of \cite{N}.

\begin{proposition}
\label{Berestovski}(Theorem 1 of \cite{Ber}) Let $M=G/H$ be an effective
homogeneous space with $G$ a connected Lie group and $H$ a compact subgroup.
Let $\mathfrak{h}$ be the Lie Algebra of $H$ and $\mathfrak{m}$ the
orthogonal complement of $\mathfrak{h}$ with respect to $g_{\mathrm{bi}}.$
Let 
\begin{equation*}
C\left( \mathfrak{m}\right) \equiv \left\{ \left. v\in \mathfrak{m}%
\right\vert \left[ v,w\right] =0\text{ for all }w\in \mathfrak{m}\right\} .
\end{equation*}%
If $\pi _{1}\left( M\right) $ is finite, then $C\left( \mathfrak{m}\right)
=0.$
\end{proposition}

Before proving Theorem \ref{generic cheeger} we establish three preliminary
results.

\begin{proposition}
\label{vertical Ricci}Let $g$ be any $G$--invariant metric on $M,$ and let $%
\mathcal{K}$ be any compact subset of $M^{\text{reg}}.$ Given any $C>0$,
there is an $l\left( C\right) >0$ so that for all $V\in TG\left( x\right) $
with $x\in \mathcal{K}$ 
\begin{equation*}
Ric_{\lambda ^{2}g_{\mathrm{bi}}+g}^{\mathrm{Horiz}}\left( \hat{V},\hat{V}%
\right) >C\left\vert \hat{V}\right\vert _{\lambda ^{2}g_{\mathrm{bi}%
}+g}^{2}>0
\end{equation*}%
for all $\lambda \in \left( 0,l\left( C\right) \right) .$
\end{proposition}

\begin{proof}
Let $\left\{ \hat{V},\hat{W}_{1},\ldots ,\hat{W}_{p},\hat{Y}_{1},\ldots ,%
\hat{Y}_{m}\right\} $ be a $\left( \lambda ^{2}g_{\mathrm{bi}}+g\right) $%
--orthogonal basis for the horizontal space of our Riemannian submersion $%
q_{G\times M}:\left( G\times M,\lambda ^{2}g_{\mathrm{bi}}+g\right)
\longrightarrow M$ with $W_{1},\ldots ,W_{p}\in TG\left( x\right) $ and $%
Y_{1},\ldots ,Y_{m}\in TG\left( x\right) ^{\perp },$ $\left\vert
V\right\vert _{g}=\left\vert W_{i}\right\vert _{g}=\left\vert
Y_{i}\right\vert _{g}=1.$ Since each $Y_{i}$ is in $TG\left( x\right)
^{\perp },$ $\hat{Y}_{i}=\left( 0,Y_{i}\right) .$ Therefore%
\begin{eqnarray}
\sum\limits_{i=1}^{m}\mathrm{curv}_{\lambda ^{2}g_{\mathrm{bi}}+g}\left( 
\hat{V},\hat{Y}_{i}\right) &=&\sum\limits_{i=1}^{m}\mathrm{curv}_{g}\left(
V,Y_{i}\right)  \notag \\
&=&\sum\limits_{i=1}^{m}\mathrm{sec}_{g}\left( V,Y_{i}\right) .
\label{vertizontal}
\end{eqnarray}%
Whereas%
\begin{eqnarray}
\sum\limits_{i=1}^{p}\mathrm{curv}_{\lambda ^{2}g_{\mathrm{bi}}+g}\left( 
\hat{V},\hat{W}_{i}\right) & = \sum\limits_{i=1}^{p}\left[ \mathrm{curv}%
_{\lambda ^{2}g_{\mathrm{bi}}}\left( \frac{\kappa _{V}}{\lambda ^{2}},\frac{%
\kappa _{W_{i}}}{\lambda ^{2}}\right) +\mathrm{curv}_{g}\left(
V,W_{i}\right) \right]  \notag \\
&= \sum\limits_{i=1}^{p}\left[ \frac{1}{\lambda ^{6}}\mathrm{curv}_{g_{%
\mathrm{bi}}}\left( \kappa _{V},\kappa _{W_{i}}\right) +\mathrm{sec}%
_{g}\left( V,W_{i}\right) \right] .  \label{vertical}
\end{eqnarray}

Combining our hypothesis that $\left\vert \pi _{1}\left( \mathrm{princ.}%
\text{ \textrm{orbit}}\right) \right\vert <\infty $ with Proposition \ref%
{Berestovski}, and the fact that for all $i,$ $\kappa _{V},\kappa
_{W_{i}}\in \mathfrak{m}_{x}$, we conclude that for at least one $i,$ $%
\mathrm{curv}_{g_{\mathrm{bi}}}\left( \kappa _{V},\kappa _{W_{i}}\right) >0.$%
\ Moreover, our normalization $\left\vert V\right\vert _{g}=\left\vert
W_{i}\right\vert _{g}=1,$ gives us a constant $C_{1}>0$ so that throughout $%
\mathcal{K}$%
\begin{equation}
\max_{i}\mathrm{curv}_{g_{\mathrm{bi}}}\left( \kappa _{V},\kappa
_{W_{i}}\right) \geq C_{1}.  \label{Berestovski Ineq}
\end{equation}

By Part 1 of Proposition \ref{bound on kappa}, there is another, constant $%
C_{3}>0$ so that throughout $\mathcal{K}$ for $\lambda $ sufficiently small%
\begin{eqnarray}
\left\vert \hat{V}\right\vert _{\lambda ^{2}g_{\mathrm{bi}}+g}^{2} &=&\frac{%
g_{\mathrm{bi}}\left( \kappa _{V},\kappa _{V}\right) }{\lambda ^{2}}+1 
\notag  \label{norm V hat} \\
&\leq &\frac{C_{3}}{\lambda ^{2}}  \notag
\end{eqnarray}%
and similarly 
\begin{equation*}
\left\vert \hat{W}_{i}\right\vert _{\lambda ^{2}g_{\mathrm{bi}}+g}^{2}\leq 
\frac{C_{3}}{\lambda ^{2}}.
\end{equation*}%
Combining Equation \ref{vertizontal} with $\left\vert \hat{V}\right\vert
_{\lambda ^{2}g_{\mathrm{bi}}+g}^{2}\geq 1$ and $\left\vert \hat{W}%
_{i}\right\vert _{\lambda ^{2}g_{\mathrm{bi}}+g}^{2}\geq 1$ we have 
\begin{eqnarray*}
\sum\limits_{i=1}^{m}\mathrm{sec}_{\lambda ^{2}g_{\mathrm{bi}}+g}\left( \hat{%
V},\hat{Y}_{i}\right) &=&\frac{1}{\left\vert \hat{V}\right\vert _{\lambda
^{2}g_{\mathrm{bi}}+g}^{2}}\sum\limits_{i=1}^{m}\mathrm{curv}_{\lambda
^{2}g_{\mathrm{bi}}+g}\left( \hat{V},\hat{Y}_{i}\right) \\
&=&\frac{1}{\left\vert \hat{V}\right\vert _{\lambda ^{2}g_{\mathrm{bi}%
}+g}^{2}}\left( \sum\limits_{i=1}^{m}\mathrm{sec}_{g}\left( V,Y_{i}\right)
\right) \\
&\geq &-m\left\vert \min \sec _{g}\right\vert .
\end{eqnarray*}%
Equation \ref{vertical} and Inequality \ref{Berestovski Ineq} give 
\begin{eqnarray*}
\sum\limits_{i=1}^{p}\mathrm{sec}_{\lambda ^{2}g_{\mathrm{bi}}+g}\left( \hat{%
V},\hat{W}_{i}\right) &=&\frac{1}{\left\vert \hat{V}\right\vert _{\lambda
^{2}g_{\mathrm{bi}}+g}^{2}}\sum\limits_{i=1}^{p}\frac{\mathrm{curv}_{\lambda
^{2}g_{\mathrm{bi}}+g}\left( \hat{V},\hat{W}_{i}\right) }{\left\vert \hat{W}%
_{i}\right\vert _{\lambda ^{2}g_{\mathrm{bi}}+g}^{2}} \\
&\geq &\frac{1}{\left\vert \hat{V}\right\vert _{\lambda ^{2}g_{\mathrm{bi}%
}+g}^{2}}\left( \frac{1}{\lambda ^{6}}C_{1}\frac{1}{\left\vert \hat{W}%
_{j_{0}}\right\vert _{\lambda ^{2}g_{\mathrm{bi}}+g}^{2}}+\sum%
\limits_{i=1}^{p}\frac{1}{\left\vert \hat{W}_{i}\right\vert _{\lambda ^{2}g_{%
\mathrm{bi}}+g}^{2}}\mathrm{sec}_{g}\left( V,W_{i}\right) \right)
\end{eqnarray*}%
for some $j_{0}.$ So using Inequality \ref{norma of hat} and $\left\vert 
\hat{W}_{i}\right\vert _{\lambda ^{2}g_{\mathrm{bi}}+g}^{2}\geq 1,$ 
\begin{eqnarray*}
\sum\limits_{i=1}^{p}\mathrm{sec}_{\lambda ^{2}g_{\mathrm{bi}}+g}\left( \hat{%
V},\hat{W}_{i}\right) &\geq &\frac{\lambda ^{4}}{C_{3}^{2}}\frac{1}{\lambda
^{6}}C_{1}-\sum\limits_{i=1}^{p}\left\vert \min \sec _{g}\right\vert \\
&\geq &\frac{C_{4}}{\lambda ^{2}}-p\left\vert \min \sec _{g}\right\vert ,
\end{eqnarray*}%
for $C_{4}=\frac{C_{1}}{C_{3}^{2}}.$ Since $m+p=n-1,$ combining gives%
\begin{equation*}
\frac{1}{\left\vert \hat{V}\right\vert _{\lambda ^{2}g_{\mathrm{bi}}+g}^{2}}%
Ric_{\lambda ^{2}g_{\mathrm{bi}}+g}^{\mathrm{Horiz}}\left( \hat{V},\hat{V}%
\right) \geq \frac{C_{4}}{\lambda ^{2}}-\left( n-1\right) \left\vert \min
\sec _{\tilde{g}}\right\vert .
\end{equation*}%
Since $C_{4}$ and $\left\vert \min \sec _{g}\right\vert $ are independent of 
$\lambda ,$ the right hand side becomes arbitrarily large as $\lambda
\rightarrow 0.$
\end{proof}

\begin{proposition}
\label{vertizontal Ricci}Let $g$ be any $G$--invariant metric on $M.$ For
any compact subset $\mathcal{K}\subset M^{\text{reg}},$ all $x\in \mathcal{K}%
,$ any unit vector $Z\in T_{x}G\left( x\right) ^{\perp },$ and any $V\in $ $%
T_{x}G\left( x\right) $ 
\begin{equation*}
\frac{1}{\left\vert \hat{V}\right\vert _{\lambda ^{2}g_{\mathrm{bi}}+g}}%
\left\vert Ric_{\lambda ^{2}g_{\mathrm{bi}}+g}^{\mathrm{Horiz}}\left( \hat{Z}%
,\hat{V}\right) \right\vert \longrightarrow 0
\end{equation*}%
uniformly on $\mathcal{K}$ as $\lambda \rightarrow 0.$
\end{proposition}

\begin{proof}
Since $\hat{Z}=\left( 0,Z\right) $ and $\lambda ^{2}g_{\mathrm{bi}}+g$ is a
product metric 
\begin{equation*}
R_{\lambda ^{2}g_{\mathrm{bi}}+g}\left( \hat{U},\hat{V},\hat{Z},\hat{U}%
\right) =R_{g}\left( U,V,Z,U\right) .
\end{equation*}%
If we assume that $\left\vert \hat{U}\right\vert _{\lambda ^{2}g_{\mathrm{bi}%
}+g}^{2}=1,$ it follows that $\left\vert U\right\vert _{g}^{2}\leq
\left\vert \hat{U}\right\vert _{\lambda ^{2}g_{\mathrm{bi}}+g}^{2}=1,$ so
dividing we get 
\begin{eqnarray}
\frac{1}{\left\vert \hat{V}\right\vert _{\lambda ^{2}g_{\mathrm{bi}}+g}}%
\left\vert R_{\lambda ^{2}g_{\mathrm{bi}}+g}\left( \hat{U},\hat{V},\hat{Z},%
\hat{U}\right) \right\vert &\leq &\left\vert R_{g}\right\vert \left\vert
U\right\vert _{g}^{2}\frac{\left\vert V\right\vert _{g}}{\left\vert \hat{V}%
\right\vert _{\lambda ^{2}g_{\mathrm{bi}}+g}}  \label{prod small} \\
&\leq &\left\vert R_{g}\right\vert \frac{\left\vert V\right\vert _{g}}{%
\left\vert \hat{V}\right\vert _{\lambda ^{2}g_{\mathrm{bi}}+g}}.  \notag
\end{eqnarray}

Normalize so that $\left\vert V\right\vert _{g}=1,$ and combine Part 2 of
Proposition \ref{bound on kappa} with 
\begin{equation*}
\left\vert \hat{V}\right\vert _{\lambda ^{2}g_{\mathrm{bi}}+g}^{2}=\frac{g_{%
\mathrm{bi}}\left( \kappa _{V},\kappa _{V}\right) }{\lambda ^{2}}+1,
\end{equation*}
to conclude that $\left\vert \hat{V}\right\vert _{\lambda ^{2}g_{\mathrm{bi}%
}+\tilde{g}_{l}}^{2}\rightarrow \infty $ uniformly on $\mathcal{K}$ as $%
\lambda \rightarrow 0.$

So $\frac{\left\vert V\right\vert _{g}}{\left\vert \hat{V}\right\vert
_{\lambda ^{2}g_{\mathrm{bi}}+g}}\rightarrow 0$ uniformly on $\mathcal{K}$
as $\lambda \rightarrow 0,$ and the result follows.
\end{proof}

In the previous two propositions we estimated $Ric_{\lambda ^{2}g_{\mathrm{bi%
}}+g}^{\mathrm{Horiz}}$ for an \emph{abstract} $G$--invariant metric $g.$ In
part, we did this because it is simpler to drop the terms involving the
Cheeger $A$--tensor, $A^{\mathrm{Ch}}.$ The metrics were abstract because we
will still need the $A^{\mathrm{Ch}}$--terms to control the Ricci curvature
on vectors in $T_{x}G\left( x\right) ^{\perp }.$ This will be achieved with
the next result, Proposition \ref{horizontal cheeger}, where we study the $%
Ric^{\mathrm{Horiz}}$ tensor on vectors in $T_{x}G\left( x\right) ^{\perp }$
for an iterated Cheeger deformation $q_{G\times M}:\left( G\times M,\lambda
_{1}^{2}g_{\mathrm{bi}}+\tilde{g}_{\lambda _{0}}\right) \longrightarrow M.$
It is of course true that $\left( \tilde{g}_{\lambda _{0}}\right) _{\lambda
_{1}}=\tilde{g}_{\lambda }$ for some $\lambda ,$ but the tensors $%
Ric_{\lambda _{1}^{2}g_{\mathrm{bi}}+\tilde{g}_{\lambda _{0}}}^{\mathrm{Horiz%
}}$ and $Ric_{\lambda ^{2}g_{\mathrm{bi}}+\tilde{g}}^{\mathrm{Horiz}}$ can
differ significantly since the corresponding Cheeger $A$--tensors that are
dropped can be quite different.

Since the previous two propositions contain estimates for $Ric_{\lambda
^{2}g_{\mathrm{bi}}+g}^{\mathrm{Horiz}}$ with respect to an abstract $G$%
--invariant metric $g,$ we will then be able to combine the three
Propositions \ref{vertical Ricci}, \ref{vertizontal Ricci}, and \ref%
{horizontal cheeger} to prove Theorem \ref{generic cheeger}.

\begin{proposition}
\label{horizontal cheeger}Let $\left( \tilde{g}_{\lambda _{0}}\right)
_{\lambda _{1}}$ be the metric on $M$ induced by the Riemannian submersion $%
q_{G\times M}:\left( G\times M,\lambda _{1}^{2}g_{\mathrm{bi}}+\tilde{g}%
_{\lambda _{0}}\right) \longrightarrow M.$ Write $Ric_{\lambda _{1}^{2}g_{%
\mathrm{bi}}+\tilde{g}_{\lambda _{0}}}^{\mathrm{Horiz}}$ for the tensor $%
Ric_{g}^{\mathrm{Horiz}}$ of Proposition \ref{Gray-O'Neill} when the
submersion is $q_{G\times M}:\left( G\times M,\lambda _{1}^{2}g_{\mathrm{bi}%
}+\tilde{g}_{\lambda _{0}}\right) \longrightarrow M.$ Let $\mathcal{K}$ be
any compact subset $M^{\text{reg}}.$ If $\lambda _{0}$ is sufficiently
small, then for all $x\in \mathcal{K}$ and for all $Z\in T_{x}G\left(
x\right) ^{\perp }$ 
\begin{equation*}
Ric_{\lambda _{1}^{2}g_{\mathrm{bi}}+\tilde{g}_{\lambda _{0}}}^{\mathrm{Horiz%
}}\left( \hat{Z},\hat{Z}\right) >\left\vert Z\right\vert _{\tilde{g}%
_{\lambda _{0}}}^{2}
\end{equation*}%
for all $\lambda _{1}\in \left( 0,\infty \right) .$
\end{proposition}

\begin{proof}
We start by studying the sectional curvature of $\tilde{g}_{\lambda _{0}}.$
For $\varepsilon >0$ as in Theorem \ref{uber conf}, $Y,Z\in TM$ we have 
\begin{equation*}
\mathrm{sec}_{\tilde{g}}\left( Y,Z\right) \geq \mathrm{sec}%
_{g}(Y,Z)-\varepsilon .
\end{equation*}

Combining this with Proposition \ref{Cheeg A--tensor} we have that for $x\in 
\mathcal{K}$ and $Y,Z\in T_{x}G\left( x\right) ^{\perp }$ 
\begin{equation}
\mathrm{sec}_{\tilde{g}_{\lambda _{0}}}\left( Y,Z\right) \geq \mathrm{sec}%
_{d\pi \left( g\right) }(d\pi ^{\mathrm{reg}}\left( Y\right) ,d\pi ^{\mathrm{%
reg}}\left( Z\right) )-2\varepsilon ,  \label{horiz Ricci 4}
\end{equation}%
provided $\lambda _{0}$ is sufficiently small.

For $V\in T_{x}G\left( x\right) $ with $\left\vert V\right\vert _{\tilde{g}%
}=1$ and $Z\in T_{x}G\left( x\right) ^{\perp }$ with $\left\vert
Z\right\vert _{\tilde{g}_{\lambda _{0}}}=\left\vert Z\right\vert _{\tilde{g}%
}=1$ 
\begin{eqnarray*}
\mathrm{sec}_{\tilde{g}_{\lambda _{0}}}\left( Ch_{\lambda _{0}}\left(
Z\right) ,Ch_{\lambda _{0}}\left( V\right) \right) &\geq &\mathrm{sec}%
_{\lambda _{0}^{2}g_{\mathrm{bi}}+\tilde{g}}\left( \hat{Z},\hat{V}\right) \\
&=&\frac{1}{\left\vert \hat{V}\right\vert _{\lambda _{0}^{2}g_{\mathrm{bi}}+%
\tilde{g}}^{2}}\mathrm{curv}_{\lambda _{0}^{2}g_{\mathrm{bi}}+\tilde{g}%
}\left( \left( 0,Z\right) ,\left( \frac{k_{V}}{\lambda _{0}^{2}},V\right)
\right) \\
&=&\frac{\mathrm{sec}_{\tilde{g}}\left( Z,V\right) }{\left\vert \hat{V}%
\right\vert _{\lambda _{0}^{2}g_{\mathrm{bi}}+\tilde{g}}^{2}}.
\end{eqnarray*}

Combining Part 2 of Proposition \ref{bound on kappa} with $\left\vert \hat{V}%
\right\vert _{\lambda _{0}^{2}g_{\mathrm{bi}}+\tilde{g}}^{2}=\frac{g_{%
\mathrm{bi}}\left( \kappa _{V},\kappa _{V}\right) }{\lambda _{0}^{2}}+1,$ we
have $\left\vert \hat{V}\right\vert _{\lambda _{0}^{2}g_{\mathrm{bi}}+\tilde{%
g}}^{2}\rightarrow \infty $ as $\lambda _{0}\rightarrow 0.$ So $\frac{%
\mathrm{sec}_{\tilde{g}}\left( Z,V\right) }{\left\vert \hat{V}\right\vert
_{\lambda _{0}^{2}g_{\mathrm{bi}}+\tilde{g}}^{2}}$ goes to $0$ uniformly on $%
M\setminus \Omega _{1}$ as $\lambda _{0}\rightarrow 0,$ since $\left\vert 
\mathrm{sec}_{\tilde{g}}\left( Z,V\right) \right\vert $ is bounded from
above by a bound that is independent of $\lambda _{0}.$

It follows that%
\begin{equation*}
\mathrm{sec}_{\tilde{g}_{\lambda _{0}}}\left( Ch_{\lambda _{0}}\left(
Z\right) ,Ch_{\lambda _{0}}\left( V\right) \right) \geq -\tau \left( \lambda
_{0}\right) ,
\end{equation*}%
where $\tau $ is as in Equation \ref{dfn of tau eqn}. Since $Ch_{\lambda
_{0}}:T_{x}M\longrightarrow T_{x}M$ is an isomorphism that preserves the
splitting $T_{x}M=T_{x}G\left( x\right) \oplus T_{x}G\left( x\right) ^{\perp
},$ we conclude that for any $Z\in T_{x}G\left( x\right) ^{\perp }$ and any $%
W\in T_{x}G\left( x\right) $ 
\begin{equation}
\mathrm{sec}_{\tilde{g}_{\lambda _{0}}}\left( Z,W\right) \geq -\tau \left(
\lambda _{0}\right) .  \label{vertizontal sec}
\end{equation}

Now let $\left\{ \hat{Z},\hat{W}_{1},\ldots ,\hat{W}_{p},\hat{Y}_{1},\ldots ,%
\hat{Y}_{m}\right\} $ be a $\left( \lambda _{1}^{2}g_{\mathrm{bi}}+\tilde{g}%
_{\lambda _{0}}\right) $--orthogonal basis for the horizontal space of our
Riemannian submersion $q_{G\times M}:\left( G\times M,\lambda _{1}^{2}g_{%
\mathrm{bi}}+\tilde{g}_{\lambda _{0}}\right) \longrightarrow M$ with $%
W_{1},\ldots ,W_{p}\in TG\left( x\right) $ and $Y_{1},\ldots ,Y_{m}\in
TG\left( x\right) ^{\perp },$ $\left\vert Z\right\vert _{\tilde{g}_{\lambda
_{0}}}=\left\vert W_{i}\right\vert _{\tilde{g}_{\lambda _{0}}}=\left\vert
Y_{i}\right\vert _{\tilde{g}_{\lambda _{0}}}=1.$ Then%
\begin{eqnarray*}
Ric_{\lambda _{1}^{2}g_{\mathrm{bi}}+\tilde{g}_{\lambda _{0}}}^{\mathrm{Horiz%
}}\left( \hat{Z},\hat{Z}\right) &=&\Sigma _{i=1}^{p}\mathrm{sec}_{\lambda
_{1}^{2}g_{\mathrm{bi}}+\tilde{g}_{\lambda _{0}}}\left( \hat{W}_{i},\hat{Z}%
\right) +\Sigma _{j=1}^{m}\mathrm{sec}_{\lambda _{1}^{2}g_{\mathrm{bi}}+%
\tilde{g}_{\lambda _{0}}}\left( \hat{Y}_{i},\hat{Z}\right) \\
&=&\Sigma _{i=1}^{p}\frac{1}{\left\vert \hat{W}_{i}\right\vert _{\lambda
_{1}^{2}g_{\mathrm{bi}}+\tilde{g}_{\lambda _{0}}}^{2}}\mathrm{curv}_{\lambda
_{1}^{2}g_{\mathrm{bi}}+\tilde{g}_{\lambda _{0}}}\left( \hat{W}_{i},\hat{Z}%
\right) +\Sigma _{j=1}^{m}\mathrm{sec}_{\lambda _{1}^{2}g_{\mathrm{bi}}+%
\tilde{g}_{\lambda _{0}}}\left( \hat{Y}_{i},\hat{Z}\right) \\
&=&\Sigma _{i=1}^{p}\frac{1}{\left\vert \hat{W}_{i}\right\vert _{\lambda
_{1}^{2}g_{\mathrm{bi}}+\tilde{g}_{\lambda _{0}}}^{2}}\mathrm{sec}_{\tilde{g}%
_{\lambda _{0}}}\left( W_{i},Z\right) +\Sigma _{j=1}^{m}\mathrm{sec}_{\tilde{%
g}_{\lambda _{0}}}\left( Y_{i},Z\right) \\
&\geq &-\Sigma _{i=1}^{p}\frac{\tau \left( \lambda _{0}\right) }{\left\vert 
\hat{W}_{i}\right\vert _{\lambda _{1}^{2}g_{\mathrm{bi}}+\tilde{g}_{\lambda
_{0}}}^{2}}+\Sigma _{j=1}^{m}\mathrm{sec}_{\tilde{g}_{\lambda _{0}}}\left(
Y_{i},Z\right) ,\text{ by \ref{vertizontal sec}} \\
&=&-\tau \left( \lambda _{0}\right) +\Sigma _{j=1}^{m}\mathrm{sec}_{\tilde{g}%
_{\lambda _{0}}}\left( Y_{i},Z\right) .
\end{eqnarray*}%
Combining this with our hypothesis that $Ric_{\left( M^{\text{reg}}/G\right)
}\geq 2$ and with Inequality \ref{horiz Ricci 4} gives 
\begin{eqnarray*}
Ric_{\lambda _{1}^{2}g_{\mathrm{bi}}+\tilde{g}_{\lambda _{0}}}^{\mathrm{Horiz%
}}\left( \hat{Z},\hat{Z}\right) &\geq &\left( -\tau \left( \lambda
_{0}\right) +\frac{3}{2}\right) \\
&>&1 \\
&=&\left\vert Z\right\vert _{\tilde{g}_{\lambda _{0}}}^{2}
\end{eqnarray*}%
as claimed.
\end{proof}

\begin{proof}[Proof of Theorem \protect\ref{generic cheeger}]
Let $\lambda _{0}$ be small enough so that the conclusion of Proposition \ref%
{horizontal cheeger} holds. An arbitrary unit vector that is horizontal for $%
q_{G\times M}:\left( G\times M,\lambda _{1}^{2}g_{\mathrm{bi}}+\tilde{g}%
_{\lambda _{0}}\right) \longrightarrow M$ has the form $\cos \sigma \hat{V}%
+\sin \sigma \hat{Z}$ where $V\in TG\left( x\right) $ and $Z\in TG\left(
x\right) ^{\perp },$ $\left\vert \hat{V}\right\vert _{\lambda _{1}^{2}g_{%
\mathrm{bi}}+\tilde{g}_{\lambda _{0}}}=\left\vert \hat{Z}\right\vert
_{\lambda _{1}^{2}g_{\mathrm{bi}}+\tilde{g}_{\lambda _{0}}}=1.$ So%
\begin{eqnarray*}
Ric_{\left( \tilde{g}_{\lambda _{0}}\right) _{\lambda _{1}}}|_{M\setminus
\Omega _{1}} &\geq &Ric_{\lambda _{1}^{2}g_{\mathrm{bi}}+\tilde{g}_{\lambda
_{0}}}^{\mathrm{Horiz}}\left( \cos \sigma \hat{V}+\sin \sigma \hat{Z},\cos
\sigma \hat{V}+\sin \sigma \hat{Z}\right) \\
&=&\cos ^{2}\sigma Ric_{\lambda _{1}^{2}g_{\mathrm{bi}}+\tilde{g}_{\lambda
_{0}}}^{\mathrm{Horiz}}\left( \hat{V},\hat{V}\right) +\sin 2\sigma
Ric_{\lambda _{1}^{2}g_{\mathrm{bi}}+\tilde{g}_{\lambda _{0}}}^{\mathrm{Horiz%
}}\left( \hat{V},\hat{Z}\right) +\sin ^{2}\sigma Ric_{\lambda _{1}^{2}g_{%
\mathrm{bi}}+\tilde{g}_{\lambda _{0}}}^{\mathrm{Horiz}}\left( \hat{Z},\hat{Z}%
\right)
\end{eqnarray*}

By Proposition \ref{vertizontal Ricci} we have $\left\vert Ric_{\lambda
_{1}^{2}g_{\mathrm{bi}}+\tilde{g}_{\lambda _{0}}}^{\mathrm{Horiz}}\left( 
\hat{V},\hat{Z}\right) \right\vert <\frac{1}{100},$ provided $\lambda _{1}$
is sufficiently small. If we choose the constant $C$ in Proposition \ref%
{vertical Ricci} to be $100,$ and apply Proposition \ref{horizontal cheeger}
we then get 
\begin{eqnarray*}
Ric_{\left( \tilde{g}_{\lambda _{0}}\right) _{\lambda _{1}}}|_{M\setminus
\Omega _{1}} &\geq &100\cos ^{2}\sigma +\frac{1}{100}\sin 2\sigma +\sin
^{2}\sigma \\
&>&\frac{99}{100},
\end{eqnarray*}%
proving Theorem \ref{generic cheeger} and Theorem \ref{main}.
\end{proof}

\section{Lifting Almost Non-negative Curvature\label{alm nonneg lift}}

Throughout this section we assume that $G$ is a compact, connected Lie group
acting isometrically and effectively on a family of compact Riemannian
manifolds $\left( M,g_{\alpha }\right) $. We further assume that the
quotient, $\left( M/G,\mathrm{dist}_{\alpha }\right) $, is an almost
non-negatively curved family of metric spaces.

We will obtain the almost non-negatively curved family of metrics on $M$ via
a sequence of three deformations, as follows.

\noindent \textbf{Step 1:} Apply Theorem \ref{Step 1} to $\left( M,g_{\alpha
}\right) ,$ yielding the Cheeger deformed metrics $\left( M,\left( g_{\alpha
}\right) _{l}\right) .$

\noindent \textbf{Step 2:} Apply Theorem \ref{Step 2} to obtain a family of $%
G$--invariant metrics $\widetilde{\left( g_{\alpha }\right) _{l}}$ of the
form $\widetilde{\left( g_{\alpha }\right) _{l}}=e^{2f_{\alpha }}\left(
g_{\alpha }\right) _{l},$ for appropriate smooth functions $f_{\alpha
}:M\longrightarrow \mathbb{R}.$

\noindent \textbf{Step 3:} Apply a further Cheeger deformation to $\left( M,%
\widetilde{\left( g_{\alpha }\right) _{l}}\right) $ to obtain $\left(
M,\left( \widetilde{\left( g_{\alpha }\right) _{l}}\right) _{\lambda
}\right) ,$ which we will show is an almost non-negatively curved family in
Theorem \ref{step 3} below.

\begin{remark}
It seems possible that Theorem \ref{alm nonneg thm} could be proven
performing these deformations in another order. The order we have chosen
allows the argument to be broken into several smaller, separately verifiable
pieces.
\end{remark}

As the diameter bound is much easier to establish we discuss it first. Since 
$\left\{ \left( M/G,\mathrm{dist}_{\alpha }\right) \right\} _{\alpha
=1}^{\infty }$ is an almost non-negatively curved family, 
\begin{equation*}
\mathrm{Diam}\left( M/G,\mathrm{dist}_{\alpha }\right) \leq D
\end{equation*}%
for some $D>0.$

Let $\left( \mathrm{dist}_{\alpha }\right) _{l}$ be the orbital metric on $%
M/G$ induced by $\left( g_{\alpha }\right) _{l}.$ Since a Cheeger
deformation does not change the metric on the distribution that is
orthogonal to the orbits, 
\begin{equation*}
\mathrm{Diam}\left( M/G,\left( \mathrm{dist}_{\alpha }\right) _{l}\right)
\leq D.
\end{equation*}

Let $\widetilde{\left( \mathrm{dist}_{\alpha }\right) _{l}}$ be the orbital
metric on $M/G$ induced by $\widetilde{\left( g_{\alpha }\right) _{l}}.$ By
Remark \ref{uber c-1 close}, our conformal factor, $e^{2f},$ is as close as
we please in the $C^{0}$--topology to $1$. In particular, we can easily
arrange that 
\begin{equation*}
\mathrm{Diam}\left( M/G,\widetilde{\left( \mathrm{dist}_{\alpha }\right) _{l}%
}\right) \leq 2D.
\end{equation*}

Finally, $\left( M,\left( \widetilde{\left( g_{\alpha }\right) _{l}}\right)
_{\lambda }\right) $ converges to $\left( M/G,\widetilde{\left( \mathrm{dist}%
_{\alpha }\right) _{l}}\right) $ in the Gromov--Hausdorff topology as $%
\lambda \rightarrow 0$, so 
\begin{equation*}
\mathrm{Diam}\left( M,\left( \widetilde{\left( g_{\alpha }\right) _{l}}%
\right) _{\lambda }\right) \leq 3D,
\end{equation*}%
provided $\lambda $ is sufficiently small.

Thus, to prove Theorem \ref{alm nonneg thm} it suffices to show that there
is a sequence of positive numbers, $\left\{ \varepsilon _{\alpha }\right\}
_{\alpha =1}^{\infty },$ Cheeger parameters $l,\lambda ,$ and $G$--invariant
conformal factors $e^{2f_{\alpha }}$ so that%
\begin{equation*}
\varepsilon _{\alpha }\rightarrow 0\text{ as }\alpha \rightarrow \infty ,
\end{equation*}%
\begin{equation*}
\mathrm{sec}\left( M,\left( \widetilde{\left( g_{\alpha }\right) _{l}}%
\right) _{\lambda }\right) \geq -\varepsilon _{\alpha },\text{ and}
\end{equation*}%
$e^{2f_{\alpha }}$ is $C^{0}$--close to $1.$ This in turn follows from the
next three results.

Applying Theorem \ref{Step 1} gives us

\begin{corollary}
\label{Step 1 cor}Let $M$ and $g_{\alpha }$ be as in Theorem \ref{alm nonneg
thm}. For any $\varepsilon >0$ there is an $\alpha _{0}\in \mathbb{N}$ so
that for all $\alpha \geq \alpha _{0},$ there is a neighborhood $\Omega
^{\prime }\left( \alpha \right) $ of $S_{1}\cup S_{2}\cup \cdots \cup S_{p},$
and a Cheeger parameter $l_{1}\left( \alpha \right) $ such that for all $%
l\in \left( 0,l_{1}\left( \alpha \right) \right) $ 
\begin{equation}
\mathrm{sec}_{\left( g_{\alpha }\right) _{l}}(Y,Z)\geq -\frac{\varepsilon }{4%
}\text{ \label{Alm on H--B}}
\end{equation}%
if either $Y,Z\in TG(x)^{\perp }|_{M\setminus \Omega ^{\prime }\left( \alpha
\right) }$ or $Y,Z\in \left\{ TG(x)^{\perp }\cap \overline{\mathcal{H}}%
^{i}\right\} |_{\Omega ^{\prime }\left( \alpha \right) \setminus S_{1}\cup
S_{2}\cup \cdots \cup S_{p}}$ for some $i\in \left\{ 1,\ldots ,p\right\} $.

Here $\left( g_{\alpha }\right) _{l}$ is the metric on $M$ induced by the
Riemannian submersion 
\begin{equation*}
q_{G\times M}:\left( G\times M,l^{2}g_{\mathrm{bi}}+g_{\alpha }\right)
\longrightarrow M.
\end{equation*}
\end{corollary}

Applying Theorem \ref{Step 2} gives us

\begin{corollary}
\label{Step 2 Cor} Let $\left( g_{\alpha }\right) _{l}$ be a metric that
satisfies the conclusion of Corollary \ref{Step 1 cor}$.$ For any $%
K,\varepsilon >0$ there is a neighborhood $\Omega _{1}\left( \alpha \right) $
of $S_{1}\cup S_{2}\cup \cdots \cup S_{p}$ and a metric $\widetilde{\left(
g_{\alpha }\right) _{l}}=e^{2f_{\alpha }}\left( g_{\alpha }\right) _{l}$ so
that if $V\in \mathrm{span}\left\{ \mathcal{V}^{i},X^{i}\right\} |_{\Omega
_{1}\left( \alpha \right) }$ for some $i\in \left\{ 1,\ldots ,p\right\} ,$
then 
\begin{equation*}
\mathrm{sec}_{\widetilde{\left( g_{\alpha }\right) _{l}}}(V,W)\geq K\text{ }
\end{equation*}%
for all $W\in T\Omega _{1}\left( \alpha \right) ,$ and 
\begin{equation*}
\mathrm{sec}_{\widetilde{\left( g_{\alpha }\right) _{l}}}(V,W)\geq \mathrm{%
sec}_{\left( g_{\alpha }\right) _{l}}(V,W)-\frac{\varepsilon }{4}.\text{ %
\label{Alm on Omega'' 2}}
\end{equation*}%
for all $V,W\in TM.$
\end{corollary}

Theorem \ref{alm nonneg thm} follows from the next result.

\begin{theorem}
\label{step 3}For any $\varepsilon >0$, let $\widetilde{\left( g_{\alpha
}\right) _{l}}$ be a metric that satisfies the conclusion of Corollary \ref%
{Step 2 Cor}$.$ There is an $l_{2}>0$ so that for all $\lambda \in \left(
0,l_{2}\right) $ 
\begin{equation*}
\mathrm{sec}_{\left( \widetilde{\left( g_{\alpha }\right) _{l}}\right)
_{\lambda }}\geq -\varepsilon ,
\end{equation*}%
where $\left( \widetilde{\left( g_{\alpha }\right) _{l}}\right) _{\lambda }$
is the metric on $M$ induced by the Riemannian submersion 
\begin{equation*}
q_{G\times M}:\left( G\times M,\lambda ^{2}g_{\mathrm{bi}}+\widetilde{\left(
g_{\alpha }\right) _{l}}\right) \longrightarrow M.
\end{equation*}
\end{theorem}

\begin{proof}
Given $\varepsilon >0,$ from Corollaries \ref{Step 1 cor} and \ref{Step 2
Cor} we have that there is a metric $\widetilde{\left( g_{\alpha }\right)
_{l}}$ so that 
\begin{equation}
\mathrm{sec}_{\widetilde{\left( g_{\alpha }\right) _{l}}}\left( V,W\right) >-%
\frac{\varepsilon }{2}  \label{step 2 curv}
\end{equation}%
if $V,W\in \mathrm{span}\left\{ TG(x)^{\perp }\cap \overline{\mathcal{H}}%
^{i},\mathcal{V}^{i},X^{i}\right\} |_{\Omega _{1}\left( \alpha \right) }$
for some $i\in \left\{ 1,\ldots ,p\right\} $ or $V,W\in TG\left( x\right)
^{\perp }$ and $x\in $ $M\setminus \Omega _{1}\left( \alpha \right) .$

By continuity, Inequality \ref{step 2 curv} continues to hold on some
neighborhood $U_{\mathrm{Sing}}$ of the set of planes spanned by vectors in%
\begin{equation*}
\bigcup\limits_{i=1}^{p}\cup _{x\in \Omega ^{i}}\mathrm{span}\left\{ 
\overline{\mathcal{H}}^{i}\cap TG\left( x\right) ^{\perp },\mathcal{V}%
^{i},X^{i}\right\} \bigcup \cup _{x\in \mathcal{C}_{i}}\mathrm{span}\left\{
T_{x}\left( S_{i}\right) \cap TG\left( x\right) ^{\perp },\nu _{x}\left(
S_{i}\right) \right\} .
\end{equation*}%
Continuity also gives Inequality \ref{step 2 curv} on some neighborhood $U_{%
\mathrm{Gen}}$ of the set of planes,%
\begin{equation*}
\left\{ \left. P\right\vert \text{ }P=\mathrm{span}\left\{ V,W\right\} ,%
\text{ }V,W\in TG\left( x\right) ^{\perp }|_{M\setminus \Omega _{1}\left(
\alpha \right) }\right\} .
\end{equation*}%
For simplicity from this point forward we set 
\begin{equation*}
g\equiv \widetilde{\left( g_{\alpha }\right) _{l}}.
\end{equation*}%
Then by Proposition \ref{orbital estimate}, 
\begin{equation*}
\mathrm{sec}_{g_{\lambda }}\left( Ch_{\lambda }\left( P\right) \right)
>-\varepsilon
\end{equation*}%
for all planes $P\in U_{\mathrm{Sing}}\cup U_{\mathrm{Gen}}.$ So we only
have to verify the same inequality for planes in the complement of $U_{%
\mathrm{Sing}}\cup U_{\mathrm{Gen}}.$

Let $P$ be any plane in the complement of $U_{\mathrm{Sing}}$ with footpoint
in $\Omega _{1}\left( \alpha \right) $ of the form $P=\mathrm{span}\left\{
V,W\right\} $ with $V$ and $W$ orthonormal with respect to $g$. By Corollary %
\ref{kappa not 0} there is a $c>0$ so that%
\begin{equation*}
\max \left\{ \left\vert \kappa _{W}\right\vert _{g_{\mathrm{bi}%
}}^{2},\left\vert \kappa _{V}\right\vert _{g_{\mathrm{bi}}}^{2}\right\} \geq
c.
\end{equation*}%
Similarly, let $P$ be any plane in the complement of $U_{\mathrm{Gen}}$ with
footpoint in $M\setminus \Omega _{1}\left( \alpha \right) $ of the form $P=%
\mathrm{span}\left\{ V,W\right\} $ with $V$ and $W$ orthonormal with respect
to $g$. By Part 2 of Proposition \ref{bound on kappa}, there is a (perhaps
different) constant $c>0$ so that%
\begin{equation*}
\max \left\{ \left\vert \kappa _{W}\right\vert _{g_{\mathrm{bi}%
}}^{2},\left\vert \kappa _{V}\right\vert _{g_{\mathrm{bi}}}^{2}\right\} \geq
c.
\end{equation*}%
So, for any plane in the complement of $U_{\mathrm{Sing}}\cup U_{\mathrm{Gen}%
},$ 
\begin{equation}
\text{\textrm{max}}\left\{ -\frac{1}{\left\vert \kappa _{W}\right\vert _{g_{%
\mathrm{bi}}}^{2}},-\frac{1}{\left\vert \kappa _{V}\right\vert _{g_{\mathrm{%
bi}}}^{2}}\right\} \geq -\frac{1}{c}.  \label{BIG GROUP}
\end{equation}

On the other hand, since $M$ is compact, 
\begin{equation}
\left\vert \mathrm{sec}_{g}\right\vert \leq K_{1}  \label{BOUNDED}
\end{equation}%
for some $K_{1}\in \mathbb{R}.$

Now consider a plane $P=\mathrm{span}\left\{ V,W\right\} $ in the complement
of $U_{\mathrm{Sing}}\cup U_{\mathrm{Gen}}.$ Combining Inequalities \ref%
{BOUNDED} and \ref{BIG GROUP} with the estimate%
\begin{equation*}
\mathrm{sec}_{g_{l}}\left( Ch_{l}\left( W\right) ,Ch_{l}\left( V\right)
\right) \geq \max \left\{ -\frac{\lambda ^{2}}{\left\vert \kappa
_{V}\right\vert _{g_{\mathrm{bi}}}^{2}},-\frac{\lambda ^{2}}{\left\vert
\kappa _{W}\right\vert _{g_{\mathrm{bi}}}^{2}}\right\} \left\vert \mathrm{sec%
}_{g_{M}}\left( V,W\right) \right\vert
\end{equation*}%
from Proposition \ref{orbital estimate} gives us 
\begin{equation*}
\mathrm{sec}_{g_{\lambda }}\left( Ch_{\lambda }\left( W\right) ,Ch_{\lambda
}\left( V\right) \right) >-\varepsilon ,
\end{equation*}%
provided $\lambda $ is sufficiently small, and hence proves Theorem \ref%
{step 3}.
\end{proof}

Finally, since the deformations used to prove Theorems \ref{main} and \ref%
{alm nonneg thm} are the same, Theorem \ref{supplement} follows by combining
the proofs of Theorems \ref{main} and \ref{alm nonneg thm}.

\section{\label{examples section}Examples}

The proof of Theorem \ref{examples} is based on Davis' $SO\left( 3\right) $%
--actions on the class $\Sigma ^{7}$, his $G_{2}$--actions on the class $%
\Sigma _{BP}^{15}$ \cite{Dav}, and the following proposition.

\begin{proposition}
\label{Lifting Prop}Let $\left( M_{1},G\right) $ and $\left( M_{2},G\right) $
be smooth $n$--dimensional $G$--manifolds with $G$ a compact Lie group, and $%
M_{1}/G=M_{2}/G=X.$ In addition, suppose that the group diagrams and
isotropy representations for $\left( M_{1},G\right) $ and $\left(
M_{2},G\right) $ are the same when parameterized by $X.$

Let $\left( X,\mathrm{dist}_{1}\right) $ be the quotient of a $G$--invariant
Riemannian metric, $g_{1},$ on $M_{1}.$ Then $\left( X,\mathrm{dist}%
_{1}\right) $ is also the quotient of a $G$--invariant Riemannian metric, $%
g_{2},$ on $M_{2}.$
\end{proposition}

\begin{proof}
Let $\pi _{1}:M_{1}\longrightarrow X$ and $\pi _{2}:M_{2}\longrightarrow X$
be the quotient maps. By the Slice Theorem for each $x\in X,$ there is a
neighborhood $N_{x}$ and a $G$--equivariant diffeomorphism%
\begin{equation*}
\Phi _{x}:\pi _{2}^{-1}\left( N_{x}\right) \longrightarrow \pi
_{1}^{-1}\left( N_{x}\right)
\end{equation*}%
so that 
\begin{equation*}
\pi _{1}\circ \Phi _{x}=\pi _{2}.
\end{equation*}%
$\Phi _{x}^{\ast }\left( g_{1}\right) $ is a $G$--invariant metric on $\pi
_{2}^{-1}\left( N_{x}\right) $ whose orbital distance metric is $\mathrm{dist%
}_{1}|_{N_{x}}.$

We glue the metrics $\Phi _{x_{i}}^{\ast }\left( g_{1}\right) $ together
with a $G$--invariant partition of unity subordinate to$\left\{
N_{x}\right\} _{x\in X}$, yielding a $G$--invariant metric $g_{2}.$ The
quotient $\left( M_{2},g_{2}\right) /G$ is $\left( X,\mathrm{dist}%
_{1}\right) $ since for all $i,$ $\left( \pi _{2}^{-1}\left(
N_{x_{i}}\right) ,\Phi _{x_{i}}^{\ast }\left( g_{1}\right) \right) /G$ is $%
\left( N_{x_{i}},\mathrm{dist}_{1}\right) .$
\end{proof}

The key point for Davis' actions is that $SO\left( 3\right) $ and $G_{2}$
are the group of automorphisms of the quaternion and octonion division
algebras, respectively. Davis starts by defining the actions on the subsets
of $\Sigma ^{7}$ and $\Sigma _{BP}^{15}$ that are $S^{3}$--bundles over $%
S^{4}$ and $S^{7}$--bundles over $S^{8},$ respectively.

Writing $\mathbb{F}$ for either $\mathbb{H}$ or $\mathbb{O},$ and $b$ for
the real dimension of $\mathbb{F},$ recall that the $S^{b-1}$--bundles over $%
S^{b}$ with structure group $SO\left( b\right) $ are classified by $\mathbb{Z%
}\oplus \mathbb{Z}$ as follows.

The total space of the bundle $p_{m,n}:E_{m,n}\longrightarrow S^{b}$ is
obtained by gluing together two copies of $\mathbb{F}\times S^{b-1}$ via%
\begin{equation*}
\Phi _{m,n}:\left( u,v\right) \longmapsto \left( \frac{u}{\left\vert
u\right\vert ^{2}},\frac{u^{m}}{\left\vert u\right\vert ^{m}}v\frac{u^{n}}{%
\left\vert u\right\vert ^{n}}\right) =\left( u^{\prime },v^{\prime }\right) .
\end{equation*}%
To describe the map $p_{m,n}:E_{m,n}\longrightarrow S^{b},$ we view $S^{b}$
as the disjoint union of two copies of $\mathbb{F}$ that are glued together
along $\mathbb{F}\setminus \left\{ 0\right\} $ via $\phi :\mathbb{F}%
\setminus \left\{ 0\right\} \longrightarrow \mathbb{F}\setminus \left\{
0\right\} ,$ $\phi \left( u\right) =\frac{u}{\left\vert u\right\vert ^{2}}.$
The map $p_{m,n}:E_{m,n}\longrightarrow S^{b}$ is then given by projecting
onto the first factor of either copy of $\mathbb{F}\times S^{b-1}.$

Let $G$ stand for either $SO\left( 3\right) $ or $G_{2}$ and observe that $G$
acts by automorphisms of $\mathbb{F}.$ So by letting $G$ act diagonally on
both copies of $\mathbb{F}\times S^{b-1}$ 
\begin{equation}
g\left( u,v\right) =\left( g\left( u\right) ,g\left( v\right) \right)
\label{Davis}
\end{equation}%
we get a well defined $G$--action on $E_{m,n}.$ In the quaternionic case,
when $m+n=\pm 1,$ Milnor constructed a Morse function on $E_{m,n}$ with only
two critical points and concluded that $E_{m,n}$ is homeomorphic to $%
S^{2b-1} $ \cite{Mil}, and Shimada carried out the analogous program in the
octonionic case, also when $m+n=\pm 1$ \cite{Shim}. Davis observed that the
Morse functions constructed by Milnor and Shimada are invariant under the $G$%
--action, and concluded that $E_{m,n}$ is $G$--equivariantly homeomorphic to 
$S^{2b-1}.$ In particular, $E_{m,n}/G$ is homeomorphic to $S^{2b-1}/G.$

It is easy to see the following.

\begin{proposition}
The Action \ref{Davis} is by symmetries of $p_{m,n}:E_{m,n}\longrightarrow
S^{b},$ and has three orbit types. In the quaternion case the isotropies are$%
\vspace*{0.1in}$:

\noindent 1. Trivial when $uv-vu\neq 0\vspace*{0.1in},$

\noindent 2. $SO\left( 2\right) $ when $uv-vu=0,$ but either $Im\left(
v\right) \neq 0$ or $Im\left( u\right) \neq 0\vspace*{0.1in},$

\noindent 3. $SO\left( 3\right) $ when $Im\left( v\right) =Im\left( u\right)
=0.\vspace*{0.1in}$

In the octonion case the isotropies are$\vspace*{0.1in}$

\noindent 1. $SU\left( 2\right) $ when $uv-vu\neq 0,\vspace*{0.1in}$

\noindent 2. $SU\left( 3\right) $ when $uv-vu=0,$ but either $Im\left(
v\right) \neq 0$ or $Im\left( u\right) \neq 0,\vspace*{0.1in}$

\noindent 3. $G_{2}$ when $Im\left( v\right) =Im\left( u\right) =0.$
\end{proposition}

\begin{proposition}
\label{model action}The $G$--action on $E_{1,0}=S^{2b-1}$ is $G$%
--equivariantly diffeomorphic to an orthogonal action. It induces a $G$%
--action on $\mathbb{F}P^{2}\#-\mathbb{F}P^{2}.$
\end{proposition}

\begin{proof}
We prove the first statement by constructing explicit coordinate charts that
identify $S^{2b-1}$ with $E_{1,0}$ and for which the corresponding action on 
$S^{2b-1}$ is 
\begin{equation}
\left( g,\left( 
\begin{array}{c}
a \\ 
b%
\end{array}%
\right) \right) \longmapsto \left( 
\begin{array}{c}
g\left( a\right) \\ 
g\left( b\right)%
\end{array}%
\right) ,  \label{conj}
\end{equation}%
where we view $S^{2b-1}$ as the unit sphere in $\mathbb{F}\oplus \mathbb{F}$%
, and $G$ is acting by automorphisms of $\mathbb{F}.$

The coordinate charts are constructed as in \cite{GromMey} or \cite{Wilh}.
Let $\phi :\mathbb{F}\longrightarrow \mathbb{R}$ be 
\begin{equation*}
\phi (u)=\frac{1}{\sqrt{1+|u|^{2}}}.
\end{equation*}%
The charts $h_{1},h_{2}:\mathbb{F}\times S^{b-1}\longrightarrow S^{2b-1}$
are defined by 
\begin{equation*}
h_{1}(u,q)=\left( 
\begin{array}{c}
uq \\ 
q%
\end{array}%
\right) \phi (u)
\end{equation*}%
and 
\begin{equation*}
h_{2}(v,r)=\left( 
\begin{array}{c}
r \\ 
\bar{v}r%
\end{array}%
\right) \phi (v).
\end{equation*}%
The charts $h_{1}$ and $h_{2}$ are embeddings onto the open dense sets 
\begin{equation*}
U_{1}=\left\{ \left. \left( 
\begin{array}{c}
a \\ 
c%
\end{array}%
\right) \right\vert ,\ \ c\not=0\right\} \;\;
\end{equation*}%
and%
\begin{equation*}
U_{2}=\left\{ \left. \left( 
\begin{array}{c}
a \\ 
c%
\end{array}%
\right) \right\vert \ a\not=0\right\}
\end{equation*}%
respectively. In fact, the formulas for the inverses are given by 
\begin{equation*}
h_{1}^{-1}\left( 
\begin{array}{c}
a \\ 
c%
\end{array}%
\right) =(\frac{a\bar{c}}{|c|^{2}},\frac{c}{\left\vert c\right\vert }),
\end{equation*}%
\begin{equation*}
h_{2}^{-1}\left( 
\begin{array}{c}
a \\ 
c%
\end{array}%
\right) =(\frac{a\bar{c}}{\left\vert a\right\vert ^{2}},\frac{a}{\left\vert
a\right\vert }).
\end{equation*}%
It follows that 
\begin{eqnarray*}
h_{2}^{-1}\circ h_{1}(u,q) &=&h_{2}^{-1}\left( \left( 
\begin{array}{c}
uq \\ 
q%
\end{array}%
\right) \phi (u)\right) \\
&=&\left( \frac{uq\bar{q}\phi (u)^{2}}{\left\vert u\right\vert ^{2}\phi
(u)^{2}},\frac{uq\phi (u)}{\left\vert u\right\vert \phi (u)}\right) \\
&=&\left( \frac{u}{\left\vert u\right\vert ^{2}},\frac{u}{\left\vert
u\right\vert }q\right) .
\end{eqnarray*}%
In the case of $E_{1,0},$ it follows that the action in \ref{Davis} is $G$%
--equivariantly diffeomorphic to the isometric action on $S^{2b-1}$ given by %
\ref{conj}.

Since $G$ acts by symmetries of the Hopf fibration, $h:S^{2b-1}%
\longrightarrow S^{b},$ we get a well defined $G$--action on the double
mapping cylinder of the Hopf fibration, 
\begin{equation}
\mathbb{F}P^{1}\cup _{h}\left\{ \left( 0,\frac{\pi }{2}\right) \times
S^{2b-1}\right\} \cup _{h}\mathbb{F}P^{1},  \label{mapping cylinder}
\end{equation}%
that is, on $\mathbb{F}P^{2}\#-\mathbb{F}P^{2}.$
\end{proof}

To get $G$--actions on the other elements of $\Sigma ^{7}$ and $\Sigma
_{BP}^{15}$ we note that, as observed by Kervaire and Milnor, $\Sigma ^{7}$
is a cyclic group of order $28,$ and $\Sigma _{BP}^{15}$ is a cyclic group
of order $8,128.$ In both cases $E_{2,-1}$ generates the cyclic group (see
page 69 of \cite{Dav} and pages 101 and 106 of \cite{Ell-Kup}). As observed
by Davis, the fixed point set of the $G$--action is a circle. At a fixed
point, we take the equivariant connected sum of $E_{2,-1}$ with itself. This
produces a $G$--action on $2E_{2,-1}\equiv E_{2,-1}\#E_{2,-1},$ which is
equivariantly homeomorphic to the standard $G$--action given in Proposition %
\ref{model action}. Since $E_{2,-1}$ generates the cyclic groups, $\Sigma
^{7}$ and $\Sigma _{BP}^{15},$ we iterate this construction to obtain a $G$%
--action on each member of $\Sigma ^{7}$ and $\Sigma _{BP}^{15}$ that is
equivariantly homeomorphic to the standard $G$--action. In particular, each $%
G$--action has the same orbit space, group diagram and isotropy
representation as the standard model.

We can therefore apply Proposition \ref{Lifting Prop} with $%
M_{1}=E_{1,0}=S^{2b-1}$ and the standard $G$--action and $M_{2}$ an
arbitrary element of $\Sigma ^{7}$ or $\Sigma _{BP}^{15}$ with the $G$%
--action from above. This yields a $G$--invariant metric on $M_{2}$ whose
quotient is positively curved.

We now apply Theorem \ref{supplement} to obtain a family of $G$--invariant
metrics with positive Ricci curvature that are also almost non-negatively
curved on each element of $\Sigma ^{7}$ and $\Sigma _{BP}^{15}.$

\subsection{Fake $\mathbb{F}\mathbf{P}^{\mathbf{2}}\mathbf{\#-}\mathbb{F}%
\mathbf{P}^{\mathbf{2}}$s}

Let $M^{2b}$ be the double mapping cylinder on 
\begin{equation*}
p_{m,n}:E_{m,n}\longrightarrow S^{b}
\end{equation*}%
where $m+n=\pm 1.$ Since $G$ acts by symmetries of $p_{m,n}$ we get a smooth 
$G$--action on $M^{2b}.$

In the case when $\left( m,n\right) =\left( 1,0\right) $ we get the
connected sum of the standard projective plane with its negative, $\mathbb{F}%
P^{2}\#-\mathbb{F}P^{2},$ with a $G$--action.

As before, the two $G$--spaces $\mathbb{F}P^{2}\#-\mathbb{F}P^{2}$ and $%
M^{2b}$ have the same orbit space, group diagrams and isotropy
representations. So, as before, we will apply Proposition \ref{Lifting Prop}
and Theorem \ref{supplement} to obtain a family of $G$--invariant metrics
with positive Ricci curvature that are also almost non-negatively curved. To
do this we need a $G$--invariant metric on $\mathbb{F}P^{2}\#-\mathbb{F}%
P^{2} $ for which, $\left( \mathbb{F}P^{2}\#-\mathbb{F}P^{2}\right) /G$ has
both almost non-negative and positive Ricci curvature.

Cheeger constructed non-negatively curved metrics on $\mathbb{F}P^{2}\#-%
\mathbb{F}P^{2}$ by gluing together two copies of the Hopf disk bundles that
correspond to the Hopf fibration $h=p_{1,0}:E_{1,0}\longrightarrow S^{b}$, 
\cite{Cheeg}. To do the gluing, Cheeger constructed metrics that are
products near the boundaries. Consequently, his metrics on $\mathbb{F}%
P^{2}\#-\mathbb{F}P^{2}$ and $\left( \mathbb{F}P^{2}\#-\mathbb{F}%
P^{2}\right) /G$ have $Ric\geq 0$, but not $Ric>0.$ The zero Ricci
curvatures occur for the field $X$ that is the gradient of the distance from
the boundary of either disk bundle. Moreover, lower Ricci curvature bounds
need not be preserved by Riemannian submersions \cite{ProWilh2}, so the
verification of positive Ricci curvature on $\left( \mathbb{F}P^{2}\#-%
\mathbb{F}P^{2}\right) /G$ requires additional calculation and a minor
modification of Cheeger's metric.

Since the case of $\mathbb{H}P^{2}\#-\mathbb{H}P^{2}$ is essentially known
by combining the results of \cite{Guij} and \cite{GrovZil1}, we will only
discuss the case of $\mathbb{O}P^{2}\#-\mathbb{O}P^{2}$ explicitly, noting
that similar methods will also apply to the fake $\mathbb{H}P^{2}\#-\mathbb{H%
}P^{2}$s.

Fortunately, it is straightforward to modify Cheeger's construction to
obtain $G$-invariant metrics on $\mathbb{O}P^{2}\#-\mathbb{O}P^{2}$ for
which the quotient metrics on $\left( \mathbb{O}P^{2}\#-\mathbb{O}%
P^{2}\right) /G$ are $Ric>0$ with non-negative curvature.

The bi-quotient approach indicated by Totaro, \cite{Tata}, provides the
means to achieve this with minimal calculations. Totaro observed that $%
\mathbb{O}P^{2}\#-\mathbb{O}P^{2}$ is the quotient of a $Spin\left( 8\right) 
$ action on $Spin\left( 9\right) \times S^{8}.$ Give $Spin\left( 9\right)
\times S^{8}$ the product metric. Let $Spin\left( 8\right) $ act on $%
Spin\left( 9\right) $ on the right. Let $S^{7}\subset S^{8},$ where we view $%
S^{7}$ and $S^{8}$ as the unit spheres in $\mathbb{R}^{8}\subset \mathbb{R}%
^{9},$ respectively. We suspend the standard $Spin(8)$ action on the $S^{7}$
to get a $Spin(8)$ action on $S^{8},$ and denote the fixed points by $\pm
e_{9}.$ We set 
\begin{equation*}
t\equiv \mathrm{dist}_{S^{8}}\left( e^{9},\cdot \right) ,
\end{equation*}%
and 
\begin{equation*}
X\equiv \mathrm{grad}\left( \mathrm{dist}_{S^{8}}\left( e^{9},\cdot \right)
\right) .
\end{equation*}%
We write points in $S^{8}\setminus \left\{ \pm e_{9}\right\} $ as $\left(
x,t\right) \in S^{7}\times \left( 0,\pi \right) .$

We then get a free $Spin(8)$--action on $Spin\left( 9\right) \times S^{8},$
and call the quotient map 
\begin{equation*}
q:Spin\left( 9\right) \times S^{8}\longrightarrow \left( Spin\left( 9\right)
\times S^{8}\right) /Spin\left( 8\right) .
\end{equation*}%
As observed in \cite{Tata}, $\left( Spin\left( 9\right) \times S^{8}\right)
/Spin\left( 8\right) $ is diffeomorphic to $\mathbb{O}P^{2}\#-\mathbb{O}%
P^{2}.$ To see this we first point out:

\begin{proposition}
\label{S^15}$\left( Spin\left( 9\right) \times S^{7}\right) /Spin\left(
8\right) $ is diffeomorphic to $Spin\left( 9\right) /Spin\left( 7\right) $,
which, in turn, is diffeomorphic to $S^{15}.$
\end{proposition}

\begin{proof}
We identify $S^{7}$ with $Spin\left( 8\right) /Spin\left( 7\right) $ and
write elements of $S^{7}$ as $\sigma Spin\left( 7\right) $ with $\sigma \in
Spin\left( 8\right) .$ This gives us a diffeomorphism $\Phi :\left(
Spin\left( 9\right) \times S^{7}\right) /Spin\left( 8\right) \longrightarrow
Spin\left( 9\right) /Spin\left( 7\right) ,$ 
\begin{equation*}
\Phi :\left( A,\sigma Spin\left( 7\right) \right) \cdot Spin\left( 8\right)
\longmapsto \left( A\sigma \right) Spin\left( 7\right) .
\end{equation*}

Finally, we identity $Spin\left( 9\right) /Spin\left( 7\right) $ with $%
S^{15},$ since, as was shown in \cite{GluWarZil}, $Spin\left( 9\right) $ is
the symmetry group of the octonionic Hopf fibration, $S^{15}\longrightarrow
S^{8},$ and the isotropy is $Spin\left( 7\right) .$
\end{proof}

In terms of $Spin\left( 8\right) $ cosets, the octonionic Hopf fibration is 
\begin{eqnarray*}
\left( Spin\left( 9\right) \times S^{7}\right) /Spin\left( 8\right)
&\longrightarrow &Spin\left( 9\right) /Spin\left( 8\right) \\
\left( A,v\right) \cdot Spin\left( 8\right) &\longmapsto &A\cdot Spin\left(
8\right) .
\end{eqnarray*}%
This yields

\begin{proposition}
\label{Bi-quot-Op2+ OP2}$\left( Spin\left( 9\right) \times S^{8}\right)
/Spin\left( 8\right) $ is diffeomorphic to the double mapping cylinder of
the octonionic Hopf fibration, which in turn is diffeomorphic to $\mathbb{O}%
P^{2}\#-\mathbb{O}P^{2}$.
\end{proposition}

We complete the proof of Theorem \ref{examples} by showing the following
result, whose proof occupies the rest of the paper.

\begin{theorem}
\label{Ricc on quotient}Give $\mathbb{O}P^{2}\#-\mathbb{O}P^{2}$ the
quotient metric, $g_{q},$ induced from the Riemannian submersion%
\begin{equation*}
q:Spin\left( 9\right) \times S^{8}\longrightarrow \left( Spin\left( 9\right)
\times S^{8}\right) /Spin\left( 8\right) =\mathbb{O}P^{2}\#-\mathbb{O}P^{2}.
\end{equation*}%
Then the regular part of the quotient of the $G_{2}$--action on $\mathbb{O}%
P^{2}\#-\mathbb{O}P^{2}$ has uniformly positive Ricci curvature.
\end{theorem}

To describe the horizontal space of $q$ at points of the form $\left(
A,\left( x,t\right) \right) \in Spin\left( 9\right) \times \left\{
S^{8}\setminus \left\{ \pm e_{9}\right\} \right\} $ we note that at $x\in
S^{7}$ the isotropy, $Spin\left( 8\right) _{x},$ of the $Spin\left( 8\right) 
$-action on $S^{7}$ is isomorphic to $Spin\left( 7\right) .$ For simplicity,
we denote $Spin\left( 8\right) _{x}$ by $Spin\left( 7\right) .$ Let $%
\mathfrak{spin}\left( 7\right) \subset \mathfrak{spin}\left( 8\right)
\subset \mathfrak{spin}\left( 9\right) $ be the Lie algebras of $Spin\left(
7\right) \subset Spin\left( 8\right) \subset Spin\left( 9\right) $. Let $%
\mathfrak{m}_{\mathfrak{spin}\left( 8\right) }$ and $\mathfrak{m}_{\mathfrak{%
spin}\left( 9\right) }$ be the vector subspaces so that the splitting 
\begin{equation*}
\mathfrak{spin}\left( 9\right) =\mathfrak{spin}\left( 7\right) \oplus 
\mathfrak{m}_{\mathfrak{spin}\left( 8\right) }\oplus \mathfrak{m}_{\mathfrak{%
spin}\left( 9\right) }
\end{equation*}%
is orthogonal and 
\begin{equation*}
\mathfrak{spin}\left( 8\right) =\mathfrak{spin}\left( 7\right) \oplus 
\mathfrak{m}_{\mathfrak{spin}\left( 8\right) }.
\end{equation*}

\begin{proposition}
\label{horiz S8 X spiin9}At any point of the form $\left( A,\left(
x,t\right) \right) \in Spin\left( 9\right) \times \left\{ S^{8}\setminus
\left\{ \pm e_{9}\right\} \right\} $ the horizontal space of $q$ is spanned
by vectors of the form 
\begin{equation}
\left\{ \left( 0,X\right) ,\left( \left( L_{A}\right) _{\ast }k^{9},0\right)
,\left( \sin ^{2}t\left( L_{A}\right) _{\ast }\left( k^{8}\right)
,k_{S^{8}}^{8}\right) \right\} ,  \label{OP^2 hoiz}
\end{equation}%
where $X\equiv \mathrm{grad}\left( \mathrm{dist}_{S^{8}}\left( e^{9},\cdot
\right) \right) $, $k^{9}\in \mathfrak{m}_{\mathfrak{spin}\left( 9\right) }$%
, $k^{8}\in \mathfrak{m}_{\mathfrak{spin}\left( 8\right) }$, and $t=\mathrm{%
dist}_{S^{8}}\left( e_{9},\cdot \right) .$

At a point of the form $\left( A,\pm e_{9}\right) \in Spin\left( 9\right)
\times \left\{ \pm e_{9}\right\} $ the horizontal space of $q$ is spanned by
vectors of the form 
\begin{equation}
\left\{ \left( 0,X\right) ,\left( \left( L_{A}\right) _{\ast }k^{9},0\right)
\right\}
\end{equation}%
where $X\in T_{\pm e_{9}}S^{8}$, $k^{9}\in \mathfrak{m}_{\mathfrak{spin}%
\left( 9\right) }$.
\end{proposition}

\begin{remark}
Recall our convention that for an abstract $G$--manifold $M$ and an element $%
k$ of the Lie algebra $\mathfrak{g}$, $k_{M}$ denotes the Killing field on $%
M $ generated by $k.$ Thus, $k_{S^{8}}^{8}$ is the Killing field on $S^{8}$
generated by $k^{8}\in \mathfrak{m}_{\mathfrak{spin}\left( 8\right) }\subset 
\mathfrak{spin}\left( 8\right) $, and $\left( L_{A}\right) _{\ast }k^{9}$
would be written as $k_{Spin\left( 9\right) }^{9}.$ However, we write $%
\left( L_{A}\right) _{\ast }k^{9},$ since the notation is standard.
\end{remark}

\begin{proof}
The definitions of $\left( 0,X\right) ,\left( \left( L_{A}\right) _{\ast
}k^{9},0\right) $ and the $Spin\left( 8\right) $-action give us that $\left(
0,X\right) $ and $\left( \left( L_{A}\right) _{\ast }k^{9},0\right) $ are $q$%
--horizontal at points of $Spin\left( 9\right) \times \left\{ S^{8}\setminus
\left\{ \pm e_{9}\right\} \right\} .$

For any $k\in \mathfrak{m}_{\mathfrak{spin}\left( 8\right) }$ we have%
\begin{eqnarray}
&&\left( g_{\mathrm{bi}}+g_{S^{8}}\right) \left( \left( \sin ^{2}t\left(
L_{A}\right) _{\ast }\left( k^{8}\right) ,\text{ }k_{S^{8}}^{8}\right)
,\left( -\left( L_{A}\right) _{\ast }k,\text{ }k_{S^{8}}\right) \right) 
\notag \\
&=&-\sin ^{2}tg_{\mathrm{bi}}\left( k^{8},k\right) +g_{S^{8}}\left(
k_{S^{8}}^{8},k_{S^{8}}\right)  \label{taut}
\end{eqnarray}

Since $S^{7}=Spin\left( 8\right) /Spin\left( 7\right) ,$ we have the
Riemannian submersion 
\begin{equation*}
Spin\left( 8\right) \longrightarrow S^{7}=S^{7}\times \left\{ \frac{\pi }{2}%
\right\} \subset S^{8}.
\end{equation*}

Recall that we have used $\mathfrak{m}_{\mathfrak{spin}\left( 8\right) }$ to
denote the horizontal space at $x,$ and in our notation, the differential is 
\begin{eqnarray*}
\mathfrak{m}_{\mathfrak{spin}\left( 8\right) } &\mapsto &TS^{7} \\
k &\mapsto &k_{S^{8}}.
\end{eqnarray*}%
So $g_{S^{8}}\left( k_{S^{8}}^{8},k_{S^{8}}\right) |_{\left( x,\frac{\pi }{2}%
\right) }=g_{S^{7}}\left( k_{S^{8}}^{8},k_{S^{8}}\right) |_{x}=g_{\mathrm{bi}%
}\left( k^{8},k\right) ,$ and $g_{S^{8}}\left(
k_{S^{8}}^{8},k_{S^{8}}\right) |_{\left( x,t\right) }=\sin ^{2}tg_{\mathrm{bi%
}}\left( k^{8},k\right) .$ So the right hand side of Equation \ref{taut} is $%
0.$

On the other hand, for $k\in \mathfrak{spin}\left( 7\right) $ we also have 
\begin{eqnarray*}
&&\left( g_{\mathrm{bi}}+g_{S^{8}}\right) \left( \left( \sin ^{2}t\left(
L_{A}\right) _{\ast }\left( k^{8}\right) ,\text{ }k_{S^{8}}^{8}\right)
,\left( -\left( L_{A}\right) _{\ast }k,\text{ }k_{S^{8}}\right) \right) \\
&=&-\sin ^{2}tg_{\mathrm{bi}}\left( k^{8},k\right) +g_{S^{8}}\left(
k_{S^{8}}^{8},k_{S^{8}}\right)
\end{eqnarray*}%
The first term is $0$ since $k^{8}\in \mathfrak{m}_{\mathfrak{spin}\left(
8\right) }$ and $k\in \mathfrak{spin}\left( 7\right) .$ Further, $%
k_{S^{8}}=0,$ since $k\in \mathfrak{spin}\left( 7\right) $ and $Spin\left(
7\right) $ is the isotropy at $\left( x,t\right) .$ So the second term is $0$%
, and it follows that $\left( \sin ^{2}t\left( L_{A}\right) _{\ast }\left(
k^{8}\right) ,\text{ }k_{S^{8}}^{8}\right) $ is in the horizontal space of $%
q,$ proving the first statement.

To prove the second statement, notice that $\pm e_{9}$ are the fixed points
of the $Spin\left( 8\right) $-action on $S^{8},$ so, at a point of the form $%
\left( A,\pm e_{9}\right) \in Spin\left( 9\right) \times \left\{ \pm
e_{9}\right\} $, the vectors $\left( 0,X\right) ,$ $X\in T_{\pm e_{9}}S^{8}$
are horizontal for $q.$ Then observe that $\left( L_{A}\right) _{\ast
}\left( \mathfrak{m}_{\mathfrak{spin}\left( 9\right) }\right) $ is the
horizontal distribution for the right $Spin\left( 8\right) $ action on $%
Spin\left( 9\right) .$
\end{proof}

Combining this with the Horizontal Curvature Equation and a linear algebra
argument we will show the following.

\begin{proposition}
\label{zero curvs}$\left( \mathbb{O}P^{2}\#-\mathbb{O}P^{2},g_{q}\right) $
is non-negatively curved.

\noindent 1. All of the zero curvature planes in $q\left( Spin\left(
9\right) \times \left\{ S^{8}\setminus \pm e_{9}\right\} \right) $ have
horizontal lifts to $Spin\left( 9\right) \times S^{8}$ of the form 
\begin{equation}
\mathrm{span}\left\{ \left( 0,X\right) ,\left( \left( L_{A}\right) _{\ast
}k^{9},0\right) \right\} ,  \label{zeros in Op2 -Op2}
\end{equation}%
where $k^{9}\in \mathfrak{m}_{\mathfrak{spin}\left( 9\right) }.$

\noindent 2. All of the zero curvature planes in $q\left( Spin\left(
9\right) \times \left\{ \pm e_{9}\right\} \right) $ have horizontal lifts to 
$Spin\left( 9\right) \times S^{8}$ of the form%
\begin{equation}
\mathrm{span}\left\{ \left( 0,X\right) ,\left( \left( L_{A}\right) _{\ast
}k^{9},0\right) \right\}  \label{singular zeros 1}
\end{equation}%
where $X\in T_{\pm e_{9}}S^{8}$ and $k^{9}\in \mathfrak{m}_{\mathfrak{spin}%
\left( 9\right) }.$
\end{proposition}

\begin{proof}
$\left( \mathbb{O}P^{2}\#-\mathbb{O}P^{2},g_{q}\right) $ is non-negatively
curved since $\left( \mathbb{O}P^{2}\#-\mathbb{O}P^{2},g_{q}\right) =\left(
Spin\left( 9\right) \times S^{8}\right) /Spin\left( 8\right) .$

If a plane, $P,$ tangent to $\left( \mathbb{O}P^{2}\#-\mathbb{O}%
P^{2},g_{q}\right) $ has zero curvature, then its horizontal lift to $%
Spin\left( 9\right) \times S^{8}$ also has zero curvature, so to prove Part
1, it suffices to show that the planes of the form \ref{zeros in Op2 -Op2}
are the only zero curvature planes in the distribution in \ref{OP^2 hoiz}.

To prove this we set 
\begin{equation*}
\mathcal{P}\equiv \left\{ \left( \left( L_{A}\right) _{\ast }k^{9},0\right)
,\left( \sin ^{2}t\left( L_{A}\right) _{\ast }\left( k^{8}\right)
,k_{S^{8}}^{8}\right) \right\} .
\end{equation*}%
Notice that the structure of $\mathcal{P}$ gives us bases $\left\{
a_{i}\right\} $ for $\mathfrak{m}_{\mathfrak{spin}\left( 9\right) }$ and $%
\left\{ b_{i}\right\} $ for $\mathfrak{m}_{\mathfrak{spin}\left( 8\right) }$
for which 
\begin{equation}
\mathcal{P}\equiv \mathrm{span}\left\{ \left( a_{i},0\right) ,\left(
b_{i},\iota \left( b_{i}\right) \right) \right\} ,  \label{structure}
\end{equation}%
where $\iota :\mathfrak{m}_{\mathfrak{spin}\left( 8\right) }\longrightarrow
TS_{\left( x,t\right) }^{7}\subset TS_{\left( x,t\right) }^{8}$ is the
isomorphism that maps $\sin ^{2}t\left( L_{A}\right) _{\ast }\left(
k^{8}\right) \mapsto k_{S^{8}}^{8}.$

Let $\pi _{1}:Spin\left( 9\right) \times S^{8}\longrightarrow Spin\left(
9\right) $ and $\pi _{2}:Spin\left( 9\right) \times S^{8}\longrightarrow
S^{8}$ be the respective projections.

From the structure of $\mathcal{P}$ in \ref{structure} it follows that for $%
P,$ a $2$--plane in $\mathcal{P},$ $d\pi _{1}\left( P\right) $ is also $2$%
--dimensional. Combining this with the fact that $\mathfrak{m}_{\mathfrak{%
spin}\left( 9\right) }\oplus \mathfrak{m}_{\mathfrak{spin}\left( 8\right) }$
is the horizontal space of $Spin\left( 7\right) \longrightarrow Spin\left(
9\right) \longrightarrow S^{15}$, it follows that 
\begin{equation}
\mathrm{sec}\left( P\right) >0\text{ for all planes }P\text{ in }\mathcal{P}%
\text{.\label{positive curv}}
\end{equation}

On the other hand, the horizontal distribution is 
\begin{equation*}
\mathrm{span}\left\{ \left( 0,X\right) ,\mathcal{P}\right\} ,
\end{equation*}%
so, in general, we can write a horizontal plane as 
\begin{equation*}
P=\mathrm{span}\left\{ \left( 0,\sigma X\right) +V,W\right\} ,
\end{equation*}%
where $V,W\in \mathcal{P},$ $V\perp W$ and $\sigma \in \mathbb{R}.$ Using
the superscripts $^{1}$ and $^{2}$ for the projections to the first and
second factors of $T\left( Spin\left( 9\right) \times S^{8}\right) $ and the
fact that $Spin\left( 9\right) \times S^{8}$ has a product metric, we see
that 
\begin{eqnarray*}
\mathrm{curv}_{g_{\mathrm{bi}}+g_{S^{8}}}\left( \left( 0,\sigma X\right)
+V,W\right) &=&\mathrm{curv}_{g_{\mathrm{bi}}+g_{S^{8}}}\left( \left(
0,\sigma X\right) ,W\right) +2R_{g_{\mathrm{bi}}+g_{S^{8}}}\left( \left(
0,\sigma X\right) ,W,W,V\right) \\
&&+\,\mathrm{curv}_{g_{\mathrm{bi}}+g_{S^{8}}}\left( V,W\right) \\
&=&\mathrm{curv}_{g_{S^{8}}}\left( \sigma X,W^{2}\right)
+2R_{g_{S^{8}}}\left( \sigma X,W^{2},W^{2},V^{2}\right) \\
&&+\,\mathrm{curv}_{g_{\mathrm{bi}}}\left( V^{1},W^{1}\right) +\mathrm{curv}%
_{g_{S^{8}}}\left( V^{2},W^{2}\right) \\
&=&\mathrm{curv}_{g_{S^{8}}}\left( \sigma X+V^{2},W^{2}\right) +\mathrm{curv}%
_{g_{\mathrm{bi}}}\left( V^{1},W^{1}\right) .
\end{eqnarray*}

Since $\mathrm{curv}_{g_{S^{8}}}\left( \sigma X+V^{2},W^{2}\right) $ is a
curvature of $S^{8}$ and $\mathrm{curv}_{g_{\mathrm{bi}}}\left(
V^{1},W^{1}\right) $ is the horizontal lift of a curvature of $S^{15}$ to $%
Spin\left( 9\right) ,$ both terms are non-negative. Since $X\perp W^{2}$ and 
$X\perp V^{2},$ the first term is positive if both $\sigma $ and $W^{2}$ are
not zero. If $\sigma =0,$ then our plane is in $\mathcal{P},$ and has
positive curvature. If $W^{2}=0,$ then 
\begin{equation*}
\mathrm{curv}_{g_{\mathrm{bi}}+g_{S^{8}}}\left( P\right) =\mathrm{curv}_{g_{%
\mathrm{bi}}}\left( V^{1},W^{1}\right) >0,
\end{equation*}%
unless $V^{1}$ is proportional to $W^{1}.$ Since $V\perp W,$ and $W^{2}=0,$
this would give $V^{1}=0.$ However, from the structure of $\mathcal{P}$ in %
\ref{structure}, we see that $V^{1}=0$ implies $V=0.$

So the planes $P=\mathrm{span}\left\{ \left( 0,\sigma X\right) +V,W\right\} $
that have zero curvature are those with $W^{2}=0$ and $V=0.$ It follows that
all horizontal zero curvature planes tangent to $Spin\left( 9\right) \times
S^{8}$ have the desired form 
\begin{equation*}
\mathrm{span}\left\{ \left( 0,X\right) ,\left( \left( L_{A}\right) _{\ast
}k^{9},0\right) \right\} .
\end{equation*}%
Since the curvature of all these planes is zero, the proof of Part 1 is
complete.

Part 2 follows by combining the second statement of Proposition \ref{horiz
S8 X spiin9} and the following facts:

\noindent 1. $Spin\left( 9\right) \times S^{8}$ has the product metric.

\noindent 2. Any plane tangent plane to $Spin\left( 9\right) \times S^{8}$
with a $2$--dimensional projection to $TS^{8}$ is positively curved.

\noindent 3. Any plane tangent plane to $Spin\left( 9\right) \times S^{8}$
with a $2$--dimensional projection to $\left( \left( L_{A}\right) _{\ast }%
\mathfrak{m}_{\mathfrak{spin}\left( 9\right) },0\right) $ is positively
curved.
\end{proof}

\begin{remark}
From Corollary 1 of \cite{ProWilh1} it also follows that all planes of the
form \ref{zeros in Op2 -Op2} or \ref{singular zeros 1} project to zero
curvature planes in $\left( \mathbb{O}P^{2}\#-\mathbb{O}P^{2},g_{q}\right) .$
\end{remark}

View the double mapping cylinder of the octonionic Hopf fibration as%
\begin{equation*}
(\left[ 0,\pi \right] \times S^{15})/\sim ,
\end{equation*}%
where $\left( 0\times S^{15}\right) /\sim $ and $\left( \pi \times
S^{15}\right) /\sim $ are diffeomorphic to $S^{8}.$ We write $\left( 0\times
S^{15}\right) /\sim $ and $\left( \pi \times S^{15}\right) /\sim $ as $%
0\times S^{8}$ and $\pi \times S^{8},$ respectively, and we let $t\equiv 
\mathrm{dist}\left( 0\times S^{8},\cdot \right) ,$ where the distance is
determined by $g_{q}.$

Under the diffeomorphism between $\left( Spin\left( 9\right) \times
S^{8}\right) /Spin\left( 8\right) $ and the double mapping cylinder of the
octonionic Hopf fibration, the equivalence classes of the sets $Spin\left(
9\right) \times \left\{ \pm e_{9}\right\} $ map to $0\times S^{8}$ and $\pi
\times S^{8},$ which are the distinguished $\mathbb{O}P^{1}$s of $\mathbb{O}%
P^{2}\#-\mathbb{O}P^{2}.$ The octonionic Hopf fibration $S^{15}%
\longrightarrow S^{8}$, written in terms of $Spin\left( 8\right) $ cosets is 
\begin{eqnarray}
\left( Spin\left( 9\right) \times S^{7}\right) /Spin\left( 8\right)
&\longrightarrow &Spin\left( 9\right) /Spin\left( 8\right)  \notag \\
\left( A,v\right) \cdot Spin\left( 8\right) &\longmapsto &A\cdot Spin\left(
8\right) .  \label{62}
\end{eqnarray}%
The field $\left( 0,X\right) $ on $Spin\left( 9\right) \times S^{8}$ is the
gradient of the distance from $Spin\left( 9\right) \times \left\{
e_{9}\right\} .$ The vectors $\left( \left( L_{A}\right) _{\ast
}k^{9},0\right) $ are horizontal for the Hopf fibration \ref{62}, so
Proposition \ref{zero curvs} gives us Part 1 of the following.

\begin{corollary}
\label{OP2-Op2 zeros final}View $\left( \mathbb{O}P^{2}\#-\mathbb{O}%
P^{2},g_{q}\right) $ as the double mapping cylinder of the octonionic Hopf
fibration.

\noindent 1. The zero curvature planes in $\mathbb{O}P^{2}\#-\mathbb{O}%
P^{2}\setminus \left\{ \mathbb{O}P^{1}\cup \mathbb{O}P^{1}\right\} $ are
precisely those of the form 
\begin{equation}
\mathrm{span}\left\{ X,Z\right\} ,  \label{zeros in Op2 -Op2 abs final}
\end{equation}%
where $X$ is the gradient of the distance from an $\mathbb{O}P^{1}\subset 
\mathbb{O}P^{2}$ and $Z$ is tangent to the levels of the same distance
function and, in addition, is horizontal for the Hopf fibration $%
S^{15}\longrightarrow S^{8}.$

\noindent 2. The zero curvature planes in $\left\{ \mathbb{O}P^{1}\cup 
\mathbb{O}P^{1}\right\} \subset \mathbb{O}P^{2}\#-\mathbb{O}P^{2}$ are
precisely those of the form%
\begin{equation}
\mathrm{span}\left\{ X,Z\right\}  \label{singular zeros}
\end{equation}%
where $X$ is normal to one of the $\mathbb{O}P^{1}$s and $Z$ is tangent to
the same $\mathbb{O}P^{1}.$

\noindent 3. The one parameter family of Berger metrics $\left\{
g_{q}|_{\left\{ t\right\} \times S^{15}}\right\} _{t\in \left( 0,\pi \right)
}$ have the following property. For any $Z\in TS^{15}$ that is horizontal
for the Hopf fibration $S^{15}\longrightarrow S^{8},$ $g_{q}|_{\left\{
t\right\} \times S^{15}}\left( Z,\cdot \right) $ is independent of $t.$
\end{corollary}

\begin{proof}
For Part 2, just observe that the planes in \ref{singular zeros 1} are
precisely the planes in \ref{singular zeros}.

For Part 3, notice that by Proposition \ref{horiz S8 X spiin9} the
horizontal lift to $Spin\left( 9\right) \times S^{8}$ of a Hopf--horizontal $%
Z\in T\left( \left\{ t\right\} \times S^{15}\right) $ has the form 
\begin{equation*}
\left( \left( L_{A}\right) _{\ast }k_{Z}^{9},0\right)
\end{equation*}%
for a fixed $k_{Z}^{9}\in \mathfrak{m}_{\mathfrak{spin}\left( 9\right) }.$
On the other hand, if $W\in T\left( \left\{ t\right\} \times S^{15}\right) $
is any vector, then its horizontal lift to $Spin\left( 9\right) \times S^{8}$
has the form 
\begin{equation*}
\left( \left( L_{A}\right) _{\ast }k_{W}^{9},0\right) +\left( \sin
^{2}t\left( L_{A}\right) _{\ast }\left( k_{W}^{8}\right)
,k_{W,S^{8}}^{8}\right) ,
\end{equation*}%
for some $k_{W}^{9}\in \mathfrak{m}_{\mathfrak{spin}\left( 9\right) }$ and
some $k_{W}^{8}\in \mathfrak{m}_{\mathfrak{spin}\left( 8\right) }.$ Thus%
\begin{eqnarray*}
g_{q}|_{\left\{ t\right\} \times S^{15}}\left( Z,W\right) &=&\left( g_{%
\mathrm{bi}}+g_{S^{8}}\right) \left( \left( \left( L_{A}\right) _{\ast
}k_{Z}^{9},0\right) ,\text{ }\left( \left( L_{A}\right) _{\ast
}k_{W}^{9},0\right) +\left( \sin ^{2}t\left( L_{A}\right) _{\ast }\left(
k_{W}^{8}\right) ,k_{W,S^{8}}^{8}\right) \right) \\
&=&g_{\mathrm{bi}}\left( k_{Z}^{9},k_{W}^{9}\right) ,\text{ since }\mathfrak{%
m}_{\mathfrak{spin}\left( 9\right) }\text{ and }\mathfrak{m}_{\mathfrak{spin}%
\left( 8\right) }\text{ are orthogonal.}
\end{eqnarray*}%
Since the right hand side is independent of $t,$ the result follows.
\end{proof}

Next, we relate the horizontal spaces of the $G_{2}$ action on $S^{15}$ and
the horizontal spaces of the Hopf fibration $h:S^{15}\longrightarrow S^{8}.$

Adopting the point of view of \cite{Wilh}, an explicit formula for the Hopf
fibration $h:S^{15}\longrightarrow S^{8}$ is given as follows. View $S^{15}$
as the unit sphere in $\mathbb{O}\oplus \mathbb{O\cong R}^{16}$, and view $%
S^{8}$ as the unit sphere in $\mathbb{O}\oplus \mathbb{R\cong R}^{9}.$ Then 
\begin{equation*}
h:\left( 
\begin{array}{c}
a \\ 
c%
\end{array}%
\right) \mapsto (a\bar{c},\frac{1}{2}(|a|^{2}-|c|^{2})).
\end{equation*}

The last ingredient in our proof of Theorem \ref{Ricc on quotient} is the
following.

\begin{proposition}
\label{Davis-Hopf}For all $\left( 
\begin{array}{c}
a \\ 
c%
\end{array}%
\right) \in S^{15},$ there is a vector in $\left\{ T_{\left( 
\begin{array}{c}
a \\ 
c%
\end{array}%
\right) }G_{2}\left( 
\begin{array}{c}
a \\ 
c%
\end{array}%
\right) \right\} ^{\perp }$ that is not Hopf horizontal, that is, it is not
in $\left\{ T_{\left( 
\begin{array}{c}
a \\ 
c%
\end{array}%
\right) }h^{-1}\left( h\left( 
\begin{array}{c}
a \\ 
c%
\end{array}%
\right) \right) \right\} ^{\perp }.$
\end{proposition}

\begin{proof}
If $Im\left( a\right) \neq 0$, set $\frac{Im\left( a\right) }{\left\vert
Im\left( a\right) \right\vert }=\alpha .$ We claim that at $\left( 
\begin{array}{c}
a \\ 
c%
\end{array}%
\right) $ the vector $\left( 
\begin{array}{c}
a\alpha \\ 
0%
\end{array}%
\right) $ is in $\left\{ T_{\left( 
\begin{array}{c}
a \\ 
c%
\end{array}%
\right) }G_{2}\left( 
\begin{array}{c}
a \\ 
c%
\end{array}%
\right) \right\} ^{\perp }.$ Indeed let $S^{7}\left( \left\vert a\right\vert
\right) $ be the octonions with norm equal to $\left\vert a\right\vert .$
The curve, 
\begin{eqnarray*}
\gamma _{\alpha } &:&\left[ 0,2\pi \right] \longrightarrow S^{7}\left(
\left\vert a\right\vert \right) \text{ } \\
\gamma _{\alpha } &:&t\mapsto \left\vert a\right\vert e^{\alpha t},
\end{eqnarray*}%
is the geodesic in $S^{7}\left( \left\vert a\right\vert \right) $ that
passes through $\pm \left\vert a\right\vert $ and $a.$ The $G_{2}$--action
on $S^{7}\left( \left\vert a\right\vert \right) $ is by cohomogeneity one
with singular orbits $\pm \left\vert a\right\vert .$ Thus $\gamma _{\alpha }$
is normal to the orbits of $G_{2}.$ On the other hand, if $\gamma _{\alpha
}\left( t_{0}\right) =a,$ then $\gamma _{\alpha }^{\prime }\left(
t_{0}\right) =a\alpha ,$ so at $\left( 
\begin{array}{c}
a \\ 
c%
\end{array}%
\right) ,$ 
\begin{equation*}
\left( 
\begin{array}{c}
a\alpha \\ 
0%
\end{array}%
\right) \in \left\{ T_{\left( 
\begin{array}{c}
a \\ 
c%
\end{array}%
\right) }G_{2}\left( 
\begin{array}{c}
a \\ 
c%
\end{array}%
\right) \right\} ^{\perp },
\end{equation*}%
as claimed.

To see that this vector is not Hopf horizontal, notice that since $\left[
a,\alpha \right] =0,$ $a,\alpha ,$ and $c$ are contained in a subalgebra
that is isomorphic to $\mathbb{H}.$ In particular, for all $t\in \mathbb{R}$
the three octonions $a,$ $c,$ and $e^{\alpha t}$ associate. So 
\begin{equation*}
\left( 
\begin{array}{c}
ae^{\alpha t} \\ 
ce^{\alpha t}%
\end{array}%
\right) \in h^{-1}\left( h\left( 
\begin{array}{c}
a \\ 
c%
\end{array}%
\right) \right) ,
\end{equation*}%
and it follows that 
\begin{equation*}
\left( 
\begin{array}{c}
a\alpha \\ 
c\alpha%
\end{array}%
\right) \in T_{\left( 
\begin{array}{c}
a \\ 
c%
\end{array}%
\right) }h^{-1}\left( h\left( 
\begin{array}{c}
a \\ 
c%
\end{array}%
\right) \right) .
\end{equation*}%
So 
\begin{equation*}
\left( 
\begin{array}{c}
a\alpha \\ 
0%
\end{array}%
\right) \notin \left\{ T_{\left( 
\begin{array}{c}
a \\ 
c%
\end{array}%
\right) }h^{-1}\left( h\left( 
\begin{array}{c}
a \\ 
c%
\end{array}%
\right) \right) \right\} ^{\perp }.
\end{equation*}

A similar argument covers points for which $Im\left( c\right) \neq 0.$

Finally, if $Im\left( a\right) =Im\left( c\right) =0,$ then $\left( 
\begin{array}{c}
a \\ 
c%
\end{array}%
\right) $ is a fixed point of $G_{2}$ and all vectors are in $\left\{
T_{\left( 
\begin{array}{c}
a \\ 
c%
\end{array}%
\right) }G_{2}\left( 
\begin{array}{c}
a \\ 
c%
\end{array}%
\right) \right\} ^{\perp }.$
\end{proof}

\begin{proof}[Proof of Theorem \protect\ref{Ricc on quotient}]
Combining Proposition \ref{zero curvs} and Corollary \ref{OP2-Op2 zeros
final} we see that $\mathbb{O}P^{2}\#-\mathbb{O}P^{2}$ is non-negatively
curved and every zero plane in $\left( \mathbb{O}P^{2}\#-\mathbb{O}%
P^{2}\right) \setminus \left( \mathbb{O}P^{1}\cup \mathbb{O}P^{1}\right) $
contains $X$ and a Hopf horizontal vector. Similarly, every zero plane in $%
\left( \mathbb{O}P^{1}\cup \mathbb{O}P^{1}\right) \subset \mathbb{O}P^{2}\#-%
\mathbb{O}P^{2}$ is spanned by a vector tangent to an $\mathbb{O}P^{1}$ and
a vector normal to the same $\mathbb{O}P^{1}.$

So $\left( \mathbb{O}P^{2}\#-\mathbb{O}P^{2}\right) ^{\text{reg}}/G_{2}$ at
least has nonnegative Ricci curvature, and the only possible direction with
zero Ricci curvature is $X.$

From Part 3 of Proposition \ref{OP2-Op2 zeros final} and Proposition \ref%
{Davis-Hopf} we have an $\alpha >0$ so that at all points of $x\in \left( 
\mathbb{O}P^{2}\#-\mathbb{O}P^{2}\right) ^{\text{reg}}$ there is a vector $%
Y\in TG_{2}\left( x\right) ^{\perp }$ with 
\begin{equation}
\sphericalangle \left( Y,\left\{ \text{Hopf horizontal vectors}\right\}
\right) >\alpha >0.
\end{equation}

Combining this with Corollary \ref{OP2-Op2 zeros final} we see that the
planes, 
\begin{equation*}
\mathrm{span}\left\{ X,Y\right\} ,
\end{equation*}%
are in the complement of a neighborhood $U$ of the zero planes of $\mathbb{O}%
P^{2}\#-\mathbb{O}P^{2}.$ Hence, by compactness of the complement of $U$,
there is a $\beta >0$ so that $sec\left( X,Y\right) >\beta >0.$ Since all
other sectional curvatures are at least nonnegative$,$ we have 
\begin{equation*}
Ric_{\left( \mathbb{O}P^{2}\#-\mathbb{O}P^{2}\right) ^{\text{reg}%
}/G_{2}}\left( X,X\right) >\beta >0.
\end{equation*}
\end{proof}


\begin{thebibliography}{99}
\bibitem{Ber} V. N. Berestovskii, \emph{Homogeneous Riemannian manifolds of
positive Ricci curvature,} Mat. Zametki, \textbf{58} issue 3 (1995),
334--340.

\bibitem{BB} L. B\'{e}rard-Bergery, \emph{Certains fibr\'{e}s \`{a} courbure
de Ricci positive}, C.R. Acad. Sc. Paris \textbf{286} (1978), 929--931.

\bibitem{BS-W1} S. Bechtluft-Sachs, D. Wraith, \emph{On the curvature of
G-manifolds with finitely many non-principal orbits}, Geom. Ded. \textbf{162}
(2013), 109--128.

\bibitem{Bet} R. G. Bettiol, \emph{Positive biorthogonal curvature on }$%
S^{2}\times S^{2},$ Proc. Amer. Math. Soc., to appear,
http://arxiv.org/pdf/1210.0043.pdf.

\bibitem{BoyGalNak} C. P. Boyer, K. Galicki, and M. Nakamaye, \emph{Sasakian
geometry, homotopy spheres and positive Ricci curvature, }Topology \textbf{42%
} (2003), 981--1002.

\bibitem{Bred} G. Bredon, \emph{An Introduction to Compact Transformation
Groups, }Academic Press, Orlando, 1972.

\bibitem{Cheeg} J. Cheeger, \emph{Some examples of manifolds of non-negative
curvature}. J. Diff. Geom. \textbf{8} (1973), 623--628.

\bibitem{Dav} M. Davis,{\normalsize \ }\emph{Some group actions on homotopy
spheres of dimensions seven and fifteen,}{\normalsize \ }Amer. Journ. of
Math. \textbf{104} (1982) 59--90.{\normalsize \ }

\bibitem{Ell-Kup} J. Eells and N. H. Kuiper, \emph{An invariant for certain
smooth manifolds}, Annali Mat. \textbf{60} (1962), 93--110.

\bibitem{FukYam} K. Fukaya, T. Yamaguchi, \emph{The fundamental groups of
almost nonnegatively curved manifolds, }Ann. of Math. \textbf{136} (1992),
253--333.

\bibitem{GluWarZil} H. Gluck, F. Warner, and W. Ziller,{\normalsize \ }\emph{%
The geometry of the Hopf fibrations,}{\normalsize \ }L'Enseignement Math. 
\textbf{32} (1986), 173--198.

\bibitem{GPT} P. Gilkey, J. H. Park, W. Tuschmann, \emph{Invariant metrics
of positive Ricci curvature on principal bundles}, Math. Z. \textbf{227}
(1998), 455--463.

\bibitem{Gray} A. Gray, \emph{Pseudo-Riemannian\ almost product manifolds
and submersions, }J. Math. Mech. \textbf{16} (1967), 413--443.

\bibitem{GrWu1} R. E. Greene and H. Wu{\normalsize , }\emph{On the
subharmonicity and plurisubharmonicity of geodesically convex functions,}%
{\normalsize \ }Indiana Univ. Math. J. \textbf{22} (1973), 641--653.

\bibitem{GrWu2} R. E. Greene and H. Wu{\normalsize , }\emph{Integrals of
subharmonic functions on manifolds of nonnegative curvature,}{\normalsize \ }%
Invent. Math., \textbf{27} (1974), 265--298.

\bibitem{GromMey} D. Gromoll and W. Meyer,{\normalsize \ }\emph{An exotic
sphere with nonnegative sectional curvature,}{\normalsize \ }Ann. of Math. 
\textbf{100} (1974), 401--406.{\normalsize \ }

\bibitem{GromWals} D. Gromoll and G. Walschap, \emph{Metric foliations and
curvature}, Progress in Mathematics, vol 268. Birkh\"{a}user, Verlag, Basel,
2009.

\bibitem{Grov} K. Grove, \emph{Geometry of, and via, symmetries, }
Conformal, Riemannian and Lagrangian geometry (Knoxville, TN, 2000), Univ.
Lecture Ser., 27, Amer. Math. Soc., Providence, RI, 2002.

\bibitem{GrovShio} K. Grove and K. Shiohama, \emph{A generalized sphere
theorem}, Ann. of Math. \textbf{106} (1977), 201--211.

\bibitem{GVWZ} K. Grove, L. Verdiani, B. Wilking, W. Ziller, \emph{%
Non-negative curvature obstructions in cohomogeneity one and the Kervaire
spheres}, Ann. Sc. Norm. Sup. Pisa, Volume 5, Issue 2, (2006) 159--170.

\bibitem{GrovZil1} K. Grove and W. Ziller, \emph{Curvature and Symmetry of
Milnor Spheres}, Ann. of Math. \textbf{152} (2000), 331--367.

\bibitem{GrovZil} K. Grove and W. Ziller, \emph{Cohomogeneity one manifolds
with positive Ricci curvature}, Invent. Math. 149, no. 3, (2002), 616--646.

\bibitem{Guij} L. Guijarro, \emph{Improving the metric in an open manifold
with nonnegative curvature, }Proc. Amer. Math. Soc., \textbf{126} (1998),
1541--1545.

\bibitem{He} C. He, \emph{New examples of obstructions to non-negative
sectional curvatures in cohomogeneity one manifolds, }Trans. Amer. Math.
Soc., published electronically on March 4, 2014, also at
http://arxiv.org/pdf/0910.5712.pdf

\bibitem{KerMil} M. Kervaire and J. Milnor, \emph{Groups of homotopy spheres
I}, Ann. of Math., \textbf{77} (1963), 504--537.

\bibitem{KuSh} K. Kuwae, T. Shioya, \emph{Infinitesimal Bishop-Gromov
condition for Alexandrov spaces}, Probabilistic Approach to Geometry,
293--302, Adv. Stud. Pure Math. 57, Math. Soc. Japan, Tokyo, 2010.

\bibitem{LawYau} H. B. Lawson, Jr. and S. T. Yau, \emph{Scalar curvature,
non-abelian group actions, and the degree of symmetry of exotic spheres,}
Comment. Math. Helv. \textbf{49} (1974), 232--244.

\bibitem{LottVill} J. Lott and C. Villani, \emph{Ricci curvature for
metric-measure spaces via optimal transport}, Ann. of Math., \textbf{169}
(2009), 903--991.

\bibitem{Mil} J. Milnor,{\normalsize \ }\emph{On manifolds homeomorphic to
the }$7$\emph{-sphere,}{\normalsize \ }Ann. of Math. \textbf{64} (1956),
399-405.

\bibitem{N} J. C. Nash, \emph{Positive Ricci Curvature on Fibre Bundles}, J.
of Diff. Geom., \textbf{14} (1979), 241--254.

\bibitem{Oht} S. Ohta, \emph{On the measure contraction property of metric
measure spaces, }Comment. Math. Helv. \textbf{82} (2007), no. 4, 805--828.

\bibitem{O'Neill} B. O'Neill, \emph{The fundamental equations of a submersion%
}, Mich. Math. J. \textbf{13} (1966), 459--469.

\bibitem{OSY} Y. Otsu, K. Shiohama and T. Yamaguchi, \emph{A new version of
differentiable sphere theorem.} Invent. Math. \textbf{98} (1989), 219--228.

\bibitem{PetWilh1} P. Petersen and F. Wilhelm, \emph{Examples of Riemannian
manifolds with positive curvature almost everywhere}. Geom. \& Topol. 
\textbf{3} (1999), 331--367.

\bibitem{PetWilh2} P. Petersen and F. Wilhelm, \emph{An exotic sphere with
positive sectional curvature}, preprint,
http://arxiv.org/pdf/0805.0812v3.pdf.

\bibitem{Poor} W. A. Poor, \emph{Some exotic spheres with positive Ricci
curvature}, Math. Ann. \textbf{216} (1975), 245--252.

\bibitem{ProWilh1} C. Pro and F. Wilhelm, \emph{Flats and submersions in
non-negative curvature}, Geom. Ded., \textbf{161} (2012),109-118.

\bibitem{ProWilh2} C. Pro and F.\emph{\ }Wilhelm, \emph{\ Riemannian
submersions need not preserve positive Ricci curvature}. Proc. Amer. Math.
Soc., \textbf{142} (2014), 2529--2535.

\bibitem{Sas} S. Sasaki, \emph{On the differential geometry of tangent
bundles of Riemannian manifolds} Tohoku Math. J. (2) Volume 10, Number 3
(1958), 338--354.

\bibitem{SchTu1} L. Schwachh\"{o}fer, W. Tuschmann, \emph{Metrics of
positive Ricci curvature on quotient spaces}, Math. Ann., \textbf{330}, no.
1 (2004), 59--91.

\bibitem{SchTu2} L. Schwachh\"{o}fer, W. Tuschmann, \emph{Almost nonnegative
curvature and cohomogeneity one, }Preprint no. 62, Max-Planck-Institut f\"{u}%
r Mathematik in den Naturwissenschaften Leipzig,
http://www.mis.mpg.de/cgi-bin/preprints.pl

\bibitem{Shim} N. Shimada, \emph{Differentiable structures on the 15-sphere
and Pontriagin classes of certain manifolds,} Nagoya Math. J., \textbf{12}
(1957), 37--69.

\bibitem{Stu1} K.-T. Sturm, \emph{On the geometry of metric measure spaces.}
I. Acta Math. \textbf{196}, no. 1 (2006), 65--131.

\bibitem{Stu2} K.-T. Sturm, \emph{On the geometry of metric measure spaces.}
II. Acta Math. \textbf{196}, no. 1 (2006), 133--177.

\bibitem{TngZng} Z. Tang and W. Zhang, \emph{On a problem of B\'{e}%
rard-Bergery and Besse, }preprint. http://arxiv.org/abs/1302.2792

\bibitem{Tata} B. Totaro, \emph{Cheeger manifolds and the classification of
biquotients, }J. Diff. Geom. \textbf{61} (2002), 397--451

\bibitem{Wals} G. Walschap, \emph{Metric Structures in Differential
Geometry, }Graduate Texts in Mathematics, Springer-Verlag, NY, 2004.

\bibitem{Wei} G. Wei, \emph{Aspects of positively Ricci curved spaces: New
examples and the fundamental group, }Ph.D. thesis, Stony Brook, 1989.
http://www.math.ucsb.edu/\symbol{126}wei/paper/Wei-thesis.pdf

\bibitem{Wilh} F. Wilhelm, \emph{Exotic spheres with lots of positive
curvatures} J. of Geom. Anal., 11 (2001) 161--186.

\bibitem{Wilk} B. Wilking, \emph{A duality theorem for Riemannian foliations
in nonnegative sectional curvature} Geom. Funct. Anal. \textbf{17}, no. 4
(2007), 1297--1320.

\bibitem{Wr1} D. Wraith, \emph{Exotic spheres with positive Ricci curvature,}
J. Diff. Geom. \textbf{45} (1997), 638--649.

\bibitem{Wr2} D. Wraith, \emph{Surgery on Ricci positive manifolds, }J.
Reine Angew. Math. \textbf{501} (1998), 99--113.

\bibitem{Wr3} D. Wraith \emph{$G$-manifolds with positive Ricci curvature
and many isolated singular orbits}, Ann. Glob. Anal. Geom. \textbf{45}, no.
4 (2014), 319--335.

\bibitem{Yam} T. Yamaguchi, \emph{Collapsing and pinching under a lower
curvature bound}, Ann. of Math., 133 (1991), 317--357

\bibitem{ZhZh} H. Zhang, X. Zhu, \emph{On a new definition of Ricci
curvature on Alexandrov spaces.} (English summary) Acta Math. Sci. Ser. B
Engl. Ed. 30 (2010), no. 6, 1949--1974.
\end{thebibliography}
\end{document}